\documentclass[12pt,oneside,british, english]{amsart}
\usepackage[LGR,T1]{fontenc}
\usepackage[latin9]{inputenc}
\usepackage{geometry}
\usepackage{verbatim}
\usepackage{float}
\usepackage{units}
\usepackage{mathrsfs}
\usepackage{mathtools}
\usepackage{enumitem}
\usepackage{amsbsy}
\usepackage{amstext}
\usepackage{amsthm}
\usepackage{amssymb}
\usepackage[all]{xy}
\PassOptionsToPackage{normalem}{ulem}
\usepackage{ulem}

\makeatletter

\newcommand{\noun}[1]{\textsc{#1}}
\DeclareRobustCommand{\greektext}{%
  \fontencoding{LGR}\selectfont\def\encodingdefault{LGR}}
\DeclareRobustCommand{\textgreek}[1]{\leavevmode{\greektext #1}}
\ProvideTextCommand{\~}{LGR}[1]{\char126#1}

\newlength{\lyxlabelwidth}      
\theoremstyle{plain}
\newtheorem{thm}{\protect\theoremname}[section]
\theoremstyle{remark}
\newtheorem{rem}[thm]{\protect\remarkname}
\theoremstyle{definition}
\newtheorem{defn}[thm]{\protect\definitionname}
\theoremstyle{plain}
\newtheorem{prop}[thm]{\protect\propositionname}
\theoremstyle{plain}
\newtheorem{lem}[thm]{\protect\lemmaname}
\theoremstyle{definition}
\newtheorem{example}[thm]{\protect\examplename}
\theoremstyle{plain}
\newtheorem{cor}[thm]{\protect\corollaryname}
	\newenvironment{elabeling}[2][]%
	{\settowidth{\lyxlabelwidth}{#2}
		\begin{description}[font=\normalfont,style=sameline,
			leftmargin=\lyxlabelwidth,#1]}
	{\end{description}}

\usepackage{tikz}
\usepackage{tikz-cd}
\usepackage{xargs}
\usetikzlibrary{arrows.meta}
\usetikzlibrary{decorations.markings}

\usepackage{quiver}

\usepackage{adjustbox}

\usepackage{inputenc}

\newcommand{\Yo}{\text{\usefont{U}{min}{m}{n}\symbol{'110}}}

\DeclareFontFamily{U}{min}{}
\DeclareFontShape{U}{min}{m}{n}{<-> dmjhira}{}

\newlength{\myline}           
\newcommandx*{\triplearrow}[4][1=0, 2=1]{
  	\draw[line width=\myline,double distance=3\myline,#3] #4;
  	\draw[line width=\myline,shorten <=#1\myline,shorten >=#2\myline,#3] #4;
}

\newcommand{\xyR}[1]{
  \xydef@\xymatrixrowsep@{#1}}
\newcommand{\xyC}[1]{
  \xydef@\xymatrixcolsep@{#1}}
\newcommand{\pullbackcorner}[1][dr]{
  \save*!/#1-1.2pc/#1:(1,-1)@^{|-}\restore}
\newcommand{\pushoutcorner}[1][dr]{
  \save*!/#1+1.2pc/#1:(-1,1)@^{|-}\restore}

\usepackage{tikzit}

\usepackage{scalerel,stackengine}
\stackMath
\newcommand\reallywidehat[1]{%
\savestack{\tmpbox}{\stretchto{%
  \scaleto{%
    \scalerel*[\widthof{\ensuremath{#1}}]{\kern.1pt\mathchar"0362\kern.1pt}%
    {\rule{0ex}{\textheight}}
  }{\textheight}%
}{2.4ex}}%
\stackon[-6.9pt]{#1}{\tmpbox}%
}

\mathchardef\mhyphen="2D

\makeatother

\usepackage{babel}
\addto\captionsbritish{\renewcommand{\corollaryname}{Corollary}}
\addto\captionsbritish{\renewcommand{\definitionname}{Definition}}
\addto\captionsbritish{\renewcommand{\examplename}{Example}}
\addto\captionsbritish{\renewcommand{\lemmaname}{Lemma}}
\addto\captionsbritish{\renewcommand{\propositionname}{Proposition}}
\addto\captionsbritish{\renewcommand{\remarkname}{Remark}}
\addto\captionsbritish{\renewcommand{\theoremname}{Theorem}}
\providecommand{\corollaryname}{Corollary}
\providecommand{\definitionname}{Definition}
\providecommand{\examplename}{Example}
\providecommand{\lemmaname}{Lemma}
\providecommand{\propositionname}{Proposition}
\providecommand{\remarkname}{Remark}
\providecommand{\theoremname}{Theorem}

\begin{document}
\begin{quotation}

\global\long\def\bA{\mathbf{A}}%

\global\long\def\sA{\mathscr{A}}%

\global\long\def\sB{\mathscr{B}}%

\global\long\def\C{\mathbf{C}}%

\global\long\def\bC{\mathbb{C}}%

\global\long\def\sC{\mathscr{C}}%

\global\long\def\sfC{\mathsf{C}}%

\global\long\def\sD{\mathscr{D}}%

\global\long\def\sfD{\mathsf{D}}%

\global\long\def\sE{\mathscr{E}}%

\global\long\def\bF{\mathbb{F}}%

\global\long\def\cF{\mathcal{F}}%

\global\long\def\sF{\mathscr{F}}%

\global\long\def\bG{\mathbf{G}}%

\global\long\def\bbG{\mathbb{G}}%

\global\long\def\cG{\mathcal{G}}%

\global\long\def\sG{\mathscr{G}}%

\global\long\def\sH{\mathscr{H}}%

\global\long\def\cI{\mathcal{I}}%

\global\long\def\sI{\mathcal{\mathscr{I}}}%

\global\long\def\fI{\mathfrak{I}}%

\global\long\def\bJ{\mathbf{J}}%

\global\long\def\cL{\mathcal{L}}%

\global\long\def\sL{\mathscr{L}}%

\global\long\def\fm{\mathfrak{m}}%

\global\long\def\sM{\mathscr{M}}%

\global\long\def\N{\mathbf{N}}%

\global\long\def\fN{\mathfrak{N}}%

\global\long\def\sN{\mathscr{N}}%

\global\long\def\cO{\mathcal{O}}%

\global\long\def\sO{\mathscr{O}}%

\global\long\def\bP{\mathbf{P}}%

\global\long\def\fp{\mathfrak{p}}%

\global\long\def\bQ{\mathbb{Q}}%

\global\long\def\fq{\mathfrak{q}}%

\global\long\def\sR{\mathscr{R}}%

\global\long\def\R{\mathbb{\mathbf{R}}}%

\global\long\def\RR{\overline{\sR}}%

\global\long\def\cS{\mathscr{\mathcal{S}}}%

\global\long\def\fS{\mathfrak{S}}%

\global\long\def\sS{\mathcal{\mathscr{S}}}%

\global\long\def\bU{\mathbf{U}}%

\global\long\def\sU{\mathfrak{\mathscr{U}}}%

\global\long\def\bV{\mathbf{V}}%

\global\long\def\sV{\mathscr{V}}%

\global\long\def\sfV{\mathsf{V}}%

\global\long\def\sW{\mathscr{W}}%

\global\long\def\sX{\mathcal{\mathscr{X}}}%

\global\long\def\bZ{\mathbb{\mathbf{Z}}}%

\global\long\def\Z{\mathbf{Z}}%

\global\long\def\frZ{\mathfrak{Z}}%


\global\long\def\e{\text{\ensuremath{\varepsilon}}}%

\global\long\def\G{\Gamma}%


\global\long\def\p{\mathcal{\prime}}%

\global\long\def\srn{\sqrt{-5}}%


\global\long\def\wh#1{\widehat{#1}}%


\global\long\def\nf#1#2{\nicefrac{#1}{#2}}%


\global\long\def\E{\mathsf{Ens}}%

\global\long\def\Fin{\mathsf{Fin}}%

\global\long\def\Ord{\mathsf{Ord}}%

\global\long\def\S{\mathsf{Set}}%

\global\long\def\sd#1{\left.#1\right|}%


\global\long\def\Aut{\text{\ensuremath{\mathsf{Aut}}}}%

\global\long\def\Arr#1{\mathsf{Arr}\left(#1\right)}%

\global\long\def\Cat{\mathsf{Cat}}%
\global\long\def\CAT{\mathsf{CAT}}%

\global\long\def\coker{\mathrm{coker}}%

\global\long\def\colim{\underrightarrow{\mathsf{lim}}}%

\global\long\def\lcolim#1{\underset{#1}{\underrightarrow{\mathsf{lim}}}}%

\global\long\def\llim{\underleftarrow{\mathsf{lim}}}%

\global\long\def\lllim#1{\underset{#1}{\underleftarrow{\mathsf{lim}}}}%

\global\long\def\laxlim{\underleftarrow{\mathsf{laxlim}}}%
\global\long\def\oplaxlim{\underleftarrow{\mathsf{oplaxlim}}}%

\global\long\def\laxcolim{\underrightarrow{\mathsf{laxlim}}}%
\global\long\def\oplaxcolim{\underrightarrow{\mathsf{oplaxlim}}}%

\global\long\def\End{\mathsf{End}}%

\global\long\def\Ext{{\rm Ext}}%

\global\long\def\id{\text{id}}%

\global\long\def\im{\text{im}}%

\global\long\def\Im{\text{Im}}%

\global\long\def\iso{\overset{\sim}{\rightarrow}}%
\global\long\def\osi{\overset{\sim}{\leftarrow}}%

\global\long\def\liso{\overset{\sim}{\longrightarrow}}%
\global\long\def\losi{\overset{\sim}{\longleftarrow}}%

\global\long\def\Hom{\mathsf{Hom}}%

\global\long\def\BHom{\mathbf{Hom}}%

\global\long\def\Homc{\text{\text{Hom}}_{\mathscr{C}}}%

\global\long\def\Homd{\text{\text{Hom}}_{\mathscr{D}}}%

\global\long\def\Homs{\text{\text{Hom}}_{\mathsf{Set}}}%

\global\long\def\Homt{\text{\text{Hom}}_{\mathsf{Top}}}%

\global\long\def\Mono{\mathsf{Mono}}%

\global\long\def\Mor{\text{\ensuremath{\mathsf{Mor}}}}%

\global\long\def\Ob{\mathsf{Ob}}%

\global\long\def\one{\mathbf{1}}%

\global\long\def\op{\mathsf{op}}%

\global\long\def\Pull{\mathsf{Pull}}%

\global\long\def\Push{\mathsf{Push}}%

\global\long\def\Psh#1{\mathsf{Psh}\left(#1\right)}%

\global\long\def\pt{\text{pt.}}%

\global\long\def\Sh#1{\mathsf{Sh}\left(#1\right)}%

\global\long\def\sh{\text{sh}}%

\global\long\def\Tor{\text{\text{Tor}}}%

\global\long\def\XoY{\nicefrac{X}{Y}}%

\global\long\def\XtY{X\times Y}%

\global\long\def\Yon{\Yo}%


\global\long\def\Ab{\mathsf{Ab}}%

\global\long\def\AbGrp{\text{\ensuremath{\mathsf{AbGrp}}}}%

\global\long\def\alg{\mathsf{alg}}%

\global\long\def\ann{\text{ann}}%

\global\long\def\Ass{{\rm Ass}}%

\global\long\def\Ch{\mathsf{Ch}}%

\global\long\def\CR{\mathsf{ComRing}}%

\global\long\def\Gal{\text{Gal}}%

\global\long\def\Grp{\mathsf{Grp}}%

\global\long\def\Inn{\mathsf{Inn}}%

\global\long\def\Frac{\text{Frac}}%

\global\long\def\mod{\text{\ensuremath{\mathsf{mod}}}}%

\global\long\def\norm{\text{Norm}}%

\global\long\def\Norm{\mathsf{Norm}}%

\global\long\def\Orb{\mathsf{Orb}}%

\global\long\def\rad{\text{\text{rad}}}%

\global\long\def\Ring{\mathbb{\mathsf{Ring}}}%

\global\long\def\sgn{\text{sgn}}%

\global\long\def\Stab{\mathsf{Stab}}%

\global\long\def\Syl{\text{Syl}}%

\global\long\def\ZpZ{\nicefrac{\mathbb{\mathbf{Z}}}{p\bZ}}%


\global\long\def\CompR{_{R}\mathsf{Comp}}%


\global\long\def\clm{\RR=\left(\sR,\mathbf{U},\cL,\bA\right)}%

\global\long\def\ebar{\overline{\sE}}%

\global\long\def\lrs{\mathsf{LocRngSpc}}%

\global\long\def\LRS{\mathsf{LRS}}%

\global\long\def\sHom{\mathscr{H}om}%

\global\long\def\RRing{\overline{\sR}\mathsf{-Ring}}%


\global\long\def\Ae{\nicefrac{A\left[\varepsilon\right]}{\left(\e^{2}\right)}}%

\global\long\def\Aff{\mathsf{Aff}}%

\global\long\def\AX{\nicefrac{\Aff}{X}}%

\global\long\def\AY{\nicefrac{\Aff}{Y}}%

\global\long\def\Der{\text{Der}}%

\global\long\def\et{\text{\ensuremath{\acute{e}}t}}%

\global\long\def\etale{\acute{\text{e}}\text{tale}}%

\global\long\def\Et{\mathsf{\acute{E}t}}%

\global\long\def\ke{\nf{k\left[\varepsilon\right]}{\varepsilon^{2}}}%

\global\long\def\Pic{\text{Pic}}%

\global\long\def\proj{\text{Proj}}%

\global\long\def\Qcoh#1{\mathsf{QCoh}\left(#1\right)}%

\global\long\def\rad{\text{\text{rad}}}%

\global\long\def\red{\text{red.}}%

\global\long\def\aScheme#1{\left(\spec\left(#1\right),\cO_{\spec\left(#1\right)}\right)}%

\global\long\def\resScheme#1#2{\left(#2,\sd{\cO_{#1}}_{#2}\right)}%

\global\long\def\Sch{\mathbb{\mathsf{Sch}}}%

\global\long\def\scheme#1{\left(#1,\cO_{#1}\right)}%

\global\long\def\SCR{\S^{\CR}}%

\global\long\def\Schs{\mathsf{\nicefrac{\Sch}{S}}}%

\global\long\def\ShAb#1{\mathsf{Sh}_{\AbGrp}\left(#1\right)}%

\global\long\def\Spec{\text{Spec}}%

\global\long\def\Sym{\text{Sym}}%

\global\long\def\Zar{\mathsf{Zar}}%


\global\long\def\co#1#2{\left[#1,#2\right)}%

\global\long\def\oc#1#2{\left(#1,#2\right]}%

\global\long\def\loc{\mathsf{\text{loc.}}}%

\global\long\def\nn#1{\left\Vert #1\right\Vert }%

\global\long\def\Re{\text{Re}}%

\global\long\def\supp{\text{supp}}%

\global\long\def\XSM{\left(X,\cS,\mu\right)}%


\global\long\def\grad{\text{grad}}%

\global\long\def\Homrv{\text{\text{Hom}}_{\mathbb{R}-\mathsf{Vect}}}%



\global\long\def\Cech{{\rm \check{C}ech}}%
\global\long\def\cCC{\check{\mathcal{C}}}%
\global\long\def\CC{\check{C}}%

\global\long\def\CH{\mathsf{CHaus}}%

\global\long\def\Cov{\mathsf{Cov}}%

\global\long\def\CW{\mathsf{CW}}%

\global\long\def\HT{\mathsf{HTop}}%

\global\long\def\Homt{\text{\text{Hom}}_{\mathsf{Top}}}%

\global\long\def\Homrv{\text{\text{Hom}}_{\mathbb{R}-\mathsf{Vect}}}%

\global\long\def\MT{\text{\ensuremath{\mathsf{Mor}}}_{\T}}%

\global\long\def\Open{\text{\ensuremath{\mathsf{Open}}}}%

\global\long\def\PT{\mathsf{P-Top}}%

\global\long\def\T{\mathsf{Top}}%


\global\long\def\Ad{\mathsf{Ad}}%

\global\long\def\Cell{\mathsf{Cell}}%

\global\long\def\Cov{\mathsf{Cov}}%

\global\long\def\Sp{\mathsf{Sp}}%

\global\long\def\Spectra{\mathsf{Spectra}}%

\global\long\def\ss{\widehat{\triangle}}%

\global\long\def\Tn{\mathbb{T}^{n}}%

\global\long\def\Sk#1{\textrm{Sk}^{#1}}%

\global\long\def\smash{\wedge}%

\global\long\def\wp{\vee}%


\global\long\def\base{\mathsf{base}}%

\global\long\def\comp{\mathsf{comp}}%

\global\long\def\funext{\mathsf{funext}}%

\global\long\def\hfib{\text{\ensuremath{\mathsf{hfib}}}}%

\global\long\def\I{\mathbf{I}}%

\global\long\def\ind{\mathsf{ind}}%


\global\long\def\lp{\mathsf{loop}}%

\global\long\def\pair{\mathsf{pair}}%

\global\long\def\pr{\mathbf{\mathsf{pr}}}%

\global\long\def\rec{\mathsf{rec}}%

\global\long\def\refl{\mathsf{refl}}%

\global\long\def\transport{\mathsf{transport}}%


\global\long\def\is{\triangle\raisebox{2mm}{\mbox{\ensuremath{\infty}}}}%

\global\long\def\Cof{\mathsf{Cof}}%

\global\long\def\sfW{\mathsf{W}}%

\global\long\def\Cyl{\mathsf{Cyl}}%

\global\long\def\Mono{\mathsf{Mono}}%

\global\long\def\t{\triangle}%

\global\long\def\tl{\triangleleft}%

\global\long\def\tr{\triangleright}%

\global\long\def\Shift{\mathrm{Shift}_{+1}}%

\global\long\def\Shiftd{\mathrm{Shift}_{-1}}%

\global\long\def\out{\mathrm{out}}%

\global\long\def\cN{\mathcal{N}}%

\global\long\def\fC{\mathcal{\mathfrak{C}}}%

\global\long\def\ev{\mathsf{ev}}%

\global\long\def\Map{\mathsf{Map}}%

\global\long\def\whp#1{\wh{#1}_{\bullet}}%

\global\long\def\bfTwo{\mathbf{2}}%

\global\long\def\bfL{\mathbf{L}}%

\global\long\def\bfR{\mathbf{R}}%

\global\long\def\sJ{\mathscr{J}}%

\global\long\def\Sing{\mathsf{Sing}}%

\global\long\def\Sph{\mathsf{Sph}}%

\global\long\def\whfin#1{\widehat{#1}_{\mathrm{fin}}}%

\global\long\def\whpfin#1{\widehat{#1}_{\mathrm{\bullet fin}}}%

\global\long\def\fin{\mathsf{fin}}%

\global\long\def\cT{\mathcal{T}}%

\global\long\def\Alg{\mathsf{Alg}}%

\global\long\def\st{\mathsf{st}}%

\global\long\def\IN{\mathsf{in}}%

\global\long\def\hr{\mathsf{hr}}%

\global\long\def\Fun{\mathsf{Fun}}%

\global\long\def\Th{\mathsf{Th}}%

\global\long\def\sT{\mathscr{T}}%

\global\long\def\Lex{\mathsf{Lex}}%

\global\long\def\FinSet{\mathsf{FinSet}}%

\global\long\def\Fib{\mathsf{Fib}}%

\global\long\def\FPS{\mathsf{FinPos}}%

\global\long\def\Mod{\mathsf{Mod}}%

\global\long\def\sfA{\mathsf{A}}%

\global\long\def\sfV{\mathsf{V}}%

\global\long\def\inner{\mathsf{inner}}%
\global\long\def\bfI{\mathbf{I}}%
\global\long\def\Kan{\mathsf{Kan}}%
\global\long\def\Berger{\mathsf{Berger}}%

\global\long\def\PR{\mathsf{P.R.}}%

\global\long\def\LFP{\mathsf{LFP}}%

\global\long\def\Mnd{\mathsf{Mnd}}%

\global\long\def\sfS{\mathsf{S}}%

\global\long\def\Adj{\mathsf{Adj}}%

\global\long\def\RAdj{\mathsf{RAdj}}%

\global\long\def\LAdj{\mathsf{LAdj}}%

\global\long\def\conj{\mathsf{conj}}%

\global\long\def\exact{\mathsf{exact}}%

\global\long\def\CwA{\mathsf{CwA}}%

\global\long\def\sf#1{\mathsf{#1}}%

\global\long\def\fA{\sf A}%

\global\long\def\fB{\sf B}%

\global\long\def\fE{\sf E}%

\global\long\def\fF{\sf F}%

\global\long\def\fG{\sf G}%

\global\long\def\fD{\sf D}%

\global\long\def\fJ{\sf I}%

\global\long\def\fJ{\sf J}%
\global\long\def\sfJ{\mathsf{J}}%

\global\long\def\fX{\mathfrak{\sf X}}%

\global\long\def\fY{\sf Y}%

\global\long\def\fZ{\sf Z}%

\global\long\def\bf#1{\mathsf{#1}}%

\global\long\def\cc{\mathsf{cc}}%
\global\long\def\Arity{\sf{Arity}}%
 
\global\long\def\conj{\sf{conj}}%
\global\long\def\Ins{\sf{Ins}}%
\global\long\def\bfOne{\mathbf{1}}%
\global\long\def\oplaxlim{\underleftarrow{\sf{oplaxlim}}}%
\global\long\def\ff{\mathsf{ff}}%
\global\long\def\cc{\mathsf{cc}}%
\global\long\def\Prof{\mathsf{Prof}}%
\global\long\def\sfA{\mathsf{A}}%
\global\long\def\sfB{\mathsf{B}}%
\global\long\def\sfC{\mathsf{C}}%
\global\long\def\sfD{\mathsf{D}}%
\global\long\def\sfE{\mathsf{E}}%
\global\long\def\sfF{\mathsf{F}}%
\global\long\def\sfG{\mathsf{G}}%
\global\long\def\ls{\mathsf{ls}}%
\global\long\def\CAT{\mathsf{CAT}}%
\global\long\def\Equ{\mathsf{Equ}}%
\global\long\def\Par{\mathsf{Par}}%
\global\long\def\Ran{\mathsf{Ran}}%
\global\long\def\Lan{\mathsf{Lan}}%
 
\global\long\def\StrCat{\mathsf{Str}\mhyphen\omega\mhyphen\Cat}%
\end{quotation}
\title{\noun{$\Z$-Categories I}}
\author{Paul Lessard}
\begin{abstract}
This paper is the first in a series of two papers, $\Z$-Categories
I and $\Z$-Categories II, which develop the notion of $\Z$-category,
the natural bi-infinite analog to strict $\omega$-categories, and
show that the $\left(\infty,1\right)$-category of spectra relates
to the $\left(\infty,1\right)$-category of homotopy coherent $\Z$-categories
as the pointed groupoids.

In this work we provide a $2$-categorical treatment of the combinatorial
spectra of \cite{Kan} and argue that this description is a simplicial
avatar of the abiding notion of homotopy coherent $\Z$-category.
We then develop the theory of limits in the $2$-category of categories
with arities of \cite{BergerMelliesWeber} to provide a cellular category
which is to $\Z$-categories as $\t$ is to $1$-categories or $\Theta_{n}$
is to $n$-categories. In an appendix we provide a generalization
of the spectrification functors of 20$^{\mathrm{th}}$ century stable
homotopy theory in the language of category-weighted limits.
\end{abstract}

\maketitle

\part*{Preface}

This paper is the first in a series of two papers, $\Z$-Categories
I and $\Z$-Categories II, which develop the notion of $\Z$-category,
the natural bi-infinite analog to strict $\omega$-categories, and
show that the $\left(\infty,1\right)$-category of spectra lives inside
the $\left(\infty,1\right)$-category of pointed homotopy coherent
$\Z$-categories as the full sub-$\left(\infty,1\right)$-category
of $\Z$-groupoids. That theorem is carried out in the language of
Cisinski model categories.\footnote{See \cite{Cisinski} or $\Z$-categories II Part II for a survey of
the relevant results translated from the French.} The work is split as it is for the first is a work of (nearly formal)
$2$-category theory, and the second is almost entirely restricted
to homotopical concerns.

In a sense these two works aims to synthesize the 1963 combinatorial
model of spectra of \cite{Kan}, with the 1983 homotopy hypothesis
of \cite{Grothendieck}, by way of Australian style $2$-category
theory and the $2$-category of categories arities of \cite{BergerMelliesWeber},
in the context of the 2006 theory of Cisinski model categories of
\cite{Cisinski} interpreted as presentations of categories of homotopy
coherent models of limit theories.

The work is aimed to be readable both by topologists and category
theorists and, as such, many elementary notions which are well known
to one group, but much less well so to the other, are treated somewhat
explicitly. We beg the reader's patience.

\part*{Introduction}

We begin with a long(er) and semi-technical introduction. This introduction
should be read with the foreknowledge that all will be explained in
more detail in the body of the text.

\section*{Part I: A $2$-categorical treatment of Kan's ``semisimplicial spectra''}

\subsection*{Kan's Suspension}

The work of \cite{Kan} begins with a description of the reduced suspension
which is adapted to the geometry/combinatorics of the simplex category.
Kan describes, albeit not in these modern terms, the reduced suspension
of pointed simplicial sets as the left Kan extension
\[
\xyR{0pc}\xyC{5pc}\Sigma_{K}=\mathsf{Lan}_{\Yo_{+}}\left(\vcenter{\vbox{\xymatrix{\t\ar[r] & \whp{\t}\\
\left[n\right]\ar@{|->}[r] & \nf{\t\left[n+1\right]_{+}}{\t\left[n\right]_{+}\vee\t\left[0\right]_{+}}
}
}}\right)
\]
which assigns to every freely pointed $n$-simplex a ``fattened''
circle whose cross section shares the geometry of the $n$-simplex.
We will note that this functor is a quotient of (a post-composition)
of the endofunctor 
\[
\xyR{0pc}\xyC{5pc}\xymatrix{K:\t\ar[r] & \t\\
\left[n\right]\ar@{|->}[r] & \left[n+1\right]
}
\]
The implied factoring will become important when we understand categories
of sequential spectrum objects and $\Omega$-spectrum objects as particular
$2$-categorical limits.

\subsection*{Sequential spectra and $\Omega$-spectra as oplax-limits and pseudo-limits
respectively}

Now, in a $1$-category there is a unique notion of the cone over
a diagram. However, in passing to $2$-categories, we immediately
come upon three further notions of cone over a $1$-diagram. Indeed,
in a $2$-category, we have in the strict cones an avatar of the $1$-categorical
notion of cone. But we also have pseudo-cones in which the role played
by identity $2$-cells in commuting triangles is instead played by
a coherent choice of invertible $2$-cells. More, we have and oplax
and lax cones wherein those $2$-cells remain coherent, i.e. functorial
in the 1-cells of the diagram, but for which the hypothesis of invertibility
is dropped (see Figures \ref{fig:strict-cone},\ref{fig:pseudo-cone},\ref{fig:oplax-cone},
and \ref{fig:lax-cone}). The oplax and lax conventions are then but
the two choices for the directions of the comparison $2$-cells. Importantly,
just as the data of a cone in a $1$-category constitutes a natural
transformation from a constant functor into the diagram, in a $2$-category,
the various notions of cones, pseudo-cones, oplax-cones, and lax-cones
constitute pseudo-natural transformations, oplax-natural transformations,
and lax-natural transformations\footnote{Importantly, these four distinct notions of natural transformation
define four distinct enrichements \emph{of} a category of categories
\emph{in} a category of categories.} from a constant functor into the diagram.

\selectlanguage{british}%
\begin{figure}
\begin{minipage}[t]{0.5\textwidth}%
\smallskip{}
\noindent\begin{minipage}[t]{1\columnwidth}%
\adjustbox{scale=0.5, center}{
$$
	\begin{tikzcd}
		&&& \bullet \\
		\\
		\\
		\\
		&&&&&& {X_{k}} \\
		\\
		{X_{i}} \\
		\\
		&&&& {X_{j}}
		\arrow["{X_{g}}"{description}, from=7-1, to=9-5]
		\arrow["{X_{fg}}"{description}, from=7-1, to=5-7]
		\arrow["{X_f}"{description}, from=9-5, to=5-7]
		\arrow[""{name=1, anchor=center, inner sep=0}, "{\varphi_{k}}"{description}, from=1-4, to=5-7]
		\arrow["{\mathrm{id}}"{description}, shorten <=15pt, shorten >=15pt, equal, from=1, to=7-1]
		\arrow["{\mathrm{id}}"{description}, shorten <=12pt, shorten >=12pt, equal, from=1, to=9-5, crossing over]
		\arrow["{\varphi_{i}}"{description}, from=1-4, to=7-1]
		\arrow[""{name=0, anchor=center, inner sep=0}, "{\varphi_{j}}"{description}, from=1-4, to=9-5, crossing over]
		\arrow["{\mathrm{id}}"{description}, shorten <=10pt, shorten >=10pt, equal, from=0, to=7-1, crossing over]
	\end{tikzcd}
$$
}%
\end{minipage}\caption{\label{fig:strict-cone}strict cone}
\end{minipage}%
\begin{minipage}[t][1\totalheight]{0.5\textwidth}%
\smallskip{}
\noindent\begin{minipage}[t]{1\columnwidth}%
\adjustbox{scale=0.5, center}{
$$
	\begin{tikzcd}
		&&& \bullet \\
		\\
		\\
		\\
		&&&&&& {X_{k}} \\
		\\
		{X_{i}} \\
		\\
		&&&& {X_{j}}
		\arrow["{X_{g}}"{description}, from=7-1, to=9-5]
		\arrow["{X_{fg}}"{description}, from=7-1, to=5-7]
		\arrow["{X_f}"{description}, from=9-5, to=5-7]
		\arrow[""{name=1, anchor=center, inner sep=0}, "{\varphi_{k}}"{description}, from=1-4, to=5-7]
		\arrow["{\widetilde{\alpha_{fg}}}"{description}, shorten <=15pt, shorten >=15pt, Rightarrow, from=1, to=7-1]
		\arrow["{\widetilde{\alpha_{f}}}"{description}, shorten <=12pt, shorten >=12pt, Rightarrow, from=1, to=9-5, crossing over]
		\arrow["{\varphi_{i}}"{description}, from=1-4, to=7-1]
		\arrow[""{name=0, anchor=center, inner sep=0}, "{\varphi_{j}}"{description}, from=1-4, to=9-5, crossing over]
		\arrow["{\widetilde{\alpha_{g}}}"{description}, shorten <=10pt, shorten >=10pt, Rightarrow, from=0, to=7-1, crossing over]
	\end{tikzcd}
$$
}%
\end{minipage}\caption{\label{fig:pseudo-cone}pseudo-cone}
\end{minipage}

\begin{minipage}[t][1\totalheight]{0.5\textwidth}%
\smallskip{}
\noindent\begin{minipage}[t]{1\columnwidth}%
\adjustbox{scale=0.5, center}{
$$
	\begin{tikzcd}
		&&& \bullet \\
		\\
		\\
		\\
		&&&&&& {X_{k}} \\
		\\
		{X_{i}} \\
		\\
		&&&& {X_{j}}
		\arrow["{X_{g}}"{description}, from=7-1, to=9-5]  
		\arrow["{X_{fg}}"{description}, from=7-1, to=5-7] 
		\arrow["{X_f}"{description}, from=9-5, to=5-7]    
		\arrow[""{name=1, anchor=center, inner sep=0}, "{\varphi_{k}}"{description}, from=1-4, to=5-7]
		\arrow["{\alpha_{fg}}"{description}, shorten <=15pt, shorten >=15pt, Rightarrow, from=1, to=7-1]
		\arrow["{\alpha_{f}}"{description}, shorten <=12pt, shorten >=12pt, Rightarrow, from=1, to=9-5, crossing over] 
		\arrow["{\varphi_{i}}"{description}, from=1-4, to=7-1]
		\arrow[""{name=0, anchor=center, inner sep=0}, "{\varphi_{j}}"{description}, from=1-4, to=9-5, crossing over]
		\arrow["{\alpha_{g}}"{description}, shorten <=10pt, shorten >=10pt, Rightarrow, from=0, to=7-1, crossing over]
	\end{tikzcd}
$$
}%
\end{minipage}\caption{\label{fig:oplax-cone}oplax cone}
\end{minipage}%
\begin{minipage}[t][1\totalheight]{0.5\textwidth}%
\smallskip{}
\noindent\begin{minipage}[t]{1\columnwidth}%
\adjustbox{scale=0.5, center}{
$$
	\begin{tikzcd}
		&&& \bullet \\
		\\
		\\
		\\
		&&&&&& {X_{k}} \\
		\\
		{X_{i}} \\
		\\
		&&&& {X_{j}}
		\arrow["{X_{g}}"{description}, from=7-1, to=9-5]  
		\arrow["{X_{fg}}"{description}, from=7-1, to=5-7] 
		\arrow["{X_f}"{description}, from=9-5, to=5-7]    
		\arrow[""{name=1, anchor=center, inner sep=0}, "{\varphi_{k}}"{description}, from=1-4, to=5-7]
		\arrow["{\alpha_{fg}}"{description}, shorten <=15pt, shorten >=15pt, Rightarrow, to=1, from=7-1]
		\arrow["{\alpha_{f}}"{description}, shorten <=12pt, shorten >=12pt, Rightarrow, to=1, from=9-5, crossing over]
		\arrow["{\varphi_{i}}"{description}, from=1-4, to=7-1]
		\arrow[""{name=0, anchor=center, inner sep=0}, "{\varphi_{j}}"{description}, from=1-4, to=9-5, crossing over]
		\arrow["{\alpha_{g}}"{description}, shorten <=10pt, shorten >=10pt, Rightarrow, to=0, from=7-1, crossing over]
	\end{tikzcd}
$$
}%
\end{minipage}\caption{\label{fig:lax-cone}lax cone}
\end{minipage}
\end{figure}

\selectlanguage{english}%
Now, it is no big leap to go from Adams spectra to sequential spectrum
objects valued in a category $\sf A$ with respect to an adjunction
$L\dashv R$. Indeed a sequential spectrum object is comprised of
an $\N$-indexed set of objects of the category $\sfA$, $\left(X_{i}\right)_{i\in\N}$,
together with a set of maps $\left(X_{i}\xrightarrow{\varphi_{i}}R\left(X_{i+1}\right)\right)_{i\in\N}$.
As we'll see however, these are but objects of the oplax limit 
\[
\oplaxlim\left\{ \cdots\xrightarrow{R}\sfA\xrightarrow{R}\sfA\right\} 
\]
in the $2$-category $\text{\ensuremath{\Cat}}$ (or $\CAT$ etc.
depending on size) where the data of those structure maps $\left(\varphi_{i}\right)_{i\in\N}$
are the $2$-cells of the associated classified cone, and $\Omega$-spectrum
objects are but objects of the pseudo-limit over that same diagram.
Indeed, when the functor $R$ is an iso-fibration, the pseudo-limit
and conical limit coincide, whence $\Omega$-spectrum objects are
in fact but objects of the strict limit in all cases we concern ourselves
with here.

So called coordinate-free spectra (see \cite{ElmendorfKrizMandellMay})
too admit this description after replacing the diagram 
\[
\cdots\xrightarrow{R}\sf A\xrightarrow{R}\sfA
\]
with the diagram
\[
\xyR{0pc}\xyC{5pc}\xymatrix{\sf{FinSub}\left(\mathscr{U}\right)^{\op}\ar[r] & \Cat\\
V\ar@{|->}[r] & \sf S_{\bullet}\\
V\subset W\ar@{|->}[r] & \Omega^{W-V}
}
\]
where the category $\sf{FinSub}\left(\sU\right)$ is the category
of sub-vector spaces of a particular choice of $\omega$-dimensional
vector space $\sU\liso\R^{\omega}$.

Returning to Kan's particular choice of suspension functor, and its
factoring we will get a natural transformation between right adjoint
functors 
\[
K_{\bullet}^{*}\Longrightarrow\Omega_{K}:\whp{\t}\longrightarrow\whp{\t}
\]
where $\Omega_{K}$ is the functor right adjoint to Kan's $\Sigma_{K}$.
These $2$-cells will then assemble to define an oplax natural transformation
\[
\xyR{2.5pc}\xyC{2.5pc}\xymatrix{\cdots\ar[r] & \whp{\t}\ar[r]^{\Omega_{K}}\ar[d]_{\id}\ar@{=>}[dl] & \whp{\t}\ar[d]^{\id}\ar@{=>}[dl]\ar[r]^{\Omega_{K}} & \whp{\t}\ar[d]^{\id}\ar@{=>}[dl]\\
\cdots\ar[r] & \whp{\t}\ar[r]_{K_{\bullet}^{*}} & \whp{\t}\ar[r]_{K_{\bullet}^{*}} & \whp{\t}
}
\]
whence a functor between oplax limits
\[
\oplaxlim\left\{ \cdots\rightarrow\whp{\t}\xrightarrow{\Omega_{K}}\whp{\t}\xrightarrow{\Omega_{K}}\whp{\t}\right\} \longrightarrow\oplaxlim\left\{ \cdots\rightarrow\whp{\t}\xrightarrow{K_{\bullet}^{*}}\whp{\t}\xrightarrow{K_{\bullet}^{*}}\whp{\t}\right\} 
\]
But since $K_{\bullet}^{*}$ is induced by an endomorphism of $\t$
we will show how may rewrite the category on the right-hand side in
a more elementary fashion.

\subsection*{Factorization through $\left[\left(\_\right)_{\bullet}^{\protect\op},\protect\S_{\bullet}\right]_{\bullet}:\protect\Cat^{\protect\sf{coop}}\protect\longrightarrow\protect\CAT$
and the collage construction}

Indeed, we will observe that the diagram 
\[
\cdots\rightarrow\whp{\t}\xrightarrow{K_{\bullet}^{*}}\whp{\t}\xrightarrow{K_{\bullet}^{*}}\whp{\t}
\]
 factors through the $2$-functor
\[
\left[\left(\_\right)_{\bullet},\S_{\bullet}\right]_{\bullet}:\Cat^{\sf{coop}}\longrightarrow\CAT
\]
 which takes a small category $\sfA$ and returns the large category
of $\S_{\bullet}$-enriched functors $\wh{\sf A}_{\bullet}\liso\left[\sf A_{\bullet}^{\op},\S_{\bullet}\right]$
from $\sf A_{\bullet}^{\op}$ to $\S_{\bullet}$ where $\sf A_{\bullet}$
is the free $\S_{\bullet}$ enriched category on $\sf A$. That $2$-functor
however sends $2$-categorical colimits to $2$-categorical limits.

As such, we will find that we have equivalences
\[
\reallywidehat{\colim\left\{ \cdots\leftarrow\t\xleftarrow{K}\t\xleftarrow{K}\t\right\} }_{\bullet}\liso\llim\left\{ \cdots\rightarrow\whp{\t}\xrightarrow{K_{\bullet}^{*}}\whp{\t}\xrightarrow{K_{\bullet}^{*}}\whp{\t}\right\} 
\]
 and
\[
\reallywidehat{\underrightarrow{\sf{oplaxlim}}\left\{ \cdots\leftarrow\t\xleftarrow{K}\t\xleftarrow{K}\t\right\} }_{\bullet}\liso\oplaxlim\left\{ \cdots\rightarrow\whp{\t}\xrightarrow{K_{\bullet}^{*}}\whp{\t}\xrightarrow{K_{\bullet}^{*}}\whp{\t}\right\} 
\]
Then, invoking the collage construction, it will be easy to describe
the oplax colimit 
\[
\underrightarrow{\sf{oplaxlim}}\left\{ \cdots\leftarrow\t\xleftarrow{K}\t\xleftarrow{K}\t\right\} 
\]
 as the small category on the set of objects
\[
\left\{ \sd{\left(\left[n\right],-k\right)}n\in\N,-k\in\Z_{\leq0}\right\} =\coprod_{-m\in\Z_{\leq0}}\Ob\left(\t\right)
\]
with $\Hom$-sets defined by the expression
\[
\t_{\mathsf{coll}}\left(\left(\left[n\right],-m\right),\left(\left[\ell\right],-\left(m+k\right)\right)\right)=\t\left(\left[n+k\right],\left[\ell\right]\right)
\]
with the composition laws coming from the right hand side of that
equality. The universal property of the strict colimit $\colim\left\{ \cdots\leftarrow\t\xleftarrow{K}\t\xleftarrow{K}\t\right\} $
is similarly enjoyed by $\t_{\st}$, the category with set of objects
\[
\Ob\left(\t_{\st}\right)=\left\{ \sd{\left[z\right]}z\in\Z\right\} 
\]
with $\Hom$-sets generated by maps 
\[
\left\{ \sd{d^{i}:\left[m-1\right]\longrightarrow\left[m\right]}i\in\N\right\} 
\]
and 
\[
\left\{ \sd{s^{j}:\left[m+1\right]\longrightarrow\left[m\right]}j\in\text{\ensuremath{\N\ }}\right\} 
\]
subject to (un-bounded) simplicial identities. The obvious functor
$\t_{\sf{coll}}\longrightarrow\t_{\st}$ then induces an adjoint triple
\[
\xyR{0pc}\xyC{3pc}\xymatrix{\reallywidehat{\t_{\sf{coll}}}_{\bullet}\ar@/^{1pc}/[rr]\ar@/_{1pc}/[rr] &  & \reallywidehat{\t_{\sf{st}}}_{\bullet}\ar[ll]_{\bot}^{\bot}}
\]
with the pullback functor in the middle recovering the inclusion (note
the direction of the arrow)
\[
\oplaxlim\left\{ \cdots\rightarrow\whp{\t}\xrightarrow{K_{\bullet}^{*}}\whp{\t}\xrightarrow{K_{\bullet}^{*}}\whp{\t}\right\} \longleftarrow\llim\left\{ \cdots\rightarrow\whp{\t}\xrightarrow{K_{\bullet}^{*}}\whp{\t}\xrightarrow{K_{\bullet}^{*}}\whp{\t}\right\} 
\]
More, the oplax limit
\[
\oplaxlim\left\{ \cdots\rightarrow\whp{\t}\xrightarrow{\Omega_{K}}\whp{\t}\xrightarrow{\Omega_{K}}\whp{\t}\right\} \longrightarrow\oplaxlim\left\{ \cdots\rightarrow\whp{\t}\xrightarrow{K_{\bullet}^{*}}\whp{\t}\xrightarrow{K_{\bullet}^{*}}\whp{\t}\right\} 
\]
 too arises as a pullback along a pointed functor 
\[
\t_{\sf{coll}\bullet}\longrightarrow\oplaxlim\left\{ \cdots\rightarrow\whp{\t}\xrightarrow{\Omega_{K}}\whp{\t}\xrightarrow{\Omega_{K}}\whp{\t}\right\} 
\]
 which we will construct explicitly.

\subsection*{Double duty and slick tricks}

Kan defines the category of semisimplicial spectra to be the full
subcategory of $\wh{\t_{\st}}_{\bullet}$ subtended by colimits of
pointed presheaves of the form
\[
\t_{\st}\nf{\left[z\right]}{d^{>n}=*}
\]
which is a stable $\left[z\right]$-simplex with something much like
the geometry of an $n$-simplex, as its set of stable $\left[z-1\right],\left[z-2\right],\dots,\left[z-n\right]$
simplices which are \emph{not }the basepoint mirror the combinatorics
of the usual $\left[n\right]$-simplex. We call this subcategory $\sf{LocFin}\left(\t_{\st}\right)$.

We will see that part of the adjoint triple induced by the functor
above factors through $\sf{LocFin}\left(\t_{\st}\right)$ as a co-reflective
subcategory of $\whp{\t_{\st}}$.
\[
\xyR{0pc}\xyC{1.5pc}\xymatrix{\reallywidehat{\t_{\st}}_{\bullet}\ar[rr]^{\bot} &  & \ar@/_{1pc}/[ll]\mathsf{LocFin}\left(\t_{\st}\right)\ar@/_{1pc}/[rr] &  & \oplaxlim\left\{ \cdots\xrightarrow{\Omega_{K}}\whp{\t}\right\} \ar[ll]^{\bot}}
\]
 We will also correct a small error made in \cite{Kan} - a related
one in made \cite{ChenKrizPultr} - which has it that the right hand
adjunction above is a reflective subcategory - this is not the case
and we provide an obvious counter-example. Instead, if we define $\sf{LocSph\left(\t_{\st}\right)}$
to be the full subcategory of $\whp{\t_{\st}}$ on pointed presheaves
which are colimits of objects of the form
\[
S^{\left(z-n\right)}\left[n\right]=\nf{\t_{\st}\left[z\right]}{\left(d^{n}d^{n-1}\cdots d^{1}d^{0}=d^{>n}=*\right)}
\]
which are quotients of $\t_{\st}\nf{\left[z\right]}{d^{>n}=*}$ which
me may think as a $\left(z-n\right)$-sphere with an $\left[n\right]$-simplex
worth of nontrivial faces\footnote{thinking in this way a banana is an arrow with $\sim5$ faces},
then we then have a sequence of adjunctions
\[
\xyR{0pc}\xyC{1.5pc}\xymatrix{\reallywidehat{\t_{\st}}_{\bullet}\ar[rr]^{\bot} &  & \ar@/_{1pc}/[ll]\mathsf{LocFin}\left(\t_{\st}\right)\ar@/_{1pc}/[rr] &  & \ar[ll]^{\bot}\sf{LocSph}\left(\t_{\st}\right)\ar@/_{1pc}/[rr] &  & \oplaxlim\left\{ \cdots\xrightarrow{\Omega_{K}}\whp{\t}\right\} \ar[ll]^{\bot}}
\]
where the rightmost adjunction is a reflective subcategory and the
left two adjunctions are co-reflective subcategories. As we will see
however in developing model category structures on these categories
the error becomes irrelevant as the data lost in the passage from
$\sf{LocFin}\left(\t_{\st}\right)$ to $\sf{LocSph}\left(\t_{\st}\right)$
are coherently contractible.

\subsection*{Homotopy-coherent models of...?}

Kan goes on describe a horn-filling condition for mapping spaces of
pointed presheaves $\t_{\st}^{\op}\longrightarrow\S_{\bullet}$ in
$\sf{LocFin}\left(\whp{\t_{\st}}\right)$. In \cite{Brown} this is
replaced with an easier description. Brown defines the set of stable
horns
\[
\text{\ensuremath{\underset{\sf B}{\Lambda}=}}\left\{ \sd{\underset{\sf B}{\Lambda^{i}}\left[n-k\right]\hookrightarrow S^{\left(-k\right)}\left[n\right]}n\in\N,k\in\N,0\leq i\leq k\right\} 
\]
where the stable horns $\underset{\sf B}{\Lambda}^{i}\left[n-k\right]$
are defined much as the simplicial horns $\Lambda^{i}\left[n\right]$
of $\t\left[n\right]$ are. Indeed, in \cite{Brown} it was proven
that this set comprises a generating set of acyclic cofibrations for
a (left transfer of) a pointed Cisinski model category\footnote{again of course not in this much later language.}
which was then proven, in \cite{BousfieldFriedlander}, to be Quillen
equivalent to the Bousfield-Friedlander model structure on sequential
spectra valued in pointed simplicial sets. The promised irrelevance
of the distinction between $\sf{LocFin}\left(\t_{\st}\right)$ and
$\sf{LocSph}\left(\t_{\st}\right)$ is that the adjunction between
them is a Quillen equivalence.

But Cisinski model categories are presentations of $\left(\infty,1\right)$-categories
of homotopy coherent models for limit theories. Consider, for example,
Joyal's model structure for quasi-categories. That model structure
is the minimal Cisinski model structure on the category of simplicial
sets in which the spine inclusions are trivial cofibrations, and the
generating set $\Lambda^{\infty}\left(\sf V\right)$ of anodyne extensions
for that model structure identify the fibrant objects as those which
are homotopy orthogonal to the spine inclusions. But strict orthogonality
of a simplicial set with respect to the set of spine inclusions is
equivalent to the so-called Segal condition on a simplicial set $X$,
\[
X\left(\left[n\right]\right)\liso X\left(\left[1\right]\right)\underset{X\left(\left[0\right]\right)}{\times}X\left(\left[1\right]\right)\underset{X\left(\left[0\right]\right)}{\times}\cdots\underset{X\left(\left[0\right]\right)}{\times}X\left(\left[1\right]\right)
\]
It is thus that we identify quasi-categories as providing a presentation
of the $\left(\infty,1\right)$-category of homotopy coherent $1$-categories.

Returning to the topic of Kan's model then one is wont to ask: what
essentially algebraic structure are Kan spectra to be understood as
homotopy coherent models of?

\subsection*{A Globular perspective on Kan's suspension functor}

In a sense the purely simplicial geometry obscures something of Kan's
suspension. Quasi-categories are as economical as they are precisely
because the directed higher cells are obscure - everything is a composition
datum. Thinking globularly however, or for that matter complicially,
we see that since the simplex category $\t$ is a full subcategory
of the category $\Cat$ of small categories we may think of Kan's
suspension as taking each simplex, thought of as a $1$-category,
and returning the $2$-category with a single object and that original
simplex as its unique non-trivial $\Hom$-category. In this view,
Kan's suspension is but the simplicial avatar of the endofunctor 
\[
\Sigma_{\omega}:\StrCat_{\bullet}\longrightarrow\StrCat_{\bullet}
\]
which takes an $\omega$-category and returns the $\omega$-category
with a single object and that original $\omega$-category as its unique
non-trivial $\Hom$-$\omega$-category. Letting $\Omega_{\omega}$
denote the functor right adjoint to $\Sigma_{\omega}$, it is clear
that the diagram which begat sequential spectra should then be understood
an avatar of the diagram
\[
\cdots\rightarrow\StrCat_{\bullet}\xrightarrow{\Omega_{\omega}}\StrCat_{\bullet}\xrightarrow{\Omega_{\omega}}\StrCat_{\bullet}
\]
and the diagram 
\[
\cdots\rightarrow\whp{\t}\xrightarrow{K_{\bullet}^{*}}\whp{\t}\xrightarrow{K_{\bullet}^{*}}\whp{\t}
\]
should be considered as but the pointing of an avatar of the diagram
\[
\cdots\rightarrow\StrCat\xrightarrow{S^{*}}\StrCat\xrightarrow{S^{*}}\StrCat
\]
where $S^{*}$ is the functor right adjoint to the functor $S:\StrCat\longrightarrow\StrCat$
which send an $\omega$-category $X$ to the $\omega$-category with
two objects $0$ and $1$ with $\Hom\left(0,1\right)=X$ and all other
$\Hom$-categories trivial.\footnote{For those already familiar with Berger's wreath notation, we might
simply say the $\omega$-category $\left[1\right];\left(X\right)$.}

\subsection*{$\protect\Z$-categories}

What then, we ask, is an object of the following limit?
\[
\llim\left\{ \cdots\rightarrow\StrCat\xrightarrow{S^{*}}\StrCat\xrightarrow{S^{*}}\StrCat\right\} 
\]
To answer this we consider that $\StrCat$ is an Eilenberg-Moore category
of algebras for a monad on the category of globular sets. Indeed we
recall that there exists a monad $T$ on globular sets and an isomorphism
$\StrCat\liso\wh{\bG}^{T}$. Then, just as with the functor 
\[
K^{*}:\wh{\t}\longrightarrow\wh{\t}
\]
the functor 
\[
S^{*}:\wh{\bG}\longrightarrow\wh{\bG}
\]
which gave rise to 
\[
S^{*}:\StrCat\longrightarrow\StrCat
\]
 is of the form $\left[S^{\op},\S\right]$ where $S$:$\bG\longrightarrow\bG$
is the obvious endomorphism of the globe category which sends the
$n$-globe $\overline{n}$ to the $\left(n+1\right)$-globe $\overline{n+1}$.

The taking of the Eilenberg-Moore category for a monad $T$ is the
taking of a lax limit over the underlying endofunctor as a diagram
in $\CAT$. As such, since this operation commutes with limits, we
are granted the existence of an isomorphism
\[
\llim\left\{ \cdots\rightarrow\StrCat\xrightarrow{S^{*}}\StrCat\xrightarrow{S^{*}}\StrCat\right\} \liso\reallywidehat{\colim\left\{ \cdots\leftarrow\bG\xleftarrow{S}\text{\ensuremath{\bG}}\xleftarrow{S}\bG\right\} }^{\llim T}
\]
for a monad $\llim T$. More, just as the universal property of the
colimit 
\[
\colim\left\{ \cdots\leftarrow\t\xleftarrow{K}\t\xleftarrow{K}\t\right\} 
\]
 is enjoyed by $\t_{\st}$ whose objects are in bijection with $\Z$,
with maps being generated by $\omega$-many face and degeneracy maps,
the universal property of 
\[
\colim\left\{ \cdots\leftarrow\bG\xleftarrow{S}\text{\ensuremath{\bG}}\xleftarrow{S}\bG\right\} 
\]
 is enjoyed by the category $\bG_{\Z}$ of integer, as opposed to
natural number, indexed globes.
\[
\bG_{\Z}=\left\langle \left.\vcenter{\vbox{\xyR{3pc}\xyC{3pc}\xymatrix{\cdots\ar@/^{.75pc}/[r]^{s}\ar@/_{.75pc}/[r]_{t} & \overline{-1}\ar@/^{.75pc}/[r]^{s}\ar@/_{.75pc}/[r]_{t} & \overline{0}\ar@/^{.75pc}/[r]^{s}\ar@/_{.75pc}/[r]_{t} & \overline{1}\ar@/^{.75pc}/[r]^{s}\ar@/_{.75pc}/[r]_{t} & \cdots}
}}\right|\begin{array}{c}
s\circ t=s\circ s\\
t\circ t=t\circ s
\end{array}\right\rangle 
\]
We refer to presheaves on $\bG_{\Z}$ as $\Z$-globular sets. Then,
to borrow a phrase from \cite{Weiner} (p. 52) : 
\begin{quotation}
``If this \emph{thing $\wh{\bG_{\Z}}^{\llim T}$} \sout{(sin)}
is to have a name, let that name be\emph{ $\Z$-categories} \sout{(simony
or sorcery)}.
\end{quotation}
We'll denote the category of such by $\Z\mhyphen\Cat$.

\section*{Part II: Limits in the $2$-category of categories with arities}

To complete the strict version of the story, we must identify a category
$\sfA$ and a regulus\footnote{In \cite{Kelly} we find the term regulus used to identify the set
of morphisms which carve out orthogonal subcategories - we will use
the term here, honoring and appreciating these ``small rulers''.} $\sf R$ in $\wh{\sf A}$ which are to $\Z$-categories as the simplex
category $\t$ and the spine inclusions are to small $1$-categories.
To clarify precisely what this means and to find such a category and
regulus we appeal to Berger, Mellies and Weber's $2$-category of
categories with arities and Bourke and Garner's synthesis of this
with Day's notion of density presentation.

In \cite{BergerMelliesWeber} those authors develop the $2$-category
of categories with arities $\CwA$ - a sage choice of a $2$-category
of small dense subcategories. That work then uses this $2$-category
to organize numerous characterizations of various algebras on presheaf
categories as instances of a more general theory. In particular, the
same apparatus is seen to characterize $1,2,\dots,n,\dots,\omega$-categories
as reflective subcategories of $\wh{\t},\wh{\Theta_{2}},\dots,\wh{\Theta_{n}},\dots,\wh{\Theta}$
respectively satisfying an ``abstract nerve criterion'', where the
$\Theta_{k}$ are subcategories of Joyal's category $\Theta$.\footnote{A much larger library of such examples is given in \cite{BourkeGarner}.}
Bourke and Garner go on to note that in many cases, including those
just mentioned, by way of Day's notion of a density presentation,
we may in fact easily extract a regulus which carves out precisely
the subcategory of those presheaves satisfying the criterion of the
``abstract nerve theorem''. In the cases mentioned the canonical
regulus is the set of spine inclusions which, together with a little
slight of hand, implies that the inner horns carve out the same subcategories
(this recovers the earlier result of \cite{Berger1}).

What we do in this paper is to show that the obvious $2$-functor
$\CwA\longrightarrow\CAT^{\bfTwo}$ which sends a full-and-faithful
right adjoint to the underlying functor reflects weighted limits.
From this it follows (see Example \ref{exa:A-cellular-description-of-Z-categories})
that 
\[
\llim\left\{ \cdots\rightarrow\StrCat\xrightarrow{S^{*}}\StrCat\xrightarrow{S^{*}}\StrCat\right\} 
\]
is a reflective subcategory of $\reallywidehat{\colim\left\{ \cdots\leftarrow\Theta\xleftarrow{S}\Theta\xleftarrow{S}\Theta\right\} }$
and more that it is carved out by the regulus whose homotopical and
groupoidal analogue carved out Kan's semisimplicial spectra in $\whp{\t_{\st}}$.
\begin{rem}
Since, in this work, we do not actually need to perform any explicit
computations with the categories 
\[
\Theta,\ \underrightarrow{\sf{oplaxlim}}\left\{ \cdots\leftarrow\Theta\xleftarrow{S}\Theta\xleftarrow{S}\Theta\right\} \ \mathrm{or}\ \colim\left\{ \cdots\leftarrow\Theta\xleftarrow{S}\Theta\xleftarrow{S}\Theta\right\} 
\]
 we leave explicit descriptions of those categories to the second
paper.
\end{rem}

\subsection*{Appendix A - Generalized spectrification}

All four $2$-categorical limits over $1$-diagrams are but instances
of the still more general enriched categorical notion of weighted
limits. If limits of diagrams of categories are compatible collections
of objects, then category-weighted limits are compatible collections
of diagrams valued in the constituent categories of a diagram of categories.
It is in this language that we identify the usual left adjoint to
the inclusion of $\Omega$-spectrum objects into sequential spectrum
objects as an instance of a rather more general reflection of weighted
limits over a $1$-diagram into conical limits over that same diagram.

\subsection*{Appendix B - Oplax limits, oplax weights}

This second appendix explains why the weight $\left(\sf J\downarrow\_\right)^{\op}:\sf J\longrightarrow\Cat$
computes oplax limits as weighted limits. This is well known, but
not in general to topologists\footnote{Consider, for example, the author.},
as such we include it for the author's edification and for the interested
reader's.

\subsection*{Appendices C\&D - Technical lemmata}

In a third appendix we assemble some technical lemmata which would
clutter the exposition were they in the body of the text.

\section*{Acknowledgements}

The author is particularly indebted to Steve Lack whose suggestions,
questions and pressure have in large part led to this work and this
author's conversion to Australian-style category theory. This author
was supported by Australian Research Council Discovery Project grant:
DP190102432 during the preparation of this work.

\newpage{}

\tableofcontents{}

\part{\label{part:A-2-Categorical-Treatment-of-Kans}A $2$-Categorical
Treatment of Kan's ``Semisimplicial Spectra''}

In this part we use some $2$-categorical notions to give a description
of the ``semisimplicial spectra'' of \cite{Kan}. Along the way,
in Remark \ref{rem:Chen-Kriz-Pultr-are-wrong}, we correct a subtle
error of \cite{Kan} and another related one found in \cite{ChenKrizPultr}.

\section{Kan's Suspension: The Kan decalage and its associated suspension
and loop-space functors}
\begin{defn}
Denote by $K$ the functor
\[
\xyR{0pc}\xyC{5pc}\xymatrix{K:\t\ar[r] & \t\\
\left[n\right]\ar@{|->}[r] & \left[n+1\right]\\
d^{i}\ar@{|->}[r] & d^{i}\\
s^{j}\ar@{|->}[r] & s^{j}
}
\]
and observe that there are natural transformations 
\[
\id_{\t}\overset{\alpha}{\Longrightarrow}K\overset{\beta}{\Longleftarrow}\left[0\right]
\]
 whose components at $\left[n\right]\in\text{\ensuremath{\t\ }}$
correspond to the maps $d^{n+1}:\left[n\right]\longrightarrow\left[n+1\right]$
and $\left\{ n+1\right\} :\left[0\right]\longrightarrow\left[n+1\right]$
. These data comprise the \textbf{Kan decalage} (see \cite{CisinskiMaltsiniotis})
and we illustrate them in Figure \ref{fig:The-Kan-decalage}. \textbf{Kan's
small suspension functor} and its associated loop-space functor (see
\cite{Kan}) are given by taking the quotient of $K$ by $\alpha$
and $\beta$.

\begin{figure}[H]
\begin{tikzcd}
		&
		\begin{tikzpicture}[line join = round, line cap = round]
			\coordinate (2) at (0,{sqrt(2)},0);
			\coordinate (0) at ({-.5*sqrt(3)},0,-.5);
			\coordinate (1) at (0,0,1);
			\begin{scope}[decoration={markings,mark=at position 0.5 with {\arrow{to}}}]
			\draw[fill=gray,fill opacity=.5] (0)--(1)--(2)--cycle;
			\draw[postaction={decorate}] (0)--(1);
			\draw[postaction={decorate}] (1)--(2);
			\draw[postaction={decorate}] (0)--(2);
			\end{scope}
		\end{tikzpicture}
		\ar[mapsto]{dl}[swap]{\textrm{id}}
		\ar[mapsto]{d}[description]{\textrm{K}}
		\ar[mapsto]{dr}[]{[0]}
			&
				\\
	\begin{tikzpicture}[line join = round, line cap = round]
		\coordinate (2) at (0,{sqrt(2)},0);
		\coordinate (0) at ({-.5*sqrt(3)},0,-.5);
		\coordinate (1) at (0,0,1);
		\begin{scope}[decoration={markings,mark=at position 0.5 with {\arrow{to}}}]
		\draw[fill=gray,fill opacity=.5] (0)--(1)--(2)--cycle;
		\draw[postaction={decorate}] (0)--(1);
		\draw[postaction={decorate}] (1)--(2);
		\draw[postaction={decorate}] (0)--(2);
		\end{scope}
	\end{tikzpicture}
	\ar{r}[swap]{\alpha_{[2]}}
		&
		\begin{tikzpicture}[line join = round, line cap = round]
			\coordinate (2) at (0,{sqrt(2)},0);
			\coordinate (0) at ({-.5*sqrt(3)},0,-.5);
			\coordinate (1) at (0,0,1);
			\coordinate (3) at ({.5*sqrt(3)},0,-.5);
			\begin{scope}[decoration={markings,mark=at position 0.5 with {\arrow{to}}}]
				\draw[densely dotted,postaction={decorate}] (0)--(3);
				\draw[fill=lightgray,fill opacity=.5] (1)--(2)--(3)--cycle;
				\draw[fill=gray,fill opacity=.5] (0)--(1)--(2)--cycle;
				\draw[postaction={decorate}] (0)--(1);
				\draw[postaction={decorate}] (1)--(2);
				\draw[postaction={decorate}] (2)--(3);
				\draw[postaction={decorate}] (1)--(3);
				\draw[postaction={decorate}] (0)--(2);
			\end{scope}
		\end{tikzpicture}
			&
			\bullet
			\ar{l}{\beta_{[0]}}
				\\
\end{tikzcd}\caption{\label{fig:The-Kan-decalage}The Kan decalage illustrated: the case
of $\left[2\right]$}
\end{figure}
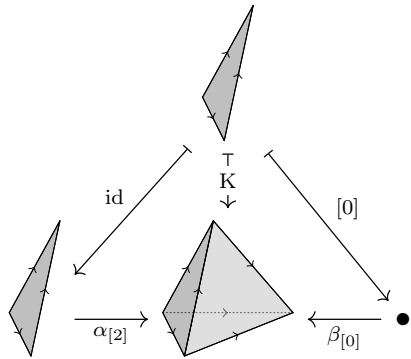
\end{defn}

Let the functor $\Sigma_{K}:\whp{\t}\longrightarrow\whp{\t}$ be the
left Kan extension
\[
\xyR{0pc}\xyC{5pc}\Sigma_{K}:=\mathsf{Lan}_{\Yo_{\bullet}}\left(\vcenter{\vbox{\xymatrix{\t_{\bullet}\ar[r] & \whp{\t}\\
\left[n\right]\ar@{|->}[r] & \nf{\t\left[n+1\right]_{+}}{\t\left[n\right]_{+}\vee\t\left[0\right]_{+}}
}
}}\right)
\]
with $\Yo_{\bullet}$ being the pointed Yoneda embedding $\t_{\bullet}\xrightarrow{\Yo_{\bullet}}\whp{\t}$
and let $\Omega_{K}$ denote its right adjoint.

Now, as proven in \cite{Kan}, we have homeomorphisms
\[
\left|\Sigma_{K}X\right|\liso\Sigma\left|X\right|
\]
 for all pointed simplicial sets $X$ where $\left|\_\right|:\whp{\t}\longrightarrow\T_{\bullet}$
is the usual geometric realization functor and $\Sigma$ is the usual
reduced suspension. More, the unit $\eta:\id_{\whp{\t}}\Longrightarrow\Omega_{K}\Sigma_{k}$
of the adjunction $\Sigma_{K}\dashv\Omega_{K}$ is invertible
\begin{prop}
(Propositon 3 \cite{ChenKrizPultr}) The unit $\eta:\id_{\whp{\t}}\Longrightarrow\Omega_{K}\Sigma_{k}$
of the adjunction $\Sigma_{K}\dashv\Omega_{K}$ is invertible.
\end{prop}

\begin{proof}
The proof given in \cite{ChenKrizPultr} is correct modulo the detail
we'll explore carefully in Remark \ref{rem:Chen-Kriz-Pultr-are-wrong}.
\end{proof}
\begin{rem}
\label{rem:unit-of-Kans-adjunction-is-invertible} The invertibility
of the unit is an important distinction between this adjunction and
the associated one in pointed spaces. For any set $E$ we have that
both the pointed simplicial set
\[
\Sigma_{K}\left(\bigvee_{e\in E}\t\left[0\right]_{+}\right)
\]
and the pointed space 
\[
\Sigma\left(\sf{Disc}\left(E\right)_{+}\right)
\]
are wedges of $E$-many circles. In $\T_{\bullet}$ however we may
compose these circles whereas in $\whp{\t}$ we cannot compose them
without first taking a fibrant replacement for $\Sigma_{K}\left(\bigvee_{e\in E}\t\left[0\right]_{+}\right)$
in the pointed Kan (a.k.a. test) model structure.\footnote{or Joyal's model structure for that matter, if one is so inclined.}
\end{rem}

\begin{rem}
\label{rem:Chen-Kriz-Pultr-are-wrong} In \cite{ChenKrizPultr} we
find a functor $\Omega:\whp{\t}\longrightarrow\whp{\t}$ defined as
follows.
\[
\xyR{0pc}\vcenter{\vbox{\xymatrix{\Omega:\whp{\t}\ar[r] & \whp{\t}\\
X\ar@{|->}[r] & \Omega\left(X\right):\t^{\op}\ar[r] & \S_{\bullet}\\
 & \left[n\right]\ar@{|->}[r] & \left\{ \sd{\left[n+1\right]\xrightarrow{x}X}d_{n+1}\left(x\right)=\bullet\right\} 
}
}}
\]
Those authors then go on to describe the functor left adjoint to $\Omega$
as follows.

Let $\t_{0}$ in $\t$ be wide subcategory thereof generated by the
objects and solid arrows of the following diagram (the dotted arrows
are present only to guide the eye, they are \emph{not} generators
of $\t_{0}$)
\[
\vcenter{\vbox{\xymatrix{\left[0\right]\ar@<1ex>[r]\ar@<-1ex>[r] & \ar@{..>}[l]\left[1\right]\ar@<2ex>[r]\ar[r]\ar@{..>}@<-2ex>[r] & \ar@{..>}@<1ex>[l]\ar@<-1ex>[l]\left[2\right]\ar@<3ex>[r]\ar@<1ex>[r]\ar@<-1ex>[r]\ar@{..>}@<-3ex>[r] & \ar@{..>}@<2ex>[l]\ar[l]\ar@<-2ex>[l]\cdots}
}}
\]
Let $\iota:\t_{0}\longrightarrow\t$ be the name of the inclusion
and see that $\iota$ induces an adjoint triple $\iota_{\bullet!}\dashv\iota_{\bullet}^{*}\dashv\iota_{\bullet}^{!}$
\[
\xyR{0pc}\xyC{3pc}\xymatrix{\whp{\t_{0}}\ar@/^{1pc}/[rr]^{\iota_{\bullet!}}\ar@/_{1pc}/[rr]_{\iota_{\bullet}^{!}} &  & \whp{\t}\ar[ll]_{\bot}^{\bot}|-{\iota_{\bullet}^{*}}}
\]
Let $\Sigma_{0}:\whp{\t}\longrightarrow\whp{\t_{0}}$ be the functor
defined as follows.
\[
\xyR{0pc}\vcenter{\vbox{\xymatrix{\Sigma_{0}:\whp{\t}\ar[r] & \whp{\t_{0}}\\
X\ar@{|->}[r] & \Sigma_{0}\left(X\right):\t_{0}^{\op}\ar[r] & \S_{\bullet}\\
 & \left[n+1\right]\ar@{|->}[r] & X\left(\left[n\right]\right)\\
 & \left[0\right]\ar@{|->}[r] & \left\{ \bullet\right\} 
}
}}
\]
It's not hard to see that we have an natural isomorphism of functors
$\Sigma_{K}\xRightarrow{\sim}\iota_{\bullet!}\Sigma_{0}$, indeed
the co-continuity of both functors provides that it suffices to check
the claim on simplices and it can be checked by hand that $\Sigma_{K}\left(\t\left[n\right]_{+}\right)$
agrees with $\iota_{\bullet!}\Sigma_{0}\left(\t\left[n\right]_{+}\right)$.

The functor $\Omega$ of \cite{ChenKrizPultr} however is \emph{not}
right adjoint to the functor $\iota_{\bullet!}\Sigma_{0}$ as is claimed
in \cite{ChenKrizPultr}. We will construct a counter-example. Consider
the pointed simplicial set $X$ defined by the following push-out.
\[
\vcenter{\vbox{\xyR{3pc}\xyC{3pc}\xymatrix{\t\left[0\right]_{+}\ar[r]^{d_{+}^{1}}\ar[d] & \t\left[1\right]_{+}\ar[d]\\
\bullet\ar[r] & X\pushoutcorner
}
}}
\]
By construction the unique non-degenerate $1$-cell of $X$, call
it $x\in X\left(\left[1\right]\right)$, has as its $1^{{\rm st}}$
face the base-point $\bullet$. Thus we compute
\begin{eqnarray*}
\whp{\t}\left(\t\left[0\right]_{+},\Omega\left(X\right)\right) & = & \Omega\left(X\right)\left(\left[0\right]\right)\\
 & = & \left\{ \sd{\t\left[1\right]_{+}\xrightarrow{y}X}d^{1}\left(y\right)=\bullet\right\} \\
 & = & \left\{ x,\bullet\right\} 
\end{eqnarray*}
However, we find that
\[
\whp{\t}\left(\iota_{\bullet!}\Sigma_{0}\left(\left[0\right]\right),X\right)=\left\{ \bullet\right\} 
\]
as the only $1$-cell of $X$ whose $1^{{\rm st}}$ and $0^{{\rm th}}$
faces are both $\bullet$ is the degenerate one at the base-point. 
\end{rem}

\section{Sequential spectra as oplax limits, $\Omega$-spectra as pseudo-limits}

Set $\sf X:\Z_{\leq0}\longrightarrow\CAT$ to be the diagram which
sends each $-n$ to $\whp{\t}$ and each $-\left(n+1\right)\longrightarrow-n$
to $\Omega_{K}$.
\[
\cdots\xrightarrow{\Omega_{K}}\whp{\t}\xrightarrow{\Omega_{K}}\whp{\t}
\]
As we stated in the introduction, $\llim\sf X$ is $\Omega\mhyphen\Sp\left(\whp{\t},\Sigma_{K}\dashv\Omega_{K}\right)$
and $\oplaxcolim\sf X$ is $\Sp\left(\whp{\t},\Sigma_{K}\dashv\Omega_{K}\right)$,
the usual categories of $\Omega$-spectrum and spectrum objects respectively
(this is treated in detail in Example \ref{exa:Sequential-spectra-}).
Then, for familiar reasons (or alternatively our weighted limit treatment
of spectrification, Theorem \ref{thm: generalized-spectrification-strict-limit})
we get a spectrification adjunction 
\[
\xyC{5pc}\xymatrix{\llim\sf X\ar@/^{.625pc}/[r]^{\sf{sp}_{\sf X}} & \oplaxlim\sf X\ar@/^{.625pc}/[l]_{\bot}^{\t_{\sf X}}}
\]
We recall that the category $\whp{\t}$ admits the pointed test model
structure Quillen equivalent to pointed spaces with the classical
model structure.%
{} Therefore, seeing as $\oplaxlim\sf X$ is equivalent to the category
of sequential spectrum objects valued in $\whp{\t}$ with respect
to the adjunction $\Sigma_{K}\dashv\Omega_{K}$, we may equip $\oplaxlim\sf X$
with \cite{Hovey}'s stable model structure.

Then, since we have homeomorphisms $\left|\Sigma_{K}X\right|\liso\Sigma\left|X\right|$
for all pointed simplicial sets $X$, it follows that Hovey's stable
model structure on $\oplaxlim\sf X$ is Quillen equivalent to the
Bousfield-Friedlander model structure on sequential spectra, whence
it presents the $\left(\infty,1\right)$-category of spectra. Kan's
observation, as we will elaborate below in the next section, is that
by having chosen a particular presentation of the suspension, i.e.
$\Sigma_{K}$, the category $\oplaxlim\sf X$ admits an elegant description
as a subcategory of a category of pointed presheaves.

\section{Factorization through $\left[\left(\_\right)_{\bullet}^{\protect\op},\protect\S_{\bullet}\right]_{\bullet}:\protect\Cat^{\protect\sf{coop}}\protect\longrightarrow\protect\CAT$
and the collage construction}

Let 
\[
\wh{\sf K}:\Z_{\leq0}\longrightarrow\CAT
\]
be the diagram which sends each $-n$ to $\whp{\t}$ and each $-\left(n+1\right)\longrightarrow-n$
to the functor 
\[
\xyR{0pc}\xyC{5pc}\xymatrix{K_{\bullet}^{*}:\whp{\t}\ar[r] & \whp{\t}\\
X\ar@{|->}[r] & \whp{\t}\left(\t K\left(\_\right)_{+},X\right)
}
\]
We observe that this functor factors through $\Cat^{\sf{coop}}$ as
follows
\[
\xyR{2.5pc}\xyC{2.5pc}\xymatrix{\Z_{\leq0}\ar[rrrr]^{\wh{\sf K}}\ar[dr]^{\sim} &  &  &  & \whp{\t}\\
 & \Z_{\geq0}^{\sf{coop}}\ar[dr]^{\sf K^{\sf{coop}}}\\
 &  & \Cat^{\sf{coop}}\ar[uurr]_{\left[\_,\S_{\bullet}\right]}
}
\]
where $\sf{K:\Z_{\geq0}\longrightarrow\Cat}$ is the functor which
sends each $n$ to $\t$ and each $n\longrightarrow\left(n+1\right)$
to $K$. Since the functor $\left[\_,\S_{\bullet}\right]:\Cat^{\sf{coop}}\longrightarrow\CAT$
sends conical, pseudo, lax and oplax colimit to conical, pseudo, lax
and oplax limits respectively it follows that we have further isomorphisms
\begin{eqnarray*}
\llim\wh{\sf K} & \liso & \reallywidehat{\colim\sf K}_{\bullet}\\
\underleftarrow{\sf{pslim}}\wh{\sf K} & \liso & \reallywidehat{\underrightarrow{\sf{pslim}}\sf K}_{\bullet}\\
\oplaxlim\wh{\sf K} & \liso & \reallywidehat{\underrightarrow{\sf{oplaxlim}}\sf K}_{\bullet}
\end{eqnarray*}
What's more, since $K_{\bullet}^{*}$ is an iso-fibration, it follows\footnote{as it did in for $\sf X$ in the previous example.}
that $\wh{\sf K}:\Z_{\leq0}\longrightarrow\CAT$ is fibrant in the
injective-canonical model structure, so we have isomorphisms
\[
\llim\wh{\sf K}\liso\underleftarrow{\sf{pslim}}\wh{\sf K}\liso\reallywidehat{\underrightarrow{\sf{pslim}}\sf K}_{\bullet}\liso\reallywidehat{\colim\sf K}_{\bullet}
\]

The categories $\underrightarrow{\sf{oplaxlim}}\sf{\sf K}$ and $\colim\sf{\sf K}$
admit explicit descriptions. Recall that the collage construction
presents the oplax colimits of an arrow in $\Prof$. Iterating this
construction gives us a presentation of the category $\underrightarrow{\mathsf{oplax}}\sf K$.
Indeed, set $\t_{\mathsf{coll}}$ to be the category generated by
\[
\coprod_{-m\in\Z_{\leq0}}\t
\]
together with 
\[
\t_{\mathsf{coll}}\left(\left(\left[n\right],-m\right),\left(\left[\ell\right],-\left(m+k\right)\right)\right)=\t\left(\left[n+k\right],\left[\ell\right]\right)
\]
and see that $\t_{\mathsf{coll}}$ enjoys the universal property of
$\underrightarrow{\mathsf{oplaxlim}}\sf K$.
\begin{rem}
The category $\t_{\sf{coll}}$ defined above can also be described
as being generated by the diagram 
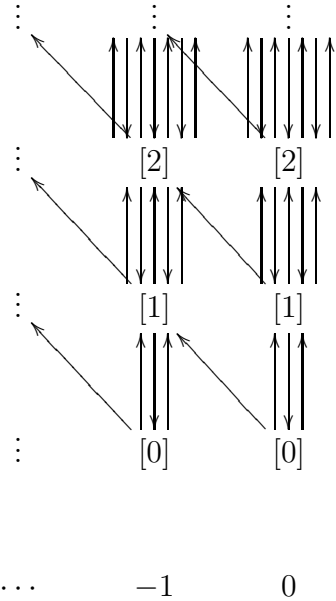
\[
\xyR{2.5pc}\xyC{2.5pc}\xymatrix{\vdots & \ar@<2ex>[d]\ar[d]\ar@<-2ex>[d]\vdots & \vdots\ar@<2ex>[d]\ar[d]\ar@<-2ex>[d]\\
\vdots & \ar[ul]\ar@<1ex>[d]\ar@<-1ex>[d]\left[2\right]\ar@<3ex>[u]\ar@<1ex>[u]\ar@<-1ex>[u]\ar@<-3ex>[u] & \ar[ul]\left[2\right]\ar@<1ex>[d]\ar@<-1ex>[d]\ar@<3ex>[u]\ar@<1ex>[u]\ar@<-1ex>[u]\ar@<-3ex>[u]\\
\vdots & \ar[ul]\ar[d]\left[1\right]\ar@<2ex>[u]\ar[u]\ar@<-2ex>[u] & \ar[ul]\left[1\right]\ar[d]\ar@<2ex>[u]\ar[u]\ar@<-2ex>[u]\\
\vdots & \ar[ul]\left[0\right]\ar@<1ex>[u]\ar@<-1ex>[u] & \ar[ul]\left[0\right]\ar@<1ex>[u]\ar@<-1ex>[u]\\
\cdots & -1 & 0
}
\]
whose vertical maps are the usual face and degeneracy maps, subject
to the usual identities, and whose diagonal maps act so that the squares
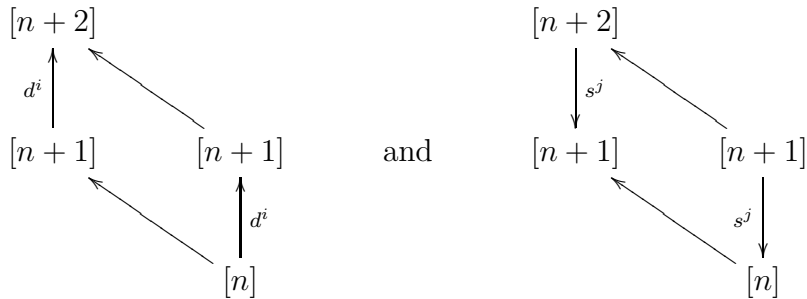
\[
\xyR{2.5pc}\xyC{2.5pc}\xymatrix{\left[n+2\right] &  &  & \left[n+2\right]\ar[d]^{s^{j}}\\
\left[n+1\right]\ar[u]^{d^{i}} & \ar[ul]\left[n+1\right] & \mathrm{and} & \left[n+1\right] & \ar[ul]\left[n+1\right]\ar[d]_{s^{j}}\\
 & \ar[ul]\left[n\right]\ar[u]_{d^{i}} &  &  & \ar[ul]\left[n\right]
}
\]
commute. 
\end{rem}

The universal property of the strict colimit $\colim\sf K$ is similarly
enjoyed by $\t_{\st}$, the category whose objects are in bijection
with the set $\Z$, 
\[
\Ob\left(\t_{\st}\right)=\left\{ \sd{\left[z\right]}z\in\Z\right\} 
\]
with $\Hom$-sets generated by maps 
\[
\left\{ \sd{d^{i}:\left[m-1\right]\longrightarrow\left[m\right]}i\in\N\right\} 
\]
and 
\[
\left\{ \sd{s^{j}:\left[m+1\right]\longrightarrow\left[m\right]}j\in\text{\ensuremath{\N\ }}\right\} 
\]
subject to (un-bounded) simplicial identities. One gets to the category
$\colim\sf K$ from the category $\underrightarrow{\sf{oplaxlim}}\sf K$
by identifying objects along the diagonal maps of $\t_{\sf{coll}}$.
Indeed, this defines a functor $\rho:\t_{\sf{coll}}\longrightarrow\t_{\sf{st}}$,
whence an adjoint triple 
\[
\xyR{0pc}\xyC{3pc}\xymatrix{\reallywidehat{\t_{\sf{coll}}}_{\bullet}\ar@/^{1pc}/[rr]^{\rho_{\bullet}^{!}}\ar@/_{1pc}/[rr]_{\rho_{\bullet}^{!}} &  & \reallywidehat{\t_{\sf{st}}}_{\bullet}\ar[ll]_{\bot}^{\bot}|-{\rho_{\bullet}^{*}}}
\]
of categories of pointed presheaves.

Now, since $\Sigma_{K}$ is a quotient of $K_{\bullet!}=\Lan_{\Yo_{\bullet}}\left(\Yo_{\bullet}K\right)$,
we have oplax squares
\[
\xyR{2.5pc}\xyC{2.5pc}\xymatrix{\whp{\t}\ar[d]_{\id}\ar@{=>}[dr] & \ar[l]_{K_{\bullet!}}\whp{\t}\ar[d]^{\id}\\
\whp{\t} & \ar[l]^{\Sigma_{K}}\whp{\t}
}
\]
whence, taking mates, we have oplax squares
\[
\xyR{2.5pc}\xyC{2.5pc}\xymatrix{\whp{\t}\ar[d]_{\id}\ar[r]^{\Omega_{K}} & \whp{\t}\ar[d]^{\id}\\
\whp{\t}\ar@{=>}[ur]\ar[r]_{K_{\bullet}^{*}} & \whp{\t}
}
\]
These squares can be pasted together providing an oplax natural transformation
$\sf X\xRightarrow{\sf{oplax}}\wh{\sf K}$, whence we have a functor
$\oplaxlim\sf X\longrightarrow\oplaxlim\wh{\sf K}$. But, since $\oplaxlim\sf{\wh K}$
is a category of pointed presheaves, this map admits a more explicit
description. Indeed that functor is a restricted Yoneda embedding,
or nerve.

Observe that, for each -$n\in\Z_{\leq0}$, there is a pseudo-cone
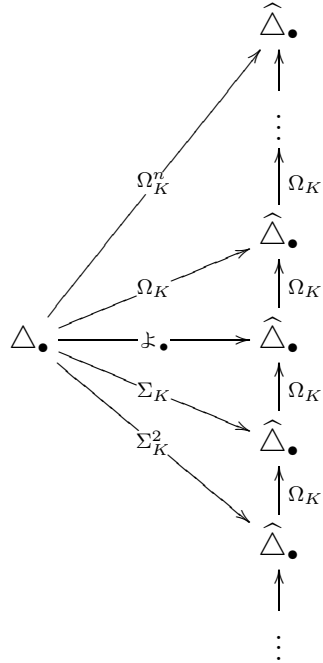
\[
\xyR{1.5pc}\xyC{6pc}\xymatrix{ & \whp{\t}\\
 & \vdots\ar[u]\\
 & \whp{\t}\ar[u]_{\Omega_{K}}\\
\t_{\bullet}\ar[r]|-{\Yo_{\bullet}}\ar[ur]|-{\Omega_{K}}\ar[uuur]|-{\Omega_{K}^{n}}\ar[dr]|-{\Sigma_{K}}\ar[ddr]|-{\Sigma_{K}^{2}} & \whp{\t}\ar[u]_{\Omega_{K}}\\
 & \whp{\t}\ar[u]_{\Omega_{K}}\\
 & \whp{\t}\ar[u]_{\Omega_{K}}\\
 & \vdots\ar[u]
}
\]
where the instance of $\Yo_{\bullet}$ targets the $-n^{th}$ copy
of $\whp{\t}$.

These pseudo-cones define maps 
\[
\Phi^{\co{-n}{\infty}}:\t_{\bullet}\longrightarrow\underleftarrow{\sf{pslim}}\sf X\liso\llim\sf X
\]

\begin{rem}
In the perhaps more familiar notation of sequential spectrum objects
the maps are described by the formula
\[
\left\{ \Phi^{\co{-n}{\infty}}\left(\t\left[r\right]\right)_{-i}=\begin{cases}
\bullet & i<n\\
\Sigma_{K}^{\left(i-n\right)}\t\left[r\right]_{+} & i\geq n
\end{cases}\right\} 
\]
\end{rem}

We observe then that the diagrams

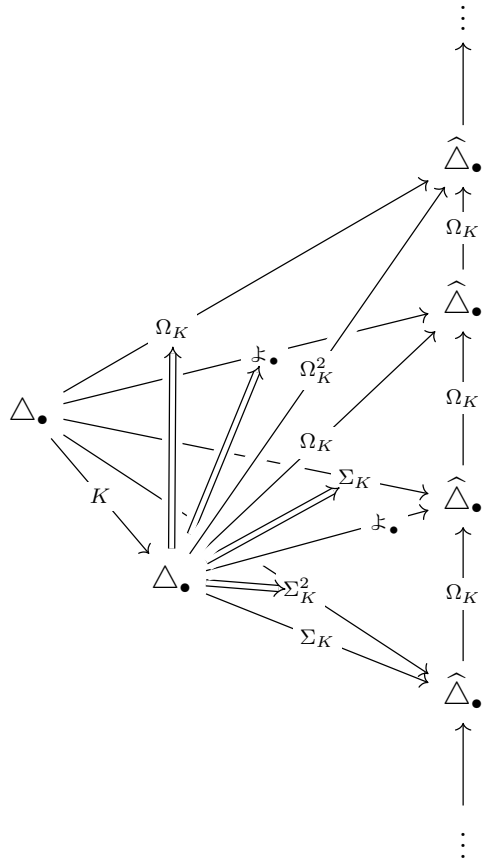
\begin{figure}[H]
\[
	\begin{tikzcd}[column sep=small, row sep=small]
		&&&&&&&& \vdots
		\\
		\\
		\\
		&&&&&&&& {\widehat{\triangle}_{\bullet}}
		\\
		\\
		\\
		&&&&&&&& {\widehat{\triangle}_{\bullet}}
		\\
		\\
		\triangle_{\bullet} 
		\\
		&&&&&&&& {\widehat{\triangle}_{\bullet}}
		\\
		&& \triangle_{\bullet}
		\\
		\\
		&&&&&&&& {\widehat{\triangle}_{\bullet}}
		\\
		\\
		\\
		&&&&&&&& \vdots
		\arrow[from=4-9, to=1-9]
		\arrow["{\Omega_{K}}"{description}, from=7-9, to=4-9]
		\arrow["{\Omega_{K}}"{description}, from=10-9, to=7-9]
		\arrow["{\Omega_{K}}"{description}, from=13-9, to=10-9]
		\arrow[from=16-9, to=13-9]
		\arrow["K"{description}, from=9-1, to=11-3]
		\arrow["{\Sigma_{K}}"{description}, from=11-3, to=13-9]
		\arrow[""{name=3, anchor=center, inner sep=0}, "{\Sigma^{2}_{K}}"{description}, from=9-1, to=13-9, pos=.65]
		\arrow[shorten >=5pt, Rightarrow, from=11-3, to=3]
		\arrow["{\Yo_{\bullet}}"{description}, pos=.8, from=11-3, to=10-9, crossing over] 		\arrow[""{name=1, anchor=center, inner sep=0}, "{\Sigma_{K}}"{description}, pos=.8, from=9-1, to=10-9, crossing over]
		\arrow[shorten <=0pt, shorten >=6pt, Rightarrow, from=11-3, to=1, crossing over]
		\arrow["{\Omega_{K}}"{description}, from=11-3, to=7-9, crossing over]
		\arrow[""{name=0, anchor=center, inner sep=0}, "{\Yo_{\bullet}}"{description}, from=9-1, to=7-9, crossing over, pos =.55]
		\arrow[shorten <=8pt, shorten >=4pt, Rightarrow, from=11-3, to=0, crossing over]
		\arrow[""{name=2, anchor=center, inner sep=0}, "{\Omega_{K}}"{description}, from=9-1, to=4-9, crossing over, pos=.3]
		\arrow["{\Omega^{2}_{K}}"{description}, from=11-3, to=4-9, crossing over] 		\arrow[shorten <=2pt, shorten >=6pt, Rightarrow, from=11-3, to=2, crossing over]
	\end{tikzcd}
\]
\end{figure}

where the $2$-cells are de-suspensions and suspensions of the quotient
$K_{\bullet!}\Longrightarrow\Sigma_{K}$, induce an oplax co-cone

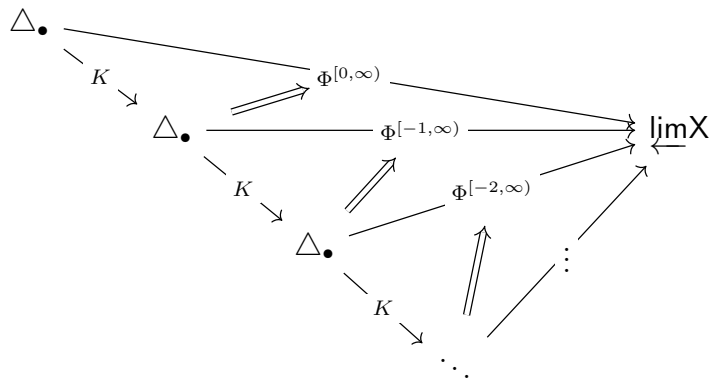
\begin{figure}[H]
\[
\begin{tikzcd}
	\triangle_{\bullet} \\
	& \triangle_{\bullet} &&&& \llim \mathsf{X} \\
	&& \triangle_{\bullet} \\
	&&& \ddots
	\arrow["K"{description}, from=1-1, to=2-2]
	\arrow["K"{description}, from=2-2, to=3-3]
	\arrow["K"{description}, from=3-3, to=4-4]
	\arrow[""{name=n, anchor=center, inner sep=0},"{\vdots}"{description}, from=4-4, to=2-6]
	\arrow[""{name=2, anchor=center, inner sep=0},"{\Phi^{[-2,\infty)}}"{description}, from=3-3, to=2-6]
	\arrow[""{name=1, anchor=center, inner sep=0}, "{\Phi^{[-1,\infty)}}"{description}, from=2-2, to=2-6]
	\arrow[""{name=0, anchor=center, inner sep=0}, "{\Phi^{[0,\infty)}}"{description}, from=1-1, to=2-6]
	\arrow[shorten >=15pt, shorten <=10pt, Rightarrow, from=2-2, to=0]
	\arrow[shorten >=13pt, shorten <=5pt, Rightarrow, from=3-3, to=1]
	\arrow[shorten >=13pt, shorten <=5pt, Rightarrow, from=4-4, to=2]
\end{tikzcd}
\]
\end{figure}

whence a map 
\[
\Phi:\underrightarrow{\mathsf{oplaxlim}}\sf K\longrightarrow\underleftarrow{\sf{pslim}}\sf X\liso\llim\sf X
\]
We may of course compose $\Phi$ with the reflective subcategory 
\[
\sf{sp_{\sf X}}\dashv\t_{\sf X}:\llim\sf X\leftrightarrows\oplaxlim\sf X
\]
That composition $\t_{\sf X}\Phi:\oplaxcolim\sf K\longrightarrow\oplaxcolim\sf X$
induces an adjoint triple.
\[
\xyR{0pc}\xyC{3pc}\xymatrix{\reallywidehat{\t_{\sf{coll}}}_{\bullet}\ar@/^{1pc}/[rr]^{\left(\t_{\sf X}\Phi\right)_{\bullet!}}\ar@/_{1pc}/[rr]_{\left(\t_{\sf X}\Phi\right)_{\bullet}^{!}} &  & \oplaxlim\sf X\ar[ll]_{\bot}^{\bot}|-{\left(\t_{\sf X}\Phi\right)_{\bullet}^{*}}}
\]
As a note before moving on: we observe that $\left(\t_{\sf X}\Phi\right)_{\bullet}^{*}$
recovers the oplax limit $\oplaxlim\sf X\longrightarrow\oplaxlim\wh{\sf K}$.
Composing the adjunction $\left(\t_{\sf X}\Phi\right)_{\bullet}^{*}\dashv\left(\t_{\sf X}\Phi\right)_{\bullet}^{!}$
above with the adjunction $\rho_{\bullet!}\dashv\rho_{\bullet}^{*}$
provides a third adjunction 

\[
\xyR{0pc}\xyC{3pc}\xymatrix{\reallywidehat{\t_{\st}}_{\bullet}\ar@/_{1pc}/[rr]_{\left(\t_{\sf X}\Phi\right)_{\bullet}^{!}\circ\rho_{\bullet}^{*}} &  & \oplaxlim\sf X\ar[ll]_{\rho_{\bullet!}\circ\left(\t_{\sf X}\Phi\right)_{\bullet}^{*}}^{\bot}}
\]

\section{Double duty and slick tricks}

We will be concerned with two interesting and important subcategories
of $\whp{\t_{\st}}$. Both are full subcategories on colimits of objects
of specific forms.
\begin{defn}
For each $n\in\N$ and $z\in\Z$, define the pointed presheaf $\t_{\st}\nf{\left[z\right]}{d^{>n}=*}$
by the following pushout in $\whp{\t_{\st}}$.
\[
\xyR{3pc}\xyC{10pc}\xymatrix{\left(\bigvee_{1\leq i}\t_{\st}\left[z-1\right]_{+}\right)\ar[r]\sp(0.65){\left(\vee_{1\leq i}d^{n+i}\right)}\ar[d] & \t_{\st}\left[z\right]_{+}\ar[d]\\
\bullet\ar[r] & \pushoutcorner\t_{\st}\nf{\left[z\right]}{d^{>n}=*}
}
\]
Similarly, for each $n\in\N$ and $z\in\Z$, let 
\[
S^{\left(z-n\right)}\left[n\right]=\nf{\t_{\st}\left[z\right]}{d^{n}d^{n-1}\cdots d^{1}d^{0}=d^{>n}=*}
\]
 be defined by the following pushout in $\whp{\t_{\st}}$.
\[
\xyR{3pc}\xyC{10pc}\xymatrix{\left(\t_{\st}\left[z-\left(n+1\right)\right]_{+}\right)\vee\left(\bigvee_{1\leq i}\t_{\st}\left[z-1\right]_{+}\right)\ar[r]\sp(0.65){d^{n}d^{n-1}\cdots d^{1}d^{0}\bigvee\left(\vee_{1\leq i}d^{n+i}\right)}\ar[d] & \t_{\st}\left[z\right]_{+}\ar[d]\\
\bullet\ar[r] & \pushoutcorner S^{\left(z-n\right)}\left[n\right]
}
\]
We define $\sf{LocFin}\left(\t_{\st}\right)$ to be the full subcategory
of $\whp{\t_{\st}}$ on objects which are colimits of objects of the
form $\t_{\st}\nf{\left[z\right]}{d^{>n}=*}$ and we define $\sf{LocSph}\left(\t_{\st}\right)$
to be the full subcategory of $\whp{\t_{\st}}$ on colimits of objects
of the form $S^{\left(z-n\right)}\left[n\right]$.
\end{defn}

\begin{rem}
As we discussed in the introduction, the pointed presheaf $\t_{\st}\nf{\left[z\right]}{d^{>n}=*}$
is a stable $\left[z\right]$-simplex with precisely $n$-many codimension
1 faces and the pointed presheaf $S^{\left(z-n\right)}\left[n\right]$
is a stable $\left[z\right]$-simplex which is, in essence, an $\left[n\right]$-simplex
worth of $\left(z-n\right)$-spheres.
\end{rem}

It is immediate, as observed in \cite{ChenKrizPultr}, that $\sf{LocFin}\left(\t_{\st}\right)$
is a coreflective subcategory of $\whp{\t_{\st}}$ with the right
adjoint functor picking out, for each presheaf, the sub-object of
cells of the form $\t_{\st}\nf{\left[z\right]}{d^{>n}=*}$ . The same
argument provides that $\sf{LocSph}\left(\t_{\st}\right)$ is coreflective
in $\sf{LocFin}\left(\t_{\st}\right)$.

\[
\xyR{0pc}\xyC{1.5pc}\xymatrix{\reallywidehat{\t_{\st}}_{\bullet}\ar[rr]^{\bot} &  & \ar@/_{1pc}/@{^{(}->}[ll]\mathsf{LocFin}\left(\t_{\st}\right)\ar[rr]^{\bot} &  & \ar@/_{1pc}/)@{^{(}->}[ll]\mathsf{LocSph}\left(\t_{\st}\right)}
\]

Now, as we will show, every presheaf $Z:\t_{\st}^{\op}\longrightarrow\S_{\bullet}$
which is $\rho_{\bullet!}\circ\left(\t_{\sf X}\Phi\right)_{\bullet}^{*}$
of something satisfies an important and interesting property - it
is a colimit of pointed presheaves $S^{\left(z-n\right)}\left[n\right]$.

Every object of the category $\oplaxlim\sf X$ is canonically the
colimit of objects of the form
\[
\IN_{\sf X}\Phi\left(\t_{\sf{coll}}\left(\left[n\right],-k\right)\right)
\]
as any simplex in any index gives rise to such. It then follows that,
since $\rho_{\bullet!}\circ\left(\IN_{\sf X}\Phi\right)_{\bullet}^{*}$
preserves colimits, every $Z:\t_{\st}^{\op}\longrightarrow\S_{\bullet}$
which is $\rho_{\bullet!}\circ\left(\t_{\sf X}\Phi\right)_{\bullet}^{*}$
of something is a colimit of objects of the form 
\[
\rho_{\bullet!}\circ\left(\IN_{\sf X}\Phi\right)_{\bullet}^{*}\circ\IN_{\sf X}\Phi\left(\t_{\sf{coll}}\left(\left[n\right],-k\right)\right)
\]
As it turns out however, these are precisely the stable simplicial
sets $S^{\left(z-n\right)_{\left[n\right]}}$.
\begin{lem}
For all $n\in\text{\ensuremath{\N}}$ and $k\in\N$, 
\[
\rho_{\bullet!}\circ\left(\t_{\sf X}\Phi\right)_{\bullet}^{*}\circ\left(\t_{\sf X}\Phi\right)\left(\t_{\sf{coll}}\left(\left[n\right],-k\right)\right)=S^{-k}\left[n\right]
\]
\end{lem}

\begin{proof}
Consider the object $\left(\t_{\sf X}\Phi\right)_{\bullet}^{*}\circ\left(\t_{\sf X}\Phi\right)\left(\t_{\sf{coll}}\left(\left[n\right],-k\right)\right)$
in the sequential notation.
\[
\left(\t_{\sf X}\Phi\right)_{\bullet}^{*}\circ\left(\t_{\sf X}\Phi\right)\left(\t_{\sf{coll}}\left(\left[n\right],-k\right)\right)_{-\ell}=\begin{cases}
\bullet & \ell<k\\
\Sigma_{K}^{\left(\ell-k\right)}\t\left[n\right]_{+} & \ell\geq k
\end{cases}
\]
\end{proof}
We've thus the promised (in the introduction) composition of adjunctions
\[
\xyR{0pc}\xyC{1.5pc}\xymatrix{\reallywidehat{\t_{\st}}_{\bullet}\ar[rr]^{\bot} &  & \ar@/_{1pc}/@{^{(}->}[ll]\mathsf{LocFin}\left(\t_{\st}\right)\ar[rr]^{\bot} &  & \ar@/_{1pc}/@{^{(}->}[ll]\mathsf{LocSph}\left(\t_{\st}\right)\ar@/^{1pc}/[rr]^{\sf{KPs}} &  & \oplaxlim\sf X\ar[ll]_{\sf{KSp}}^{\bot}}
\]

\begin{example}
Now, Kan claims that the composite adjunction 
\[
\xyR{0pc}\xyC{1.5pc}\xymatrix{\mathsf{LocFin}\left(\t_{\st}\right)\ar[rr]^{\bot} &  & \ar@/_{1pc}/@{^{(}->}[ll]\mathsf{LocSph}\left(\t_{\st}\right)\ar@/^{1pc}/[rr]^{\sf{KPs}} &  & \oplaxlim\sf X\ar[ll]_{\sf{KSp}}^{\bot}}
\]
is a reflective subcategory, by asserting the co-unit to be a natural
isomorphism, whence asserting the right adjoint functor to be full-and-faithful.
It's not hard to show however that $\t_{\st}\nf{\left[z\right]}{d^{>n}=*}$
is in $\sf{LocFin}\left(\t_{\st}\right)$ but \emph{not} in $\sf{LocSph}\left(\t_{\st}\right)$,
so the component of co-unit of the adjunction at the object $\t_{\st}\nf{\left[z\right]}{d^{>n}=*}$
is \emph{not }the identity.
\end{example}

The right hand adjunction however is, by Kan's observation, a reflective
subcategory.
\begin{lem}
The co-unit of the adjunction 
\[
\xyR{0pc}\xyC{1.5pc}\xymatrix{\mathsf{LocSph}\left(\t_{\st}\right)\ar@/_{1pc}/[rr]_{\sf{KPs}} &  & \oplaxlim\sf X\ar[ll]_{\sf{KSp}}^{\bot}}
\]
 is a natural isomorphism, whence $\sf{LocSph}\left(\t_{\st}\right)\longrightarrow\oplaxcolim\sf X$
is full-and-faithful.
\end{lem}

\begin{proof}
This can be shown by direct computation. Suppose $Z:\t_{\st}^{\op}\longrightarrow\S_{\bullet}$
to be of $\sf{LocSph}\left(\t_{\st}\right)$. Then the pointed simplicial
sets $\sf{KPs}\left(Z\right)_{-i}$ may be described as follows
\[
\vcenter{\vbox{\xyR{0pc}\xyC{1.5pc}\xymatrix{\sf{KPs}\left(Z\right)_{-i}:\t^{\op}\ar[r] & \S_{\bullet}\\
\left[n\right]\ar@{|->}[r] & \whp{\t_{\st}}\left(S^{-i}\left[n\right],Z\right)
}
}}
\]
so it is clear enough that the co-unit is a natural isomorphism.
\end{proof}
\begin{rem}
For the above to really be as clear as we assert, it is important
to note that there are quotient relationships between the spheres
$S^{-k}\left[n\right]$. For example $S^{1}\left[1\right]\longrightarrow S^{2}\left[0\right]$
is a quotient.
\end{rem}

As it turns out however, in the next section, we'll show that this
subtle error of Kan's, repeated in a slightly different guise in \cite{ChenKrizPultr},
ends up irrelevant once we go from categories to model categories.

\section{Homotopy-coherent models of...?}

In 1973 \cite{Brown} synthesized Kan's treatment of spectra with
the then new notion of model categories (really with the notion of
categories of fibrant objects also developed in \cite{Brown}).
\begin{defn}
( \cite{Brown} ) Let $z\in\Z$, $n\in\N$ and $0\leq i\leq n$ be
given and set 
\[
\underset{\sf B}{\Lambda}^{i}\nf{\left[z\right]}{d^{>n}=*}\hookrightarrow\t_{\st}\nf{\left[z\right]}{d^{>n}=*}
\]
to be the union, in $\t_{\st}\nf{\left[z\right]}{d^{>m}=*}$, of the
images of the set of simplices 
\[
\left\{ d^{k}:\t_{\st}\left[z-1\right]\longrightarrow\t_{\st}\nf{\left[z\right]}{d^{>n}=*}\right\} _{k\in\N-\left\{ i\right\} }
\]
Similarly, let 
\[
\partial\t_{\st}\nf{\left[z\right]}{d^{>n}=*}\hookrightarrow\t_{\st}\nf{\left[z\right]}{d^{>n}=*}
\]
 be the union of the images of the set of simplices
\[
\left\{ d^{k}:\t_{\st}\left[z-1\right]\longrightarrow\t_{\st}\nf{\left[z\right]}{d^{>n}=*}\right\} _{k\in\N}
\]
\end{defn}

\begin{thm}
\label{thm:Browns-model-strucutre}(\cite{Brown}) The sets 
\[
\Lambda_{\sf{Brown}}=\left\{ \sd{\underset{\sf B}{\Lambda}^{i}\nf{\left[z\right]}{d^{>n}=*}\hookrightarrow\t_{\st}\nf{\left[z\right]}{d^{>n}=*}}z\in\Z,n\in\N,0\leq i\leq n\right\} 
\]
and 
\[
\partial_{\sf{Brown}}=\left\{ \sd{\partial\t_{\st}\nf{\left[z\right]}{d^{>n}=*}\hookrightarrow\t_{\st}\nf{\left[z\right]}{d^{>n}=*}}z\in\Z,n\in\N\right\} 
\]
comprise a set of generating acyclic cofibrations and generating co-fibration
for a model structure on $\sf{LocFin}\left(\t_{\st}\right)$. More:
\begin{itemize}
\item This model structure is the left transfer of a pointed Cisinski model
structure on $\whp{\t_{\st}}$; and
\item The adjunction 
\[
\xyR{0pc}\xyC{1.5pc}\xymatrix{\mathsf{LocFin}\left(\t_{\st}\right)\ar@/_{1pc}/[rr]_{\sf{KPs}} &  & \oplaxlim\sf X\ar[ll]_{\sf{KSp}}^{\bot}}
\]
 where the right hand side is given the stable model structure is
a Quillen equivalence.
\end{itemize}
\end{thm}

\begin{proof}
See Theorem 5 and Proposition 7 of \cite{Brown}.
\end{proof}
This model structure can further be left transported to $\sf{LocSph}\left(\t_{\st}\right)$.
\begin{defn}
Let $z\in\Z$, $n\in\N$, and $0\leq i\leq n$ be given and set 
\[
\underset{\sf S}{\Lambda}^{i}\nf{\left[z\right]}{d^{>n}=*}\hookrightarrow S^{\left(z-n\right)}\left[n\right]
\]
to be defined by way of the epi-mono factorization 
\[
\xyR{3pc}\xyC{10pc}\xymatrix{\underset{\sf B}{\Lambda^{i}}\nf{\left[z\right]}{d^{>n}=*}\ar[r]\ar[d]_{\sf{epi}} & \t_{\st}\nf{\left[z\right]}{d^{>n}=*}\ar[d]\\
\underset{\sf S}{\Lambda}^{i}\nf{\left[z\right]}{d^{>n}=*}\ar[r]^{\sf{mono}} & S^{\left(z-n\right)}\left[n\right]
}
\]
Let $\text{}\Lambda_{\sf{Spherical}}$ be the following set of monomorphisms.
\[
\Lambda_{\sf{Spherical}}=\left\{ \sd{\underset{\sf S}{\Lambda}^{i}\nf{\left[z\right]}{d^{>n}=*}\hookrightarrow S^{\left(z-n\right)}\left[n\right]}z\in\Z,n\in\N,0\leq i\leq n\right\} 
\]
Similarly we define 
\[
\partial S^{\left(z-n\right)\left[n\right]}\hookrightarrow S^{\left(z-n\right)}\left[n\right]
\]
 by the epi-mono factorization 
\[
\xyR{3pc}\xyC{10pc}\xymatrix{\partial\t_{\st}\nf{\left[z\right]}{d^{>n}=*}\ar[r]\ar[d]_{\sf{epi}} & \t_{\st}\nf{\left[z\right]}{d^{>n}=*}\ar[d]\\
\partial S^{\left(z-n\right)}\left[n\right]\ar[r]^{\sf{mono}} & S^{\left(z-n\right)}\left[n\right]
}
\]
and we define the set $\partial_{\sf{Spherical}}$ as follows.
\[
\partial_{\sf{\sf{Spherical}}}=\left\{ \sd{\partial S^{\left(z-n\right)}\left[n\right]\hookrightarrow S^{\left(z-n\right)}\left[n\right]}z\in\Z,n\in\N\right\} 
\]
\end{defn}

\begin{cor}
The sets $\Lambda_{\sf{Spherical}}$ and $\partial_{\sf{Spherical}}$
comprise a set of generating acyclic cofibrations and generating cofibrations
for the left transport of the model category on $\sf{LocFin}\left(\t_{\st}\right)$
of Theorem \ref{thm:Browns-model-strucutre} along the adjunction
\[
\xyR{0pc}\xyC{3pc}\xymatrix{\sf{LocFin}\left(\t_{\st}\right)\ar@/_{1pc}/[rr] &  & \mathscr{\mathbf{\sf{LocSph}\left(\t_{\st}\right)}}\ar[ll]_{\sf{}}^{\bot}}
\]
Moreover the the Quillen adjunction adjunction is a Quillen equivalence.
\end{cor}

\begin{proof}
To prove that the model structure generated by $\Lambda_{\sf{Spherical}}$
and $\partial_{\sf{Spherical}}$ coincides with the left transferred
model structure is to prove, for $f:X\longrightarrow Y$ in $\sf{LocSph}\left(\t_{\st}\right)$,
that $f\in\sf{Cell}\left(\Lambda_{\sf{Spherical}}\right)$ if and
only if $f\in\sf{Cell}\left(\Lambda_{\sf{Brown}}\right)$ and likewise
for $\partial_{\sf{Spherical}}$ and $\partial_{\sf{Brown}}$.

Since the inclusion functor is left adjoint, and both $\Lambda_{\sf{Spherical}}$
and $\partial_{\sf{Spherical}}$ lie in $\sf{Cell}\left(\Lambda_{\sf{Brown}}\right)$
and $\sf{Cell}\left(\partial_{\sf{Brown}}\right)$ respectively, the
``if direction'' is clear. Conversely, since membership in $\sf{LocSph}\left(\t_{\st}\right)$
is defined simplex-wise it follows by a simple factorization argument
that any $f$ from $\text{\ensuremath{\sf{LocSph}\left(\t_{\st}\right)}}$
which lies in $\sf{Cell}\left(\Lambda_{\sf{Brown}}\right)$ and $\sf{Cell}\left(\partial_{\sf{Brown}}\right)$
actually lies in $\Cell\left(\Lambda_{\sf{Spherical}}\right)$ or
$\sf{Cell}\left(\partial_{\sf{Spherical}}\right)$ respectively. Indeed,
it's clear that we have factorizations as below.
\[
\xyR{3pc}\xyC{1pc}\xymatrix{ & \underset{\sf B}{\Lambda^{i}}\t_{\st}\nf{\left[z\right]}{d^{>n}=*},\partial\t_{\st}\nf{\left[z\right]}{d^{>n}=*}\ar[rr]\ar[dd]\ar[dl] &  & \ar[dl]\t_{\st}\nf{\left[z\right]}{d^{>n}=*}\ar[dd]\\
\underset{\sf S}{\Lambda}^{i}S^{\left(z-k\right)}\left[k\right],\partial S^{\left(z-k\right)}\left[k\right]\ar@{-->}[dr]\ar[rr] &  & \ar@{-->}[dr]S^{\left(z-k\right)}\left[k\right]\\
 & X\ar[rr] &  & Y
}
\]
That the Quillen adjunction is in fact a Quillen equivalence follows
from the determination of weak-equivalences by their inducement of
isomorphisms on stable homotopy groups, which are determined as homotopy
classed of maps from objects in $\sf{LocSph}\left(\t_{\st}\right)$
- consider that
\[
\pi_{z}\left(\_\right)=\left[S^{z}\left[0\right],\_\right]_{\bullet}
\]
\end{proof}
Having found model categories for spectra which admit interpretation
as homotopy coherent models of a limit theory, both the model structure
on $\sf{LocSph}\left(\t_{\st}\right)$ and $\sf{LocFin}\left(\t_{\st}\right)$
permit this, we are wont to ask - what is that limit theory a description
of?

\section{A Globular Perspective on Kan's Suspension Functor}

Kan's model structure is intuitively interpreted as describing homotopy
coherent groupoids. However, in this interpretation one usually thinks
of every simplex merely as composition datum - unless one is accustomed
to thinks in terms of higher categorical nerves and oriented simplices.
As such it may not be clear how to relate the notions of composition
embodied by lifting against 
\[
\underset{\sf S}{\Lambda}^{i}\nf{\left[z\right]}{d^{>n}=*}\longrightarrow\t_{\st}\nf{\left[z\right]}{d^{>n}=*}
\]
for different values of $z$ and $n$. We contend that making the
globular content more explicit makes things a bit more clear.

Since the simplex category $\t$ is a full subcategory of the category
$\Cat$ of small categories we may think of Kan's suspension as taking
each simplex, thought of as a $1$-category, and returning the $2$-category
with a single object and that original simplex as its unique $\Hom$-category.
In this view, Kan's suspension is but the simplicial avatar of the
endo-functor 
\[
\Sigma_{\omega}:\StrCat_{\bullet}\longrightarrow\StrCat_{\bullet}
\]
which takes an $\omega$-category and returns the $\omega$-category
with a single object and that original $\omega$-category as its unique
$\Hom$-$\omega$-category.
\begin{rem}
Leaving simplicial sets and going to globular ones is not of course
the only way to see this. Indeed we could, and probably should, think
of Kan's suspension $\Sigma_{K}$ as an endo-functor on strict $\omega$-categories
as complicial sets, which sends an $\left[n\right]$-simplex to a
thick $\left[n+1\right]$-simplex $\Sigma_{K}\left[n+1\right]_{+}$.
Had this author known of that material prior to developing this work,
he might well have gone in that direction instead.
\end{rem}

\begin{rem}
Importantly, this perspective is compatible with the interpretation
of Kan's model structure as providing a description of homotopy coherent
$\omega$-groupoids, as a composition datum for $n$-many $1$-cells
passes to composition datum for $n$-many $2$-cells.

Let $\Omega_{\omega}$ denote the functor right adjoint to $\Sigma_{\omega}$.
The diagram which begat sequential spectra should then be understood
an avatar of the diagram
\[
\cdots\rightarrow\StrCat_{\bullet}\xrightarrow{\Omega_{\omega}}\StrCat_{\bullet}\xrightarrow{\Omega_{\omega}}\StrCat_{\bullet}
\]
and the diagram 
\[
\cdots\rightarrow\whp{\t}\xrightarrow{K_{\bullet}^{*}}\whp{\t}\xrightarrow{K_{\bullet}^{*}}\whp{\t}
\]
should be considered as but the pointing of an avatar of the diagram
\[
\cdots\rightarrow\StrCat\xrightarrow{S^{*}}\StrCat\xrightarrow{S^{*}}\StrCat
\]
where $S^{*}$ is the functor right adjoint to the functor $S:\StrCat\longrightarrow\StrCat$
which sends an $\omega$-category $X$ to the $\omega$-category with
two objects $0$ and $1$ with $\Hom\left(0,1\right)=X$ and all other
$\Hom$-categories empty.\footnote{For those already familiar with Berger's wreath notation, we might
simply say the $\omega$-category $\left[1\right];\left(X\right)$.}
\end{rem}

\section{$\protect\Z$-categories}

\subsection{Strict $\omega$-categories and Berger's Cellular Nerve}
\begin{defn}
Let $\bG$ denote the\textbf{ globe category}, the category 
\[
\left\langle \left.\vcenter{\vbox{\xyR{3pc}\xyC{3pc}\xymatrix{\overline{0}\ar@/^{.75pc}/[r]^{s}\ar@/_{.75pc}/[r]_{t} & \overline{1}\ar@/^{.75pc}/[r]^{s}\ar@/_{.75pc}/[r]_{t} & \cdots}
}}\right|\begin{array}{c}
s\circ t=s\circ s\\
t\circ t=t\circ s
\end{array}\right\rangle 
\]
Let the category of \textbf{globular sets }be the presheaf category
$\wh{\bG}$ and let $T:\wh{\bG}\longrightarrow\wh{\bG}$ be the monad
whose algebras are strict-$\omega$-categories.
\end{defn}

It was first proved in \cite{Berger1} that strict $\omega$-categories
are a reflective, indeed orthogonal, subcategory of the category of
cellular sets $\wh{\Theta}$. We repeat the treatment of the category
$\Theta$ developed first in \cite{Berger2}, however our presentation
borrows more from \cite{CisinskiMaltsiniotis}.

\subsubsection{Segal's category $\Gamma$}

Segal's category $\Gamma$ is a skeleton - in the sense of a small
category of chosen representatives for each isomorphism class - for
the opposite category of the category of finite pointed sets.
\begin{defn}
Let $\Gamma$, \textbf{Segal's gamma category}, be the category specified
thus: let
\[
\Ob\left(\Gamma\right)=\left\{ \sd{\left\langle k\right\rangle =\left\{ 1,\dots,k\right\} }k\geq1\right\} \cup\left\{ \left\langle 0\right\rangle =\varnothing\right\} ,
\]
 and let $\Gamma\left(\left\langle n\right\rangle ,\left\langle m\right\rangle \right)$
be defined by the expression
\[
\Gamma\left(\left\langle n\right\rangle ,\left\langle m\right\rangle \right)=\left\{ \sd{\varphi:\left\langle n\right\rangle \longrightarrow\mathsf{Sub}_{\S}\left(\left\langle m\right\rangle \right)}\forall i\neq j\in\left\langle m\right\rangle ,\ \varphi\left(i\right)\cap\varphi\left(j\right)=\varnothing\right\} 
\]
where, for any category $A$ and object $a$ thereof, $\mathsf{Sub}_{A}\left(a\right)$
is the category of subobjects of $a$. Define the composition of morphisms
in $\Gamma$ by setting
\[
\left\langle \ell\right\rangle \overset{\varphi}{\longrightarrow}\left\langle m\right\rangle \overset{\sigma}{\longrightarrow}\left\langle n\right\rangle 
\]
to be the map
\[
\sigma\circ\varphi:i\longmapsto\bigcup_{j\in\varphi\left(i\right)}\sigma\left(j\right).
\]
\end{defn}

\begin{rem}
The equivalence of categories between $\Gamma$ and $\FinSet_{\bullet}^{\op}$
is a $0$-truncated analogue of the Grothendieck construction - a
map of finite pointed sets is replaced with the data of the fibres
it parameterizes.
\end{rem}

\subsubsection{\label{subsec:The-Categorical-wreath}Berger's Categorical wreath
product}
\begin{defn}
Let $A$ and $B$ be small categories. Given a functor $G:B\longrightarrow\Gamma$,
we define $B\int_{G}A=B\int A$ (with the second notation suppressing
the functor $G$ when the meaning is clear) be the category the objects
of which are 
\[
b;\left(a_{1},\dots,a_{m}\right)
\]
where $b$ is an object of $B$, $G\left(b\right)=\left\langle m\right\rangle $,
and $\left(a_{1},\dots,a_{m}\right)$ describes a function $G\left(b\right)\longrightarrow\Ob\left(A\right)$.
The morphisms of $B\int A$, denoted 
\[
g;\mathbf{f}:b;\left(a_{i}\right)_{i\in G\left(b\right)}\longrightarrow d;\left(c_{i}\right)_{i\in G\left(d\right)}
\]
 are constituted of a morphism
\[
g:b\longrightarrow d
\]
 of $B$ and a morphism of $\wh A$, 
\[
\mathbf{f}=\left(\left(f_{ji}:a_{i}\rightarrow c_{j}\right)_{j\in G\left(g\right)\left(i\right)}\right)_{i\in G\left(b\right)}:\prod_{i\in G\left(b\right)}\left(A^{a_{i}}\longrightarrow\prod_{j\in G\left(g\right)\left(i\right)}A^{c_{j}}\right)
\]
 The composition
\[
b;\left(a_{i}\right)_{i\in G\left(b\right)}\overset{g;\mathbf{f}}{\longrightarrow}d;\left(c_{i}\right)_{i\in G\left(d\right)}\overset{r;\mathbf{q}}{\longrightarrow}\ell;\left(k_{i}\right)_{i\in G\left(\ell\right)}
\]
is denoted $r\circ g;\mathbf{q}\circ\mathbf{f}$ where the meaning
of $r\circ g$ is clear and
\[
\mathbf{q}\circ\mathbf{f}=\left(\left(q_{jk}\circ f_{ki}\right)_{j\in G\left(r\circ g\right)\left(i\right)}\right)_{i\in G\left(b\right)}
\]
with the varying values for $k\in G\left(d\right)$ in the expression
being those unique $k$ in $G\left(g\right)\left(i\right)$ such that
$j\in G\left(r\right)\left(k\right)$.
\end{defn}

\begin{rem}
We'll also use the more explicit notation $g;\left(f_{1},f_{2},\dots,f_{n}\right)$
where doing so simplifies the exposition.
\end{rem}

\begin{example}
\label{exa:from-simplex-category-to-segal's-gamma}

Define the functor $F:\t\longrightarrow\Gamma$ by setting 
\[
F\left(\left[n\right]\right)=\left\langle n\right\rangle 
\]
 and setting for each $\varphi:\left[m\right]\longrightarrow\left[n\right]$,
\[
F\left(\varphi\right):\left\langle m\right\rangle \longrightarrow\left\langle n\right\rangle 
\]
to be the function 
\[
F\left(\varphi\right):\left\langle m\right\rangle \longrightarrow\mathsf{Sub}_{\S}\left(\left\langle n\right\rangle \right)
\]
 given thus:
\[
F\left(\varphi\right)\left(i\right)=\left\{ \sd j\varphi\left(i-1\right)<j\leq\varphi\left(i\right)\right\} 
\]
\end{example}

\subsubsection{The Categories $\Theta$ and $\Theta_{n}$}
\begin{defn}
Let 
\[
\gamma:\t\longrightarrow\t\int\t
\]
 be the obvious functor extending the assignment $\gamma\left(\left[n\right]\right)=\left[n\right];\left(\left[0\right],\dots,\left[0\right]\right)$.
We define the categories $\Theta_{n}$ to be the $n$-fold wreath
product of $\t$ with itself, i.e. we set
\[
\Theta_{n}=\underbrace{\t\int\left(\cdots\int\t\right)}_{n-\mathrm{times}}
\]
 We set $\Theta$ to be the conical colimit\footnote{An elementary argument about the projective-canonical and Reedy-canonical
model structures on $\CAT\left(\N,\Cat\right)$ provides that this
conical colimit enjoys the universal property of the pseudo-colimit.}
\[
\underset{\longrightarrow}{\lim}\left\{ \t\overset{\gamma}{\longrightarrow}\t\int\t\overset{\id\int\gamma}{\longrightarrow}\cdots\right\} 
\]
\end{defn}

\begin{rem}
It should also be noted that 
\[
\Theta\iso\t\int\Theta\iso\t\int\t\int\Theta\iso\cdots
\]
 so we may denote cells - where cells are the objects of $\Theta$
- in many compatible ways. For example for any $T$ a cell of $\Theta$
we may also write $T=\left[n\right];\left(T_{1},\dots,T_{n}\right)$
for some unique $n\in\N$ and unique $T_{1},\dots,T_{n}$ cells of
$\Theta$.
\end{rem}

\subsubsection{Reedy Structures on $\Theta_{n}$ and $\Theta$}

It is important to note that the categories $\Theta$ and $\Theta_{n}$
are Reedy categories. This fact was first proved in \cite{Berger1}
and again by more general means in \cite{BergnerRezk}.
\begin{defn}
\label{def:Reedy-on-theta-Berger}Inductively, we define 
\[
\lambda_{\Theta_{n+1}}:\Ob\left(\Theta_{n+1}\right)\longrightarrow\N
\]
by setting 
\[
\lambda_{\Theta_{n+1}}\left(\left[m\right];\left(T_{1},\dots,T_{m}\right)\right)=m+\sum_{i=1}^{m}\lambda_{\Theta_{n}}\left(T_{i}\right)
\]
 with $\lambda_{\Theta_{1}}=\lambda_{\t}$. The direct (wide) subcategories
$\Theta_{n}^{+},\Theta^{+}$ are subtended by the morphisms $f:S\longrightarrow T$
which pass to monomorphisms of presheaves under $\Yo$. Similarly
the inverse (wide) subcategories $\Theta_{n}^{-},\Theta^{-}$ are
subtended by the morphisms of $\Theta_{n}$ (resp. $\Theta$) which
pass to epimorphisms under $\Yo$.
\end{defn}

\begin{lem}
(Lemma 2.4(a) of \cite{Berger1}) The datum $\left(\Theta_{n},\Theta_{n}^{+},\Theta_{n}^{-},\lambda_{n}\right)$
(resp. $\left(\Theta,\Theta^{+},\Theta^{-},\lambda\right)$) comprises
a Reedy structure.
\end{lem}

\subsubsection{Inner horns and $\omega$-categories, outer horns and $\omega$-groupoids}
\begin{defn}
\label{def:horns-a-la-Berger}Recall that a codimension 1 face $d:\left[m\right]\longrightarrow\left[m+1\right]$
of is \textbf{inner} if it preserves top and bottom elements and \textbf{outer}
if it does not. Inductively, we may then define \textbf{inner faces
of $\Theta_{n}$} as those maps 
\[
f;\left(\mathbf{g}\right):\left[m\right];\left(S_{1},\dots,S_{m}\right)\longrightarrow\left[n\right];\left(T_{1},\dots,T_{n}\right)
\]
 for which:
\begin{itemize}
\item $f$ is inner for $\t$;
\item and each component $g_{ji}:S_{i}\longrightarrow T_{j}$ of $\mathbf{g}$
is inner for $\Theta_{n-1}$.
\end{itemize}
Given a face $\kappa:S\longrightarrow T$, let $\Lambda^{\kappa}\left[T\right]\subset\Theta\left[T\right]$
be the colimit 
\[
\underset{\Theta^{+}\downarrow T-\left\{ \id,\kappa\right\} }{\colim}\Theta\left[S\right].
\]
 Then the canonical inclusion $\Lambda^{\kappa}\left[T\right]\longrightarrow\Theta\left[T\right]$
is to be called an an \textbf{inner (resp. outer) horn of $\Theta\left[T\right]$}
if $\kappa$ is an inner (resp. outer) face of $T$. Let 
\[
\Lambda_{\mathsf{InnerBerger}}=\left\{ \sd{\Lambda^{\kappa}\left[T\right]\longrightarrow\Theta\left[T\right]}\kappa\in\mathsf{InnerFace}\left(T\right),\ T\in\Ob\left(\Theta\right)\right\} 
\]
 and let 
\[
\Lambda_{\mathsf{Berger}}=\left\{ \sd{\Lambda^{\kappa}\left[T\right]\longrightarrow\Theta\left[T\right]}\kappa\in\mathsf{Face}\left(T\right),\ T\in\Ob\left(\Theta\right)\right\} .
\]
\end{defn}

\begin{prop}
\label{prop:Horn-filling-strict-shit}(\cite{Berger1}) An $n$-cellular
set (resp. cellular set) is the nerve of a strict $n$-category (resp.
$\omega$-category) if and only if it is right orthogonal to the regulus
$\Lambda_{\mathsf{InnerBerger}}$. Similarly, an $n$-cellular set
(resp. cellular set) is the nerve of a strict $n$-groupoid (resp.
strict $\omega$-groupoid) precisely if it is right orthogonal to
the regulus $\Lambda_{\sf{Berger}}$.
\end{prop}

\subsection{$\protect\Z$-categories}
\begin{defn}
Denote by $S$ the functor
\[
\xyR{0pc}\xyC{5pc}\xymatrix{S:\bG\ar[r] & \bG\\
\overline{n}\ar@{|->}[r] & \overline{n+1}
}
\]
This functor induces $S^{*}:\wh{\bG}\longrightarrow\wh{\bG}$ which
``pulls down'' $n$-cells as $\left(n-1\right)$-cells and forgets
the preexisting $0$-cells. This operation preserves $T$-algebras;
$S^{*}$ induces a functor $S^{*}:\wh{\bG}^{T}\longrightarrow\wh{\bG}^{T}$.
\end{defn}

What then we ask is an object of the limit
\[
\llim\left\{ \cdots\rightarrow\StrCat\xrightarrow{S^{*}}\StrCat\xrightarrow{S^{*}}\StrCat\right\} 
\]
which we consider as the limit of the diagram $\mathbf{S}^{*}:\Z_{\leq0}\longrightarrow\CAT$
which sends each $-n$ to $\wh{\bG}^{T}$ and each $-\left(n+1\right)\longrightarrow-n$
to $S^{*}$.
\[
\cdots\xrightarrow{S^{*}}\wh{\bG}^{T}\xrightarrow{S^{*}}\wh{\bG}^{T}
\]
Now, Eilenberg-Moore objects for a monad $\left(T,\mu,\eta\right)$
in a $2$-category $A$ (here $\CAT$) are but lax-limits over the
diagram $\left\lceil T\right\rceil :\bfTwo\longrightarrow A$, which
classifies the underlying endo-functor. Then, since the taking of
limits commutes with the taking of limits, we find that there exists
some monad $\llim T$ for which we have an isomorphism
\[
\llim\left(\wh{\bG}^{T}\right)\liso\llim\wh{\bG}^{\llim T}
\]
More, we have an isomorphism
\[
\llim\wh{\bG}\liso\wh{\bG_{\Z}}
\]
with $\bG_{\Z}$ defined as follows.
\begin{defn}
Let $\bG_{\Z}$ be the category 
\[
\left\langle \left.\vcenter{\vbox{\xyR{3pc}\xyC{3pc}\xymatrix{\cdots\ar@/^{.75pc}/[r]^{s}\ar@/_{.75pc}/[r]_{t} & \overline{-1}\ar@/^{.75pc}/[r]^{s}\ar@/_{.75pc}/[r]_{t} & \overline{0}\ar@/^{.75pc}/[r]^{s}\ar@/_{.75pc}/[r]_{t} & \overline{1}\ar@/^{.75pc}/[r]^{s}\ar@/_{.75pc}/[r]_{t} & \cdots}
}}\right|\begin{array}{c}
s\circ t=s\circ s\\
t\circ t=t\circ s
\end{array}\right\rangle 
\]
\end{defn}

Using the generalized spectrification of Theorem \ref{thm: generalized-spectrification-strict-limit}
we observe that the monad $\llim T$, which we'll denote now by $T_{\Z}$,
is but $\mathsf{sp}\circ\oplaxlim T$ where $\sf{sp}$ is the left
adjoint to the inclusion 
\[
\llim\left\{ \cdots\xrightarrow{S^{*}}\wh{\bG}^{T}\xrightarrow{S^{*}}\wh{\bG}^{T}\right\} \longrightarrow\oplaxlim\left\{ \cdots\xrightarrow{S^{*}}\wh{\bG}^{T}\xrightarrow{S^{*}}\wh{\bG}^{T}\right\} 
\]

\begin{defn}
Let the category $\wh{\bG_{\Z}}^{T_{\Z}}$ be called the $1$-category
of \textbf{strict-$\Z$-categories}.
\end{defn}

\part{Limits in the $2$-category of Categories with Arities}

In Proposition \ref{prop:Horn-filling-strict-shit} above we recalled
Berger's characterization of strict $\omega$-categories and strict
$\omega$-groupoids as cellular sets satisfying horn-filling conditions.
In this second part we develop theory enough to allow us to extend
this result to characterize $\Z$-categories as an orthogonal subcategory
of a category of $\Z$-cellular sets. We do this by developing the
theory of limits in Berger, Mellies and Weber's $2$-category with
arities and Bourke and Garner's synthesis of that work with Day's
density presentations.

\section{Basic definitions: Variations on the $2$-category $\protect\CwA$ }

In \cite{BergerMelliesWeber} we find the following $2$-category
defined.
\begin{defn}
\label{def:BMW-definition-of-CwA}Let $\CwA$, the $2$-category of
\textbf{categories with arities,} be the $2$-category with:
\begin{description}
\item [{CWA0}] $0$-cells are locally small categories $\sfE$ together
with a dense, small, full subcategory $\sfA\overset{I}{\hookrightarrow}\sfE$,
denoted $\left(\sfA,\sfE,I\right)$;
\item [{CWA1}] $1$-cells $\left(\sfA,\sfE,I\right)\longrightarrow\left(\sfB,\sfF,J\right)$,
so called \textbf{arities-respecting} \textbf{functors}, are functors
$K:\sfE\longrightarrow\sfF$ such the the squares 
\[
\xyR{3pc}\xyC{3pc}\xymatrix{\sfE\ar[r]^{\sN_{I}}\ar[d]_{K} & \wh{\sfA}\ar[d]^{\overline{K}}\\
\sf{\sfF}\ar[r]^{\sN_{J}} & \wh{\sfB}
}
\]
pseudo-commute, where $\overline{K}$ is (and will remain by convention
- at least most of the time) the following left Kan extension
\[
\xyR{.5pc}\xyC{3pc}\xymatrix{\sfA\ar[r]^{I}\ar[dr]_{I} & \sfE\ar[r]^{\sN_{I}} & \wh{\sfA}\ar[ddddd]^{\overline{K}=\mathsf{Lan}_{\Yo_{\sfA}}\left(\sN_{J}\circ K\circ I\right)}\\
 & \sfE\ar[ddd]_{K}\\
\\
\\
 & \sf{\sfF}\ar[dr]_{\sN_{J}}\\
 &  & \wh{\sfB}
}
\]
\item [{CWA2}] $2$-cells $\alpha:K\Rightarrow L:\left(\sfA,\sfE\right)\rightarrow\left(\sfB,\sfF\right)$
are (arbitrary) natural transformations $\alpha:K\Rightarrow L:\sfE\rightarrow\sfF$.
\end{description}
\end{defn}

Perhaps unsurprisingly we can, in a sense, switch the roles of $K$
and $\overline{K}$ in the definition of the $1$ and $2$-cells of
$\CwA$ - from $K$ as datum and $\overline{K}$ as part of a condition
to the reverse.
\begin{lem}
Given categories with arities $\left(\sfA,\sfE,I\right)$ and $\left(\sfB,\sfF,J\right)$,
and a co-continuous functor $\overline{K}:\wh{\sfA}\longrightarrow\wh{\sfB}$,
if the square
\[
\xyR{3pc}\xyC{3pc}\xymatrix{\sfE\ar[r]^{\sN_{I}}\ar[d]_{G} & \wh{\sfA}\ar[d]^{\overline{K}}\\
\sf{\sfF}\ar[r]_{\sN_{J}} & \wh{\sfB}
}
\]
pseudo-commutes, then $\overline{K}=\Lan_{\Yo}\left(\sN_{J}\circ G\circ I\right)$
.
\end{lem}

\begin{proof}
Since $\overline{K}$ is co-continuous and $\wh{\sfA}$ is a presheaf
category, whence it is the free completion of $\sfA$ under colimits,
it follows that $\overline{K}=\Lan_{\Yo}\left(K\right)$ for some
$K:\sfA\longrightarrow\wh{\sfB}$. More, since 
\[
\Lan_{\Yo}\left(K\right)=\Lan_{\sN_{I}}\left(\Lan_{I}\left(K\right)\right)
\]
and the square pseudo-commutes by hypothesis we have 
\[
\sN_{J}\circ G=\Lan_{\sN_{I}}\left(\Lan_{I}\left(K\right)\right)\circ\sN_{I}
\]
which, since $\sN_{I}$ is full-and-faithful, provides that 
\[
\sN_{J}\circ G=\Lan_{I}\left(K\right)
\]
which in turn provides, by pre-composition, the equality
\[
\sN_{J}\circ G\circ I=\Lan_{I}\left(K\right)\circ I
\]
and lastly, by the full-and-faithfulness of $I$, the equality
\[
\sN_{J}\circ G\circ I=K
\]
\end{proof}
In light of the lemma above, we provide an alternative description
of $\CwA$.
\begin{cor}
The $2$-category defined by the axioms \textbf{CWA0 }from\textbf{
}in Definition \ref{def:BMW-definition-of-CwA} with the further axioms
\begin{description}
\item [{CWA1$*$}] $1$-cells $\left(\fA,\sfE,I\right)\rightarrow\left(\sfB,\sfF,J\right)$
are co-continuous functors $\overline{K}:\wh{\sfA}\rightarrow\wh{\sfB}$
which pseudo-factor as in 
\[
\xyR{3pc}\xyC{3pc}\xymatrix{\fE\ar[r]\ar@{-->}[d] & \wh{\sfA}\ar[d]\\
\fF\ar[r] & \wh{\sfB}
}
\]
\end{description}
and we may replace \textbf{CWA2 }in that definition with
\begin{description}
\item [{CWA2$*$}] $2$-cells $\overline{K}\Longrightarrow\overline{L}:\left(\fA,\sfE,I\right)\rightarrow\left(\sfB,\sfF,J\right)$
of $\CwA_{\mathsf{}}$ are natural transformations of the underlying
co-continuous functors.
\end{description}
comprise a $2$-category $\CwA^{\p}$ which is bi-equivalent to $\CwA$
$2$-category.
\end{cor}

\begin{rem}
We will never envoke $\CwA^{\p}$ and $\CwA$ as distinct again.
\end{rem}

\begin{example}
In the category of globular sets, \textbf{$\wh{\bG}$}, we may identify
the subcategory of \textbf{linear pasting diagrams }of globes: let
$\bP\longrightarrow\wh{\bG}$ denote the full subcategory of $\wh{\bG}$
subtended by objects of the form
\[
\colim\left\{ \vcenter{\vbox{\xyR{1.5pc}\xyC{1.5pc}\xymatrix{\overline{n_{0}} &  & \overline{n_{1}} &  & \overline{n_{\ell-1}} &  & \overline{n_{\ell}}\\
 & \overline{m_{1}}\ar[ul]|-{t^{n_{0}-m_{1}}}\ar[ur]|-{s^{n_{1}-m_{1}}} &  & \cdots\ar[ur]\ar[ul] &  & \overline{m_{\ell-1}}\ar[ul]|-{t^{n_{\ell-1}-m_{\ell-1}}}\ar[ur]|-{s^{n_{\ell}-m_{\ell-1}}}
}
}}\right\} 
\]
for lists of non-negative integers, 
\[
n_{0},m_{1},n_{1},\dots,n_{\ell-1},m_{\ell-1},n_{\ell}
\]
with each $m_{i}\leq n_{i-1},n_{i}$. Since the subcategory $\bP\longrightarrow\wh{\bG}$
factors the Yoneda embedding it is dense, whence $\left(\bP,\wh{\bG},I\right)$
is a $0$-cell of $\CwA$.
\end{example}

\section{Monads with Arities}
\begin{defn}
A monad $T$ on $\sfE$ has arities $\sfA$ if the endofunctor $T$
is arities-respecting in the sense of \textbf{CWA1}.
\end{defn}

\begin{rem}
Those concerned with the fact that no criteria are imposed on the
$2$-cells of the monad $T$ may appeal to Lemma \ref{lem:Extension-of-monads}.
\end{rem}

\begin{lem}
\label{lem:monad-with-arities-and-presheaf-crap} Given a monad $T\circlearrowright\sfE$
and category with arities $\left(\sfA,\sfE,I\right)$, $T$ has arities
$\sfA$ precisely if, setting $\overline{T}=\Lan_{\sN_{I}}\left(\sN_{I}\circ T\right)$
, we have that $\wh{\sfA}^{\overline{T}}\liso\wh{\sfA_{T}}$.
\end{lem}

\begin{proof}
Suppose $T\circlearrowright\left(\sfA,\sfE,I\right)$. Since $\overline{T}$
is co-continuous $\wh{\fA}^{\overline{T}}=\wh{\sfA_{\overline{T}}}$
and then, since $\sN_{I}\circ T=\overline{T}\circ\sN_{I}$, $\sfA_{\overline{T}}=\sfA_{T}$.
Conversely, suppose $\wh{\sfA}^{\overline{T}}=\wh{\sfA_{T}}$ with
$\overline{T}=\Lan_{\sN_{I}}\left(\sN_{I}\circ T\right)$. Since $\wh{\sfA}^{\overline{T}}$
is a presheaf topos precisely if $\overline{T}$ is co-continuous
it follows that $\overline{T}=\Lan_{\Yo_{\sfA}}\left(\overline{T}\circ\Yo_{\sfA}\right)$.
But then 
\begin{eqnarray*}
\overline{T} & = & \Lan_{\Yo_{\sfA}}\left(\overline{T}\circ\Yo_{\sfA}\right)\\
 & = & \Lan_{\Yo_{\sfA}}\left(\overline{T}\circ\sN_{I}\circ I\right)\\
 & = & \Lan_{\Yo_{\sfA}}\left(\Lan_{\sN_{I}}\left(\sN_{I}\circ T\right)\circ\sN_{I}\circ I\right)\\
 & = & \Lan_{\Yo_{\sfA}}\left(\sN_{I}\circ T\circ I\right)
\end{eqnarray*}
 which, according \textbf{CWA1},\textbf{ }is to say that $T$ preserves
arities $\sfA$.
\end{proof}
\begin{rem}
Of the most useful and interesting aspects of the $2$-category $\CwA$
is that not only does it have Eilenberg-Moore objects for all monads,
but that the Eilenberg-Moore object of a monad is computable by way
of the bijective-on-objects full-and-faithful factorization system.
We will see this arising perhaps a bit less mysteriously from a limit
reflection result later on.
\end{rem}

\begin{example}
The monad $T$ on $\wh{\bG}$ whose algebras are strict-$\omega$-categories
is an arities preserving monad for $\mathbf{P}\longrightarrow\mathbf{G}$.
Moreover $\Theta\longrightarrow\StrCat$ is the Eilenberg-Moore object
for that monad in $\CwA$.
\end{example}

\section{From Arities to Sketches}

As described in \cite{BourkeGarner} the notion of density presentation,
due to Day, can be used to sharpen the abstract nerve theorm of Berger,
Mellies and Weber.
\begin{defn}
A \textbf{set of colimits $\Psi$ in a category $\sf C$ }is a set
of diagrams 
\[
\left\{ D_{j};I_{j}\longrightarrow\sf C\right\} _{j\in J}
\]
 for whichthe colimit $\colim D_{i}$ exists for all $j\in J$.

Given a full-and-faithful functor $K:\sfC\longrightarrow\sf A$ and
a diagram $D:I\longrightarrow\sfC$ with colimit $\colim D$, we say
that the \textbf{colimit of $D$ is} \textbf{$K$-absolute }if the
nerve $\sN_{K}:\sf A\longrightarrow\wh{\sfC}$ preserves that colimit,
i.e. if the canonical comparison map $\colim\sN_{K}D\liso\sN_{K}\left(\colim D\right)$
is an isomorphism.

Given a set of colimits $\Psi$ in a category $\sfC$ and a full-and-faithful
functor $K:\sfC\longrightarrow\sfA$, if $\sfA$ is the closure of
$\sfC$ under a set of $K$-absolute colimits then $\Psi$ is said
to be a \textbf{density presentation} for $K$.
\end{defn}

\begin{example}
The set of wide spans
\[
\left\{ \vcenter{\vbox{\xyR{1.5pc}\xyC{1.5pc}\xymatrix{\overline{n_{0}} &  & \overline{n_{1}} &  & \overline{n_{\ell-1}} &  & \overline{n_{\ell}}\\
 & \overline{m_{1}}\ar[ul]|-{t^{n_{0}-m_{1}}}\ar[ur]|-{s^{n_{1}-m_{1}}} &  & \cdots\ar[ur]\ar[ul] &  & \overline{m_{\ell-1}}\ar[ul]|-{t^{n_{\ell-1}-m_{\ell-1}}}\ar[ur]|-{s^{n_{\ell}-m_{\ell-1}}}
}
}}\right\} 
\]
indexed by the set of tuples of integers
\[
\left\{ \sd{\left(n_{0},m_{1},\dots,m_{\ell-1},n_{\ell}\right)}\ell\in\N,\ n_{0}\geq m_{1}\leq n_{1}\geq\dots\leq n_{\ell-1}\geq m_{\ell-1}\leq n_{\ell}\in\N\right\} 
\]
constitute a density presentation for the functor $\bG\longrightarrow\mathbf{P}$
which induced the functor $\mathbf{P}\longrightarrow\wh{\bG}$.
\end{example}

\begin{prop}
\label{prop:bourke-garner-density-pres-to-regulus-machine}Given a
density presentation 
\[
\Psi=\left\{ D_{\ell}:I_{\ell}\longrightarrow\sf A\right\} _{\ell\in L}
\]
 for a full-and-faithful functor $\sf C\longrightarrow\sf A$, if
$M$ is a monad with arities for $J$ as a $0$-cell of $\CwA$, so
we have a diagram
\[
\xyR{1.5pc}\xyC{1.5pc}\xymatrix{ & \sf A_{M}\ar[r]^{J_{M}} & \wh{\sf C}^{M}\ar[r]^{\sN_{J_{M}}}\ar@<1ex>[d] & \wh{\sf A_{M}}\ar@<1ex>[d]\\
\sf C\ar[r] & \sf A\ar[r]_{J}\ar[u] & \wh{\sf C}\ar[r]_{\sN_{J}}\ar@<1ex>[u]^{F^{M}} & \wh{\sf A}\ar@<1ex>[u]^{F^{M}}
}
\]
then a presheaf $X:\sf A_{M}^{\op}\longrightarrow\S$ is the nerve
of an $M$-algbra if and only if $X$ is right orthogonal to the set
of monomorphisms 
\[
\sf V_{\Psi}=\left\{ \colim\sN_{J_{M}}F^{M}JD_{\ell}\longrightarrow\sN_{J_{M}}\colim F^{M}JD_{\ell}\right\} 
\]
\end{prop}

\begin{proof}
This follows from Theorem 36 of \cite{BourkeGarner}.
\end{proof}
\begin{rem}
In the case where the nerves of the colimits $\sN_{J_{M}}\colim F^{M}D_{\ell}$
are representable by $T_{\ell}$ we will denote the colimit $\colim\sN_{J_{M}}F^{M}D_{\ell}$
by $V\left[T_{\ell}\right]$ for reasons which will become clear.
\end{rem}

\begin{example}
Recall that the set of wide spans 
\[
\Psi=\left\{ \vcenter{\vbox{\xyR{1.5pc}\xyC{1.5pc}\xymatrix{\overline{n_{0}} &  & \overline{n_{1}} &  & \overline{n_{\ell-1}} &  & \overline{n_{\ell}}\\
 & \overline{m_{1}}\ar[ul]|-{t^{n_{0}-m_{1}}}\ar[ur]|-{s^{n_{1}-m_{1}}} &  & \cdots\ar[ur]\ar[ul] &  & \overline{m_{\ell-1}}\ar[ul]|-{t^{n_{\ell-1}-m_{\ell-1}}}\ar[ur]|-{s^{n_{\ell}-m_{\ell-1}}}
}
}}\right\} 
\]
 constitute a density presentation for the inclusion $\bG\longrightarrow\mathbf{P}$.
The set $\sfV_{\Psi}$ is the set of spine inclusions
\[
\sf V=\left\{ V\left[T\right]\longrightarrow\Theta\left[T\right]\right\} 
\]
It follows then, from the proposition above, that a cellular set $X$
is the nerve of a strict $\omega$-category precisely if it is right
orthogonal to this set $\sf V$. Right orthogonality with respect
to the spines is of course the familiar Segal condition.
\end{example}

We can use this proposition this to provide a more abstract proof
of Berger's cellular nerve criterion in terms of inner horns.
\begin{example}
\label{exa:Inner-horns-by-density-presentation} Let $\kappa:S\longrightarrow T$
be a hyperface of a cell $T$ of $\Theta$. We recall that $\Lambda^{\kappa}\left[T\right]$
is the union in $\Theta\left[T\right]$ of all hyperfaces of $T$
which are not the hyperface $\kappa$. Each presheaf $\Lambda^{\kappa}\left[T\right]$
is therefore a co-equalizer. Let $F_{\left(\kappa,T\right)}$ be the
set of hyperfaces of $T$ which are not the hyperface $\kappa$ and
let $I_{\left(\kappa,T\right)}$ be the set of representables including
into intersections if hyperfaces in $F_{\left(\kappa,T\right)}$.
\begin{rem}
The strange wording regarding the definition of $I_{\left(\kappa,T\right)}$
here is chosen to accomodate the fact that hyperfaces in $\Theta$
may intersect as coproducts of representables and not always representables
as in the case of $\t$. Consider for example the intersection 
\[
\left[1\right];d^{0}\cap\left[1\right];d^{1}\longrightarrow\left[1\right];\left[1\right]
\]
 which computes to the coproduct of representable presheaves $\Theta\left[0\right]\coprod\Theta\left[0\right]$.
\end{rem}

Given any hyperface $\kappa:S\longrightarrow T$ of cell $T$ of $\Theta$
it follows that the canonical map 
\[
\colim\left\{ \prod_{R^{\p}\in I_{\left(\kappa,T\right)}}\Theta\left[R^{\p}\right]\rightrightarrows\prod_{R\in F_{\left(\kappa,T\right)}}\Theta\left[R\right]\right\} \longrightarrow\Theta\left[T\right]
\]
 is an isomorphism.

The set of all of these pairs of parralel arrows comprise a density
presentation, call it $\Upsilon$, for the identity $\Theta\longrightarrow\Theta$.
We will use this density presentation to provide an alternative proof
of Berger's cellular nerve theorem as follows. Let $\overline{T}$
be the obvious left Kan extension of the monad $T\circlearrowright\wh{\bG}$
whose algebras are strict $\omega$-categories. Observe that $\Theta\longrightarrow\wh{\Theta}^{\overline{T}}$
is the Eilenberg-Moore object for the arities preserving monad $\overline{T}$.
Thus Proposition \ref{prop:bourke-garner-density-pres-to-regulus-machine}
applies and a cellular set is the nerve of a strict $\omega$-category
precisely if it is right orthogonal to $\sf V_{\Upsilon}$ and the
set $\sf{V_{\Upsilon}}$ is the set $\Lambda_{\sf{inner}}$ by design.
\end{example}

\section{$\protect\Cat$-enriched Relative Adjunctions between $\protect\Cat\protect\hookrightarrow\protect\CAT$
and $\protect\Prof$, and $\protect\CwA$}

\subsection{$\protect\Cat$-enriched relative adjunctions}

In this short subsection we recall the definition of $\Cat$-enriched
relative adjunctions and recall the associated limit and colimit preserving
properties of relative adjoint functors.
\begin{defn}
Given a (not necessarily commuting) triangle of $2$-functors 
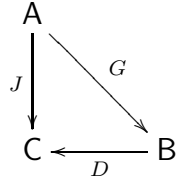
\[
\xyR{3pc}\xyC{3pc}\xymatrix{\sfA\ar[dr]^{G}\ar[d]_{J}\\
\sfC & \sfB\ar[l]^{D}
}
\]
if there exists some natural equivalence of categories
\[
\sfC\left(J\left(\_\right),D\left(\_\right)\right)\liso\sfD\left(G\left(\_\right),\_\right)
\]
we say that \textbf{$G$ }is\textbf{ left adjoint to $D$ relative
to $J$.}
\end{defn}

\begin{lem}
Given \textbf{$G$ }left adjoint to $D$ relative to $J$ as in the
definition above:
\end{lem}

\begin{itemize}
\item $G$ preserves all weighted colimits that $J$ does; and
\item $D$ preserves all weighted limits if $J$ is dense\footnote{As observed in footnote 13 of \cite{Ulmer}, this is certainly not
necessary, but is at least sufficient.}.
\end{itemize}
\begin{proof}
Suppose $X:\sfJ\longrightarrow\sfA$ to be a $2$-functor and $W:\sfJ\longrightarrow\Cat$
to be a weight. Then, letting $\colim^{W}X$ denote the $W$-weighted
colimit of $X$, we see that if $J$ preserves $\colim^{W}X$ then,
for each $n\in\sfB$ we've the following sequence of equivalences
of categories.
\begin{eqnarray*}
\sfB\left(G\left(\colim^{W}X\right),n\right) & \liso & \sfC\left(J\left(\colim^{W}X\right),D\left(n\right)\right)\\
 & \liso & \sfC\left(\colim^{W}JX,D\left(n\right)\right)\\
 & \liso & \left[\sfJ,\Cat\right]\left(W,\sfC\left(J\circ X\left(\_\right),D\left(n\right)\right)\right)\\
 & \liso & \left[\sfJ,\Cat\right]\left(W,\sfB\left(G\circ X\left(\_\right),n\right)\right)\\
 & \liso & \sfB\left(\colim^{W}GX,n\right)
\end{eqnarray*}
Since the equivalences are natural in $n$ and vary over all of $\sfB$,
this provides an isomorphism
\[
G\left(\colim^{W}X\right)\liso\colim^{W}GX
\]
.

The similar argument for the weighted limit preservation arrives at
equivalences of categories 
\begin{eqnarray*}
\sfC\left(J\left(a\right),D\left(\llim^{W}X\right)\right) & \liso & \sfC\left(J\left(A\right),\llim^{W}DX\right)
\end{eqnarray*}
which are natural in $a\in\sfA$. In the context of our assumption
that $J$ is dense, it follows that $D\left(\llim^{W}X\right)\liso\llim^{W}DX$.
\end{proof}

\subsection{$\protect\Cat\protect\hookrightarrow\protect\CAT$-relative adjunctions}
\begin{lem}
The functor $\Prof\longrightarrow\CAT$ is $\Cat\hookrightarrow\CAT$-right
adjoint to $\Cat\rightarrow\Prof$.
\[
\xyR{3pc}\xyC{3pc}\xymatrix{\Cat\ar[dr]\ar[d]\\
\CAT & \Prof\ar[l]
}
\]
\end{lem}

\begin{proof}
See that we've a natural isomorphism of $\Hom$-categories 
\[
\Prof\left(\wh{\sfA},\wh{\sfB}\right)\liso\CAT\left(\sfA,\wh{\sfB}\right)
\]
\end{proof}
\begin{lem}
The functor
\[
\xyR{0pc}\xyC{3pc}\xymatrix{\Cat\ar[r] & \CwA\\
\sfB\ar@{|->}[r] & \left(\sfB,\sfB,\id_{\sfB}\right)
}
\]
is left $\Cat\hookrightarrow\CAT$ adjoint to the functor 
\[
\xyR{0pc}\xyC{5pc}\xymatrix{\CwA\ar[r] & \mathsf{CAT}\\
\left(\fA,\fE,I\right)\ar@{|->}[r] & \sfE
}
\]
\end{lem}

\begin{proof}
See that 
\[
\CwA\left(\left(\sfB,\sfB,\id_{\sfB}\right),\left(\sfA,\sfE,I\right)\right)\liso\mathsf{CAT}\left(\sfB,\sfE\right)
\]
as the categories (left side) of pseudo-commuting squares (with right
vertical map co-continuous) and (right side) arrows are equivalent.
\[
\left\{ \vcenter{\vbox{\xyR{3pc}\xyC{3pc}\xymatrix{\sfB\ar[r]\ar[d] & \wh{\sfB}\ar[d]^{\cc}\\
\fE\ar[r] & \wh{\fA}
}
}}\right\} \liso\left\{ \vcenter{\vbox{\xyR{3pc}\xyC{3pc}\xymatrix{\sfB\ar[d]\\
\sfE
}
}}\right\} 
\]
\end{proof}
In fact, there is a morphism of $\Cat\hookrightarrow\CAT$ adjunctions
between these two\footnote{Indeed, someone more sophisticated than this author will smell Yoneda
structures lurking behind the scenes as the non-identity $2$-cells
are the Yoneda embedding and the restricted Yoneda embedding respectively.}.
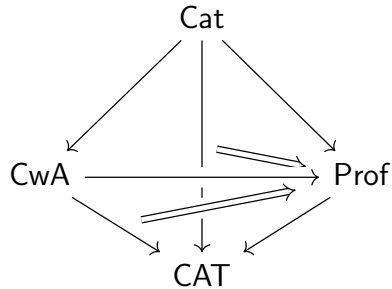
\[
	\begin{tikzcd}
		& {\Cat } \\
		\\
		\CwA && {\Prof } \\
		& \CAT
		\arrow[""{name=0, anchor=center, inner sep=0}, from=3-1, to=4-2, crossing over]
		\arrow[from=3-3, to=4-2]
		\arrow[from=1-2, to=3-1]
		\arrow[""{name=1, anchor=center, inner sep=0}, from=1-2, to=4-2] 
		\arrow[from=1-2, to=3-3]
		\arrow[shorten <=5pt, shorten >=5pt, Rightarrow, from=1, to=3-3]
		\arrow[from=3-1, to=3-3, crossing over]
		\arrow[shorten <=9pt, shorten >=9pt, Rightarrow, from=0, to=3-3, crossing over]
	\end{tikzcd}
\]
\begin{rem}
\label{rem:limit-and-colimit-preservation-stuff-about-the-relative-adjunctions}
As consequences of these lemmata then, we find that:
\end{rem}

\begin{itemize}
\item $\Cat\longrightarrow\Prof$ and $\Cat\longrightarrow\CwA$ preserves
all colimits which the inclusion $\Cat\longrightarrow\CAT$ does;
and
\item since $\Cat\longrightarrow\CAT$ is dense, the functors $\Prof\longrightarrow\CAT$
and $\CwA\longrightarrow\CAT$ preserves all limits.
\end{itemize}

\section{A $3$-Oriental, a Comma Category, and the Limit Reflection of $\protect\CwA\protect\longrightarrow\protect\CAT^{\protect\bfTwo}$}

Now, it's pleasing to the eye to observe that the relative adjunctions
of the previous section assemble into a $3$-oriental suggestive of
a Yoneda structure. More, it's probable that some rumination upon
this will yield a formal theory of $\CwA$'s in an sufficiently endowed
$3$-category. However, our concern is more pedestrian.

The unit $2$-cell 
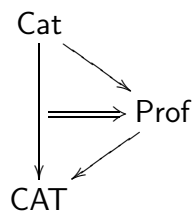
\[
\xyR{1.5pc}\xyC{1.5pc}\xymatrix{\Cat\ar[dr]\ar[dd]\\
\ar@{=>}[r] & \Prof\ar[dl]\\
\CAT
}
\]
induces a $2$-functor $\CwA\longrightarrow\CAT\downarrow\left(\Prof\rightarrow\CAT\right)$,
and it's not hard to see that this $2$-functor is full-and-faithful
- consider our second definition for $\CwA$. Full-and-faithful functors
however reflect all (weighted) limits and colimits; by way of this
fact, and the limit preservation properties of Remark \ref{rem:limit-and-colimit-preservation-stuff-about-the-relative-adjunctions},
we develop a criterion for the existence of some weighted limits in
$\CwA$.
\begin{prop}
The $2$-functor 
\[
\CwA\longrightarrow\CAT^{\bfTwo}
\]
reflects all limits.
\end{prop}

\begin{proof}
As we've already observed, the $2$-functor
\[
\CwA\longrightarrow\CAT\downarrow\left(\Prof\rightarrow\CAT\right)
\]
reflects limits as it is a full-sub-$2$-category. As such, we will
show that 
\[
\CAT\downarrow\left(\Prof\rightarrow\CAT\right)\longrightarrow\CAT^{\bfTwo}
\]
creates limits, whence the composition $\CwA\longrightarrow\CAT^{\bfTwo}$
will reflect limits as well.

For the proof that $\CAT\downarrow\left(\Prof\rightarrow\CAT\right)\longrightarrow\CAT^{\bfTwo}$
creates limits we will apply Corollary \ref{cor:creativity-from-creativity}
and for that it suffices to show that $\Prof\longrightarrow\CAT$
creates limits. For the preservation part it is enough to recall that
$\Prof\longrightarrow\CAT$ is right adjoint relative to a dense functor.
For the reflectivity we observe that $\Prof\longrightarrow\CAT$ factors
as 
\[
\Prof\longrightarrow\mathsf{CoComp}\CAT\longrightarrow\CAT
\]
 and the first leg is a full sub-$2$-category, whence it reflects
all limits, and 
\[
\mathsf{CoCompCAT}\longrightarrow\CAT
\]
 creates all limits as it is monadic.\footnote{for the small presheaves monad}
\end{proof}
\begin{cor}
\label{thm:Prof-is-the-obstruction}Let
\[
\xyR{0pc}\xyC{3pc}\xymatrix{\left(\sfA,\sfE,I\right):J\ar[r] & \CwA\\
j\ar@{|->}[r] & \left(\sfA_{j},\sfE_{j},I_{j}\right)
}
\]
be a $2$-diagram and let $W:J\longrightarrow\Cat$ be a weight. Then
the weighted limit $\llim^{W}\left(\sfA,\sfE,I\right)$ in $\CwA$
exists if:
\begin{itemize}
\item the weighted limit $\llim^{W}\sfA$ exists in $\Prof$;
\item the weighed limit $\llim^{W}\sN_{I}$ taken in $\CAT^{\bfTwo}$ is
induced by a full subcategory
\[
\llim^{W}\sfA\longrightarrow\llim^{W}\sfE
\]
\end{itemize}
and in such case it is 
\[
\left(\llim_{\Prof}^{W}\sfA,\llim_{\CAT}^{W}\sfE,\llim^{W}\sfA\longrightarrow\llim^{W}\sfE\right)
\]
\end{cor}

\begin{proof}
By the limit reflection proposition above it suffices to show that,
under the hypotheses of this corollary, the limit $\llim^{W}\sN_{I}$
lies in the image of $\CwA$. To that end, see that if the weighted
limit $\llim^{W}\sfA$ exists in $\Prof$, $\llim^{W}\sN_{I}$ in
$\CAT^{\bfTwo}$ will be in the image of $\CwA$ if $\llim^{W}\sN_{I}$
is the nerve associated to a subcategory - density follows from the
fact that full-and-faithful functors are closed under limits, being,
as they are, a right orthogonality class.
\end{proof}

\section{Examples}
\begin{example}
\textbf{Eilenberg-Moore Objects of Monads with Arities }- Suppose
that $T\circlearrowright\left(\sfA,\sfE,I\right)$ is a monad in $\CwA$.
It follows from Corollary \ref{thm:Prof-is-the-obstruction} and Lemma
\ref{lem:monad-with-arities-and-presheaf-crap} that $\left(\sfA_{T},\sfE^{T},I^{T}\right)$
is the Eilenberg-Moore object for $T$ if $\sfA_{T}\longrightarrow\sfE^{T}$
is a full subcategory, but this much is clear.
\end{example}

\begin{rem}
The promised connection between this limit reflection result and the
bijective-on-objects full-and-faithful factorization system comes
from the fact that Kleisli objects in $\Prof$ are bijective on objects
functors $A\longrightarrow A_{T}$.
\end{rem}

\begin{example}
\textbf{\label{exa:A-cellular-description-of-Z-categories}A cellular
description of strict-$\Z$-categories - }In this example we will
identify useful arities for the category 
\[
\llim\left\{ \cdots\xrightarrow{S^{*}}\StrCat\xrightarrow{S^{*}}\StrCat\right\} \liso\wh{\bG_{\Z}}^{T_{\Z}}
\]
 of $\Z$-categories.

Now, the functor $S:\bG\longrightarrow\bG$ which induced $S^{*}$
also induces $S_{!}:\wh{\bG}\longrightarrow\wh{\bG}$ which restricts
along $\mathbf{P\longrightarrow\bG}$ to $S_{\bP}=\sd{S_{!}}_{\bP}:\bP\longrightarrow\bP$.
These functors fit into the the following commuting diagram.
\[
\xyR{1.5pc}\xyC{1.5pc}\xymatrix{\bG\ar[r]^{J}\ar[d]_{S} & \bP\ar[d]|-{S_{\bP}}\ar[r]^{I} & \wh{\bG}\ar[d]^{S_{!}}\\
\bG\ar[r]_{J} & \bP\ar[r]_{I} & \wh{\bG}
}
\]
It follows from that commutativity that the right square in 
\[
\xyR{1.5pc}\xyC{1.5pc}\xymatrix{\bP\ar[d]_{S_{\bP}}\ar[r]^{I} & \wh{\bG}\ar[r]^{\sN_{I}} & \wh{\bP}\\
\bP\ar[r]_{I} & \wh{\bG}\ar[u]|-{S^{*}}\ar[r]_{\sN_{I}} & \wh{\bP}\ar[u]_{\left(S_{\bP}\right)^{*}}
}
\]
pseudo-commutes, whence it comprises a $1$-cell in $\CwA$, as for
any $X\in\wh{\bG}$, 
\begin{eqnarray*}
\sN_{I}S^{*}X & = & \wh{\bG}\left(I,S^{*}X\right)\\
 & \iso & \wh{\bG}\left(S_{!}\circ I,X\right)\\
 & \iso & \wh{\bG}\left(I\circ S_{\bP}\right)\\
 & \liso & \wh{\bP}\left(S_{\bP},\sN_{I}X\right)\\
 & = & S_{\bP}^{*}\sN_{I}X
\end{eqnarray*}
and the left square pseudo-commutes as $S_{!}$ is full-and-faithful\footnote{It is important here that we have used the \emph{irreflexive }globe
category and not the reflexive globe category. This is an instance
of the principle/phenomenon which undergirds Makkai's FOLDS or Henry's
realated work on the language of a model category. The point is that
inverse categories provide good ``theories of sorts'', whereas other
categories do not.} so the unit of the adjunction $S_{!}\dashv S^{*}$ is invertible.

We may then assemble a much larger pseudo-commuting diagram. 
\[
\xyR{1.5pc}\xyC{1.5pc}\xymatrix{\bP\ar[d]\ar[r] & \wh{\bG}\ar[r] & \wh{\bP}\\
\bP\ar[r]\ar[d] & \wh{\bG}\ar[u]\ar[r] & \wh{\bP}\ar[u]\\
\bP\ar[r]\ar[d] & \wh{\bG}\ar[u]\ar[r] & \wh{\bP}\ar[u]\\
\vdots & \vdots\ar[u] & \vdots\ar[u]
}
\]
We note that the columns of $\wh{\bG}$ 's and $\wh{\bP}$'s are comprised
of iso-fibrations, so the conical limits over them compute pseudo-limits.
The left-hand column of commuting squares induces a map 
\[
I_{\Z}:\underbrace{\colim\left\{ \cdots\leftarrow\bP\leftarrow\bP\right\} }_{\bP_{\Z}}\longrightarrow\underbrace{\llim\left\{ \cdots\rightarrow\wh{\bG}\rightarrow\wh{\bG}\right\} }_{\wh{\bG_{\Z}}}
\]
and the nerve associated to this map
\[
\sN_{I_{\Z}}:\wh{\bG_{\Z}}\longrightarrow\wh{\bP_{\Z}}
\]
is the pseudo-limit in $\CAT^{\bfTwo}$ of the horizontal parts of
the right-hand column. By Theorem \ref{thm:Prof-is-the-obstruction}
to show that $\left(\bP_{\Z},\wh{\bG_{\Z}},I_{\bZ}\right)$ is the
conical limit in $\CwA$ of 
\[
\cdots\rightarrow\left(\bP,\wh{\bG},I\right)\rightarrow\left(\bP,\wh{\bG},I\right)
\]
it then suffices to show that 
\[
I_{\bZ}:\bP_{\Z}\longrightarrow\wh{\bG_{\Z}}
\]
 is a full subcategory.

Using the $\Z_{\leq0}^{\op}$ indexing for the copies of $\bP$ in
$\bP_{\Z}$ and the $\Z_{\leq0}$ indexing for the copies of $\wh{\bG}$
in $\llim\wh{\bG}$ we may write the map $I_{\Z}$ explicitly as 
\[
\xyR{0pc}\xyC{5pc}\xymatrix{I_{\Z}:\bP_{\Z}\ar[r] & \llim\left\{ \cdots\rightarrow\wh{\bG}\rightarrow\wh{\bG}\right\} \\
\left(-i,T\right)\ar@{|->}[r] & \left(\left(S^{*}\right)^{i}\sN_{I}\left(T\right),\dots,S^{*}\sN_{I}\left(T\right),\sN_{I}S_{\bP}\left(T\right),\dots\right)
}
\]
Since $\sN_{I}$ is a full subcategory, it follows that $I_{\bZ}$
is injective on objects and faithful, but since $S_{\bP}$ is also
full-and-faithful, it follows that $I_{\bZ}$ is full as well. Composing
$I_{\bZ}$ with the equivalence $\llim\left\{ \cdots\rightarrow\wh{\bG}\rightarrow\wh{\bG}\right\} \liso\wh{\bG_{\Z}}$
permits to identify $I_{\Z}$ as the full subcategory of $\wh{\bG_{\Z}}$
on presheaves of the form
\[
\colim\left\{ \vcenter{\vbox{\xyR{1.5pc}\xyC{1.5pc}\xymatrix{\overline{n_{0}} &  & \overline{n_{1}} &  & \overline{n_{\ell-1}} &  & \overline{n_{\ell}}\\
 & \overline{m_{1}}\ar[ul]\ar[ur] &  & \cdots\ar[ur]\ar[ul] &  & \overline{m_{\ell-1}}\ar[ul]\ar[ur]
}
}}\right\} 
\]
 for \emph{integers} $n_{0},m_{1},n_{1},\dots,n_{\ell-1},m_{\ell-1},n_{\ell}$
with each $m_{i}\leq n_{i-1},n_{i}$; $\bP_{\Z}$ is the full subcategory
of $\wh{\bG_{\Z}}$ of finite $\Z$-globular pasting diagrams. These
wide spans moreover comprise a density presentation $\Psi_{\Z}$ for
$\bG_{\Z}\longrightarrow\mathbf{P}_{\Z}$.

Lastly, since limits commute with limits, whence limits over towers
commute with the taking of Eilenberg-Moore objects, it follows that
\[
\left(\Theta_{\bZ}:=\colim\left\{ \dots\leftarrow\Theta\leftarrow\Theta\right\} ,\wh{\bG_{\Z}}^{T_{\Z}},I_{\Z}^{T}\right)
\]
 is the Eilenberg-Moore object for $T_{\Z}$ as a monad in $\CwA$
on $\left(\bP_{\Z},\wh{\bG_{\Z}},I_{\Z}\right)$. Thus, by Proposition
\ref{prop:bourke-garner-density-pres-to-regulus-machine}, and a repeat
of our argument from Example \ref{exa:Inner-horns-by-density-presentation}
it follows that a $\Z$-cellular set is the nerve of a strict-$\Z$-category
if and only if it is:
\begin{itemize}
\item right orthogonal to the $\Z$-cellular analog of the spines; or
\item right orthogonal to the $\Z$-cellular analog of the inner horns.
\end{itemize}
\end{example}

\begin{rem}
In a sequel to this work, we will reprove a result of \cite{LessardThesis},
and we will show how the homotopy coherent version of this cellular
presentation of $\Z$-categories provides a model of Spectra as pointed
$\Z$-groupoids.
\end{rem}

\appendix

\part*{Appendix}

\section{\label{part:Generalized-Spectrification}\label{part:Generalized-Spectrification-1}Generalized
Spectrification}

Recall that in the $2$-category $\Cat$, given:
\begin{itemize}
\item a $2$-category $\fJ$;
\item a weight $2$-functor $W:\fJ\longrightarrow\Cat$; and
\item a $2$-diagram of categories $\fX:\fJ\longrightarrow\Cat$
\end{itemize}
the $\Hom$-category of natural transformations and modifications
is determined by the $\Hom$-categories $\Cat\left(W_{i},\mathsf{X}_{j}\right)$
by way of the end formula
\[
\left[\fJ,\Cat\right]\left(W,\fX\right)\liso\int_{j\in\fJ}\Cat\left(W_{j},\fX_{j}\right)
\]

We may then be wont to ask: which properties or structures involving
the values taken by the functor $\Cat\left(W_{\left(\_\right)},\mathsf{X}_{\left(\_\right)}\right)$
are sufficiently (and appropriately) coherent so that they induce
similar properties or structures on the weighted limit? In particular,
if for some objects $i$ and $j$ of $\mathsf{J}$, we have that $\mathsf{X}_{j}$
has all $W_{i}$ colimits then we have adjunctions 
\[
\left\{ \xyC{5pc}\xymatrix{\left[W_{i},\fX_{j}\right]\ar@/^{.625pc}/[r]^{\colim^{\left(i,j\right)}} & \fX_{j}\ar@/^{.625pc}/[l]_{\bot}^{\t^{\left(i,j\right)}}}
\right\} _{\left(i,j\right)\in\Ob\left(\mathsf{J}\right)^{2}}
\]
When do these object-wise adjunctions beget adjunctions in $\left[\mathsf{J},\Cat\right]$?
Taking a hint from classical treatments of stable homotopy theory
we develop an easy criterion and show, by explicit example, that the
spectrification functors of stable homotopy theory are recovered from
this $2$-categorical treatment.

\subsection{A Criterion for small conical limits in $\protect\CAT$ to be reflective
in weighted-limits in $\protect\CAT$: generalized spectrification}
\begin{thm}
\label{thm: generalized-spectrification-strict-limit}Let $\mathsf{J}$
be a $1$-category and let $\mathsf{X}:\sfJ\longrightarrow\CAT$ be
a diagram of categories. Suppose further that, for all $f:j\rightarrow k$
in $\sfJ$:
\begin{itemize}
\item the diagram $\fX:\fJ\longrightarrow\CAT$ is fibrant in the injective-canonical
model structure on $\left[\sf J,\CAT\right]$;
\item the category $\mathsf{X}_{k}$ admits all $W_{j}$ and $W_{k}$-colimits;
\item the functor $\mathsf{X}_{f}$ preserves $W_{j}$ colimits; and
\item the functor $W_{f}$ is final.
\end{itemize}
Then, the functor 
\[
\xymatrix{\left[\mathsf{J},\CAT\right]\left(W,\mathsf{X}\right) & \left[\sfJ,\CAT\right]\left(\bullet,\mathsf{X}\right)\ar[l]_{\Delta}}
\]
 admits a left adjoint, which we will call $\sf{sp}$.%
\end{thm}

\begin{proof}
Consider the commutative diagram
\[
\vcenter{\vbox{\xyR{3pc}\xyC{3pc}\xymatrix{\sf{Ps}\left(\sf J,\CAT\right)\left(W,\sf X\right) & \sf{Ps}\left(\sf J,\CAT\right)\left(\bullet,\sf X\right)\ar[l]_{\Delta^{\p}}\\
\left[\sf J,\CAT\right]\left(W,\sf X\right)\ar[u]^{I_{W}} & \left[\sf J,\CAT\right]\left(\bullet,\sf X\right)\ar[u]_{I}\ar[l]^{\Delta}
}
}}
\]
where $I$ and $I_{W}$ are the obvious inclusions. Since $I_{W}$
is full-and-faithful, if $I_{W}\Delta$ admits a left adjoint functor
$L$, then $LI_{W}$ is left adjoint to $\t$. But $I_{W}\Delta=\Delta^{\p}I$,
so to produce a left adjoint to $\Delta$ it suffices to produce one
for $\Delta^{\p}I$. Since $\sf X$ is fibrant however, and $\left[\sf J,\CAT\right]\left(\bullet,\sf X\right)$
enjoys the universal property of the conical limit $\llim\sf X$,
$I$ is an adjoint equivalence, whence to produce a left adjoint to
$\Delta^{\p}I$ it is enough to produce a left adjoint for $\Delta^{\p}$.
It is just such an adjoint we will now describe.

Recall that for any weight $W$, the $\Hom$-category $\left[\mathsf{J},\CAT\right]\left(W,\mathsf{X}\right)$,
may be computed as the equalizer
\[
\left[\fJ,\CAT\right]\left(W,\fX\right)\liso\llim\left\{ \vcenter{\vbox{\xyR{0pc}\xyC{5pc}\xymatrix{\prod_{i\in\fJ}\left[W_{i},\fX_{i}\right]\ar@/^{.5pc}/[r]^{\rho}\ar@/_{.5pc}/[r]_{\lambda} & \prod_{j\in\fJ}\prod_{k\in\fJ}\left[\fJ\left(j,k\right),\left[W_{j},\fX_{k}\right]\right]}
}}\right\} 
\]
where 
\[
\frac{\rho_{j,k}:\left[W_{j},\fX_{j}\right]\longrightarrow\left[\fJ\left(j,k\right),\left[W_{j},\fX_{k}\right]\right]}{\left[W_{j},\sf X_{\left(\_\right)}\right]:\fJ\left(j,k\right)\longrightarrow\left[\left[W_{j},\fX_{j}\right],\left[W_{j},\fX_{k}\right]\right]}
\]
and like-wise 
\[
\frac{\lambda_{j,k}:\left[W_{k},\fX_{k}\right]\longrightarrow\left[\fJ\left(j,k\right),\left[W_{j},\fX_{k}\right]\right]}{\left[W_{\left(\_\right)},\fX_{k}\right]:\fJ\left(j,k\right)\longrightarrow\left[\left[W_{k},\fX_{k}\right],\left[W_{j},\fX_{k}\right]\right]}
\]
More, since $\fJ$ is a $1$-category, we have an isomorphism
\[
\prod_{j\in\fJ}\prod_{k\in\fJ}\left[\fJ\left(j,k\right),\left[W_{j},\fX_{k}\right]\right]\liso\prod_{j\in\fJ}\prod_{k\in\fJ}\prod_{f\in\fJ\left(j,k\right)}\left[W_{j},\fX_{k}\right]
\]
and thus the map $\rho$ may be made explicit as a product.
\[
\left(\left(\left[W_{j},\fX_{f}\right]\right)_{f\in\fJ\left(j,k\right)}\right)_{k\in\fJ}:\prod_{j\in\fJ}\left(\xyR{0pc}\xyC{5pc}\xymatrix{\left[W_{j},\fX_{j}\right]\ar[r] & \underset{k\in\fJ}{\prod}\underset{f\in\fJ\left(j,k\right)}{\prod}\left[W_{j},\fX_{k}\right]}
\right)
\]
Similarly $\lambda$ may be made explicit as the product.
\[
\left(\left(\left[W_{f},\fX_{k}\right]\right)_{f\in\fJ\left(j,k\right)}\right)_{j\in\fJ}:\prod_{k\in\fJ}\left(\xyR{0pc}\xyC{5pc}\xymatrix{\left[W_{k},\fX_{k}\right]\ar[r] & \underset{j\in\fJ}{\prod}\underset{f\in\fJ\left(j,k\right)}{\prod}\left[W_{j},\fX_{k}\right]}
\right)
\]
The pseudo-limit of that same diagram computes the $\Hom$-category
$\sf{Ps}\left(W,\sf X\right)$.

Thus, thinking of the taking of pseudo-equalizers as a functor 
\[
\underleftarrow{\sf{pslim}}:\sf{Ps}\left(\partial\overline{1},\CAT\right)\longrightarrow\Cat
\]
 we see that $\Delta^{\p}$ is $\underleftarrow{\sf{pslim}}$ of the
pseudo-natural transformations of functors $\partial\overline{1}\longrightarrow\CAT$
which comprises the diagram
\[
\vcenter{\vbox{\xyR{5pc}\xyC{7pc}\xymatrix{\prod_{i\in\fJ}\left[W_{i},\fX_{i}\right]\ar@/^{.5pc}/[r]^{\rho}\ar@/_{.5pc}/[r]_{\lambda} & \prod_{j\in\fJ}\prod_{k\in\fJ}\prod_{f\in\fJ\left(j,k\right)}\left[W_{j},\fX_{k}\right]\\
\prod_{i\in\fJ}\fX_{i}\ar@/^{.5pc}/[r]^{\rho}\ar@/_{.5pc}/[r]_{\id}\ar[u]_{\prod\t^{\left(i,i\right)}} & \prod_{j\in\fJ}\prod_{k\in\fJ}\prod_{f\in\fJ\left(j,k\right)}\fX_{k}\ar[u]_{\prod\prod\prod\t^{\left(j,k\right)}}
}
}}
\]
As such, it suffices to find left adjoints $L^{\left(i,j\right)}$
\[
\xyC{5pc}\xymatrix{\left[W_{i},\fX_{j}\right]\ar@/^{.625pc}/[r]^{L^{\left(i,j\right)}} & \fX_{j}\ar@/^{.625pc}/[l]_{\bot}^{\t^{\left(i,j\right)}}}
\]
 for each pair of objects $i$ and $j$ of $\sf J$ such that the
squares - solid and dashed, and solid and dotted respectively - in
the assembled diagram
\[
\vcenter{\vbox{\xyR{5pc}\xyC{7pc}\xymatrix{\prod_{i\in\fJ}\left[W_{i},\fX_{i}\right]\ar@/^{.5pc}/@{-->}[r]^{\prod\left(\left(\left[W_{k},\sf X_{f}\right]\right)\right)}\ar@/_{.5pc}/@{..>}[r]_{\prod\left(\left(\left[W_{f},\sf X_{k}\right]\right)\right)}\ar[d]_{\prod L^{\left(i,i\right)}} & \prod_{j\in\fJ}\prod_{k\in\fJ}\prod_{f\in\fJ\left(j,k\right)}\left[W_{j},\fX_{k}\right]\ar[d]^{\prod\prod\prod L^{\left(j,k\right)}}\\
\prod_{i\in\fJ}\fX_{i}\ar@/^{.5pc}/@{-->}[r]^{\prod\left(\left(\sf X_{f}\right)\right)}\ar@/_{.5pc}/@{..>}[r]_{\id} & \prod_{j\in\fJ}\prod_{k\in\fJ}\prod_{f\in\fJ\left(j,k\right)}\fX_{k}
}
}}
\]
pseudo-commute.

Our hypotheses will imply that the family 
\[
\left\{ L^{\left(i,j\right)}=\colim^{\left(i,j\right)}:\left[W_{i},\sf X_{j}\right]\longrightarrow\sf X_{j}\right\} _{\left(i,j\right)\in\Ob\left(\sf J\right)^{2}}
\]
 satisfies this criterion. Since $W_{f}$ is final for all morphisms
$f:j\rightarrow k$ of $\sf J$, for any $\phi\in\Ob\left(\left[W_{k},\sf X_{k}\right]\right)$
we have the following chase.
\[
\vcenter{\vbox{\xyR{2.5pc}\xyC{2.5pc}\xymatrix{\left[W_{k},\sf X_{k}\right]\ar[rr]^{\left[W_{f},\sf X_{k}\right]}\ar[dd]_{\colim^{\left(k,k\right)}} &  & \left[W_{j},\sf X_{k}\right]\ar[dd]\sp(0.3){\colim^{\left(j,k\right)}}\\
 & \phi\ar@{|->}[rr]\ar@{|->}[dd] &  & \phi W_{f}\ar@{|->}[dd]\\
\sf X_{k}\ar[rr]\sp(0.3){\id} &  & \sf X_{k}\\
 & \colim\phi\ar@{|->}[rr] &  & \colim\phi\osi\colim\phi W_{f}
}
}}
\]
Thus, the square
\[
\vcenter{\vbox{\xyR{5pc}\xyC{7pc}\xymatrix{\prod_{i\in\fJ}\left[W_{i},\fX_{i}\right]\ar@/_{.5pc}/[r]_{\prod\left(\left(\left[W_{f},\sf X_{k}\right]\right)\right)}\ar[d]_{\prod\colim^{\left(i,i\right)}} & \prod_{j\in\fJ}\prod_{k\in\fJ}\prod_{f\in\fJ\left(j,k\right)}\left[W_{j},\fX_{k}\right]\ar[d]^{\prod\prod\prod\colim^{\left(j,k\right)}}\\
\prod_{i\in\fJ}\fX_{i}\ar@/_{.5pc}/[r]_{\id} & \prod_{j\in\fJ}\prod_{k\in\fJ}\prod_{f\in\fJ\left(j,k\right)}\fX_{k}
}
}}
\]
pseudo-commutes. Similarly, since $\sf X_{f}$ preserves $W_{j}$
colimits for all $f:j\rightarrow k$ of $\sf J$, we have for each
$\phi\in\Ob\left(\left[W_{j},\sf X_{j}\right]\right)$ the following
chase.
\[
\vcenter{\vbox{\xyR{2.5pc}\xyC{2.5pc}\xymatrix{\left[W_{j},\sf X_{j}\right]\ar[rr]^{\left[W_{j},\sf X_{f}\right]}\ar[dd]_{\colim^{\left(j,j\right)}} &  & \left[W_{j},\sf X_{k}\right]\ar[dd]\sp(0.3){\colim^{\left(j,k\right)}}\\
 & \phi\ar@{|->}[rr]\ar@{|->}[dd] &  & \sf X_{f}\phi\ar@{|->}[dd]\\
\sf X_{j}\ar[rr]\sp(0.3){\sf X_{f}} &  & \sf X_{k}\\
 & \colim\phi\ar@{|->}[rr] &  & \sf X_{f}\left(\colim\phi\right)\osi\colim\left(\sf X_{f}\phi W_{f}\right)
}
}}
\]
Thus, the square 
\[
\vcenter{\vbox{\xyR{5pc}\xyC{7pc}\xymatrix{\prod_{i\in\fJ}\left[W_{i},\fX_{i}\right]\ar@/^{.5pc}/[r]^{\prod\left(\left(\left[W_{k},\sf X_{f}\right]\right)\right)}\ar[d]_{\prod\colim^{\left(i,i\right)}} & \prod_{j\in\fJ}\prod_{k\in\fJ}\prod_{f\in\fJ\left(j,k\right)}\left[W_{j},\fX_{k}\right]\ar[d]^{\prod\prod\prod\colim^{\left(j,k\right)}}\\
\prod_{i\in\fJ}\fX_{i}\ar@/^{.5pc}/[r]^{\prod\left(\left(\sf X_{f}\right)\right)} & \prod_{j\in\fJ}\prod_{k\in\fJ}\prod_{f\in\fJ\left(j,k\right)}\fX_{k}
}
}}
\]
pseudo-commutes. Then, defining functor $\sf{sp}^{\p}$ as the pseudo-limit
$\underleftarrow{\sf{pslim}}\left(\prod\colim^{\left(i,i\right)},\prod\prod\prod\colim^{\left(j,k\right)}\right)$,
we have that $\sf{sp^{\p}}$ is left adjoint to $\Delta^{\p}=\underleftarrow{\sf{pslim}}\left(\prod\Delta^{\left(i,i\right)},\prod\prod\prod\Delta^{\left(j,k\right)}\right)$.
In light of our earlier observations, we have $\sf{sp}=\sf{sp}^{\p}I_{W}$
is left adjoint to $\Delta$ and the theorem is proved.
\end{proof}

\begin{example}
To make clear the necessity of the condition that $\fX$ is injective-canonical
fibrant, consider the following example. Let $\fJ=\bG_{\leq1}$ (the
ur-pair-of-parallel-arrow) and let $W=\fX$ be the $\sfJ$-diagram
\[
\bullet\xymatrix{\ar@/^{.5pc}/[r]^{s}\ar@/_{.5pc}/[r]_{t} & \mathbf{I}}
\]
where $\bullet$ is the terminal category and $\mathbf{I}$ is the
ur-isomorphism. Since $s$ and $t$ are equivalences of categories
and $\bullet$ is co-complete it follows that all hypotheses of the
theorem, save fibrancy, are satisfied. We see however that there does
not exist \emph{any }adjunction between $\left[\fJ,\CAT\right]\left(\bullet,\fX\right)$
and $\left[\fJ,\CAT\right]\left(W,\fX\right)$ as the former is the
empty category whereas the latter is a terminal category.

To see that the fibrancy hypothesis excludes this example we observe
that, since $\fJ$ is both direct and inverse, it follows that the
Reedy-canonical model structure on $\left[\fJ,\CAT\right]$ ( picking
$\fJ^{-}=\fJ$ ) and the injective-canonical model structure coincide.
Thus $\fX$ is fibrant precisely if the matching maps 
\[
\fX\left(\overline{0}\right)\longrightarrow\fX\left(\overline{1}\right)\times\fX\left(\overline{1}\right)
\]
 and 
\[
\fX\left(\overline{1}\right)\longrightarrow\bullet
\]
 are iso-fibrations. While the second of those maps is an iso-fibration,
as all maps into the terminal category are iso-fibrations, the first
map, which evaluates to 
\[
\bullet\xrightarrow{\left(s,t\right)}\mathbf{I}\times\mathbf{I}
\]
is \emph{not }an iso-fibration as $\left(s,t\right)$, as an object
of $\mathbf{I}\times\mathbf{I}$, is isomorphic therein to all three
other objects of $\mathbf{I}\times\mathbf{I}$ and none of those isomorphisms
lift.
\end{example}

\subsection{Examples: Spectrification Three Ways}

We now attend to the promised examples.
\begin{example}
\textbf{\label{exa:Sequential-spectra-}Sequential spectra} - Let
$\sf S$ denote the usual category of spaces and let $\sf S_{\bullet}$
denote the category of pointed spaces. Let $\sf{J=\Z_{\leq0}}$, let
$W=\left(\Z_{\leq0}\downarrow\_\right)^{\op}:\sf J\longrightarrow\Cat$,
and let $\sf X:\Z_{\leq0}\longrightarrow\CAT$ be the functor sending
each non-positive integer to $\sfS_{\bullet}$ and each $-\left(n+1\right)\longrightarrow n$
to $\Omega:\sfS_{\bullet}\longrightarrow\sfS_{\bullet}$. As we will
show, these data satisfy the hypotheses of Theorem \ref{thm: generalized-spectrification-strict-limit}.

First, since $\Z_{\leq0}$ is an inverse category, it follows that
the fibrant objects of the injective-canonical model structure are
but sequences of iso-fibrations - for any inverse category $A$ and
model category $C$, the fibrant objects are diagrams of fibrations
between fibrant objects and all categories are fibrant. Then since
$\Omega$ is an iso-fibration\footnote{there are of course numerous aguments for this, for example $\T_{\bullet}\longrightarrow\S_{\bullet}$
is conservative, and non-identity isomorphisms of sets are just renamings,
so we may rename the elements of $X$ as $Y^{\p}$ in such a way that
$\Omega Y=Y^{\p}$.}, it follows that $\sf X$ is injective-canonical fibrant. Second,
$\sf S_{\bullet}$ is co-complete so it has all $W_{-n}$-colimits.
Third, as $W_{-n}=\left(\Z_{\leq0}\downarrow-n\right)^{\op}\liso\Z_{\geq n}$
and $\Z_{\geq n+1}\longrightarrow\Z_{\geq n}$ is final, it follows
that $W$ is a diagram whose morphisms are all final functors. Lastly,
the functor $\Omega=\Hom\left(S^{1},\_\right)$ preserves sequential
colimits; given any $G:\Z_{\geq0}\longrightarrow\sfS_{\bullet}$,
since any map $S^{1}\longrightarrow\colim G$ factors through some
$G\left(n\right)\longrightarrow\colim G$ it follows that $\sf S_{\bullet}\left(S^{1},\colim G\right)\liso\colim\sf S_{\bullet}\left(S^{1},G\left(\_\right)\right)$.
Theorem \ref{thm: generalized-spectrification-strict-limit} then
provides that there is an adjunction 
\[
\xyC{5pc}\xymatrix{\left[\Z_{\leq n},\CAT\right]\left(W,\sf X\right)\ar@/^{.625pc}/[r] & \left[\Z_{\leq0},\CAT\right]\left(\bullet,\sf X\right)\ar@/^{.625pc}/[l]_{\bot}^{\t}}
\]
What we will now show is that this adjunction recovers $\Omega$-spectra
as a reflective subcategory of sequential spectra.

A sequential spectrum object in $\sf S_{\bullet}$ is comprised of:
\begin{itemize}
\item an $\N$ indexed family of pointed spaces $\left(X_{-n}\right)_{n\in\N}$;
together with
\item morphisms $\left(\phi_{-n}:X_{-n}\longrightarrow\Omega X_{-\left(n+1\right)}\right)_{n\in\N}$
where $\Omega=\sfS_{\bullet}\left(S^{1},\_\right)$, the usual loop-space
functor.
\end{itemize}
The data of these objects is readily seen equivalent to that of oplax
cones over $\sf X:\Z_{\leq0}\longrightarrow\CAT$, so the usual category
of sequential spectra is here recovered as $\oplaxlim\sf X$. As we
will show, the weight $W$ is in fact the oplax-weight for $\Z_{\leq0}$,
meaning that those oplax cones are equivalent to natural transformations
\[
\Xi:W\Longrightarrow\sf X:\Z_{\leq0}\longrightarrow\CAT
\]

Observe that the component of some natural transformation $\Xi$ 
\[
\Xi:W\Longrightarrow\sf X:\Z_{\leq0}\longrightarrow\CAT
\]
 at the object $-n$ of $\Z_{\leq0}$, is but a diagram
\[
\cdots\longleftarrow X_{-\left(n+2\right)}^{-n}\longleftarrow X_{-\left(n+1\right)}^{-n}\longleftarrow X_{-n}^{-n}
\]
 in $\sfS_{\bullet}$. Consider then that for each $f:-m\rightarrow-n$
of $\Z_{\leq0}$, the naturality criterion
\[
\xyR{3pc}\xyC{3pc}\xymatrix{\left(\Z_{\leq0}\downarrow-m\right)^{\op}\ar[d]\ar[r] & \sfS_{\bullet}\ar[d] & \left(-k\rightarrow-m\right)\ar@{|->}[d]\ar@{|->}[r] & X_{-k}^{-m}\ar@{|->}[d]\\
\left(\Z_{\leq0}\downarrow-n\right)^{\op}\ar[r] & \sfS_{\bullet} & \left(-k-m-n\right)\ar@{|->}[r] & X_{-k}^{-n}=\Omega^{n-m}\left(X_{-k}^{-n}\right)
}
\]
imposes the identities 
\[
\left\{ X_{-k}^{-n}=\Omega^{n-m}\left(X_{-k}^{-n}\right)\right\} _{-k\rightarrow-m\rightarrow-n\in\Z_{\leq0}^{\bfTwo}}
\]
 so our natural transformations are equivalent to oplax cones as follows:
the data of a natural transformation $\left(\Z_{\leq0}\downarrow\_\right)^{\op}\Longrightarrow\sf X$
is given in the solid arrows and equalities in the diagram below,
and the data of an oplax cone is given by the objects along the diagonal
together with the dashed arrows.
\[
\xyR{.5pc}\xyC{3pc}\xymatrix{\ddots\\
\cdots & \Omega^{2}\left(X_{-2}^{-2}\right)\ar[l]\ar@{=}[d]\ar@{-->}[ul]\\
\cdots & \Omega\left(X_{-2}^{-1}\right)\ar@{=}[d]\ar[l] & \Omega\left(X_{-1}^{-1}\right)\ar@{=}[d]\ar[l]\ar@{-->}[ul]\\
\cdots & X_{-2}^{0}\ar[l] & \ar[l]X_{-1}^{0} & X_{0}^{0}\ar[l]\ar@{-->}[ul]
}
\]

Now, taking a very strict definition\footnote{our reason here for picking the strict definition is that this works
aims to synthesize stabilization and strict notions of algebra. It
seems likely however that the treatment of the notion should hold
up well in an $\left(\infty,2\right)$-categorical treatment - a topic
of future work.}, the category of sequential $\Omega$-spectra is the sub-category
of sequential spectra on spectrum objects for which the constituent
families of structure maps
\[
\left(\phi_{-n}:X_{-n}\longrightarrow\Omega X_{-\left(n+1\right)}\right)_{n\in\N}
\]
 are families of homeomorphisms. In terms of our functor $\sf X:\Z_{\leq0}\longrightarrow\CAT$
this is but $\underleftarrow{\mathsf{pslim}}\sf X$ inside of $\oplaxlim\sf X$.
But $\sf X$ is injective-canonical fibrant, so $\llim\sf X\liso\underleftarrow{\sf{pslim}}\sf X$.
Now, it's not hard to see that the diagram
\[
\xyR{1pc}\xyC{3pc}\xymatrix{\oplaxlim\sf X\ar[d]_{\wr} & \underleftarrow{\sf{pslim}}\sf X\ar[l] & \llim\sf X\ar[l]_{\sim}\ar[d]_{\wr}\\
\left[\Z_{\leq0},\CAT\right]\left(W,\sf X\right) &  & \ar[ll]^{\Delta}\left[\Z_{\leq0},\CAT\right]
}
\]
commutes, whence the left adjoint to $\Delta$ is the left adjoint
to the inclusion of sequential $\Omega$-spectra into sequential spectra,
which is to say that this left adjoint is spectrification.
\end{example}

Lastly, this $2$-categorical treatment of spectrification also applies
to the more sophisticated notion of coordinate-free spectra found
in \cite{ElmendorfKrizMandellMay}.
\begin{example}
\textbf{Coordinate-free spectra} - keeping $\sfS$ and $\sf S_{\bullet}$
as above and fix some $\mathscr{U}\liso\R^{\infty}$ (as vector spaces).
Let $\sf{FinSub}\left(\mathscr{U}\right)$ be the category of finite
dimensional sub-vector-spaces of $\mathscr{U}$. For each $V\in\sf{FinSub}\left(\mathscr{U}\right)$,
denote by $S^{V}$ the one-point compactification of $V$ and for
each $V\subset W\in\sf{FinSub}\left(\mathscr{U}\right)$, let $W-V$
denote the orthogonal complement of $V$ in $W$. A ($\mathscr{U}$
)-coordinate free spectrum object $X$ of $\sf{S_{\bullet}}$ is comprised
of:
\begin{itemize}
\item for each $V\in\sf{FinSub}\left(U\right)$, a pointed space $X_{V}$;
and
\item for each $V\longrightarrow W$ of $\sf{FinSub}\left(U\right)$, a
morphism $\sigma_{V,W}:X_{V}\longrightarrow\Omega^{W-V}\left(X_{W}\right)$
where $\Omega^{W-V}=\sf S_{\bullet}\left(S^{W-V},\_\right).$ 
\end{itemize}
Denote the category of such by $\sf{CFSp}\left(\mathscr{U},\sf S_{\bullet},\Omega^{W-V}\right)$

A coordinate-free spectrum is, as in the previous example, equivalent
to an oplax cone over a $1$-diagram in $\Cat$ - here the diagram
is as follows.
\[
\xyR{0pc}\xyC{5pc}\xymatrix{\boldsymbol{\Omega}_{\mathsf{CFS}}:\sf{FinSub}\left(\mathscr{U}\right)^{\op}\ar[r] & \Cat\\
V\ar@{|->}[r] & \sf S_{\bullet}\\
V\subset W\ar@{|->}[r] & \Omega^{W-V}
}
\]
In the case of coordinate-free spectra, the meaning of $\Omega$-spectra
is essentially the same as with sequential spectra: a coordinate free
$\Omega$-spectrum has all is constituent morphisms $\sigma_{V,W}:X_{V}\longrightarrow\Omega^{W-V}\left(X_{W}\right)$
being homeomorphisms. We see again that
\[
\xyR{1.5pc}\xyC{5pc}\xymatrix{\Omega\mhyphen\sf{CFSp}\left(\sf U,\sf S_{\bullet},S^{\left(\_\right)}\right)\ar@{^{(}->}[r]\ar[d]^{\wr} & \sf{CFSp}\left(\sf U,\sf S_{\bullet},S^{\left(\_\right)}\right)\ar[d]^{\wr}\\
\left[\bullet,\fX\right]_{\left[\Z_{\leq0},\Cat\right]}\ar@{^{(}->}[r]^{\t} & \left[\left(\Z_{\leq0}\downarrow\_\right)^{\op},\fX\right]_{\left[\Z_{\leq0},\Cat\right]}
}
\]
commutes with the so labeled functors being equivalences of categories.
As before, by Theorem \ref{thm: generalized-spectrification-strict-limit},
\[
\xyR{1.5pc}\xyC{5pc}\xymatrix{\Omega\mhyphen\sf{CFSp}\left(\sf U,\sf S_{\bullet},S^{\left(\_\right)}\right)\ar@{^{(}->}[r] & \sf{CFSp}\left(\sf U,\sf S_{\bullet},S^{\left(\_\right)}\right)}
\]
will have a left adjoint as:
\begin{itemize}
\item $\boldsymbol{\Omega}_{\mathsf{CFS}}$ is injective-canonical fibrant
in $\left[\sf{FinSub}\left(\mathscr{U}\right)^{\op},\CAT\right]$;
\begin{itemize}
\item Observe that picking $\sf{FinSub}\left(\mathscr{U}\right)^{\op}=\left(\sf{FinSub}\left(\mathscr{U}\right)^{\op}\right)^{-}$
defines a Reedy structure on $\sf{FinSub}\left(\mathscr{U}\right)^{\op}$
such that the injective-canonical and Reedy-canonical model structures
on $\left[\sf{FinSub}\left(\mathscr{U}\right)^{\op},\CAT\right]$
coincide. Thus to check fibrancy of $\boldsymbol{\Omega}_{\sf{CFS}}$
is to check that, for every $W\in\sf{FinSub}\left(\mathscr{U}\right)^{\op}$,
the map 
\[
\sf S_{\bullet}\longrightarrow\underset{W\hookleftarrow V:W\downarrow\sf{FinSub}\left(\mathscr{U}\right)^{\op}}{\llim}\sf S_{\bullet}
\]
is an iso-fibration. But since every square in $\sf{FinSub}\left(\mathscr{U}\right)^{\op}$
commutes, that limit is just the product 
\[
\prod_{V:\sf{codim}\mhyphen1\left(W\right)}\sf{S_{\bullet}}
\]
over all co-dimension $1$ subspaces of $W$ and the map is the one
induced by that universal property by the family of maps $\left(\Omega\right)_{V:\sf{codim}\mhyphen1\left(W\right)}$.
Indeed this is an iso-fibration as $\Omega$ is an iso-fibration.
\end{itemize}
\item the category $\sfS_{\bullet}$ admits all small colimits;
\item the functors $\Omega^{W-V}$ preserves filtered colimits, whence $\left(\mathsf{FinSub}\left(\mathscr{U}\right)^{\op}\downarrow U\right)^{\op}$-colimits;
and
\item for any $V\subset U$ functor 
\[
\xyR{1.5pc}\xyC{5pc}\xymatrix{\left(\mathsf{FinSub}\left(\mathscr{U}\right)^{\op}\downarrow U\right)^{\op}\ar[r] & \left(\mathsf{FinSub}\left(\mathscr{U}\right)^{\op}\downarrow V\right)^{\op}}
\]
 is final.
\end{itemize}
\end{example}

\section{Oplax limits and The Oplax weight $\left(\protect\sfJ\downarrow\_\right)^{\protect\op}:\protect\sfJ\protect\longrightarrow\protect\Cat$}

The material of this section is well known to the $2$-categorical
cognoscenti\footnote{read: ``not the author until the writing of this article''}.
We include it here for completeness, for notation, and for ease of
reference.

\subsection{Oplax natural transformations and the universal property of the oplax
limit}
\begin{defn}
Given $2$-categories $\sfJ$ and $\sfD$, and two $2$-functors $X,Y:\sfJ\longrightarrow\sfD$,
an oplax-natural transformation $\boldsymbol{\alpha}:X\overset{\mathsf{oplax}}{\Longrightarrow}Y$
is comprised of:
\begin{itemize}
\item a set of $1$-cells $\left\{ a_{j}:X_{j}\longrightarrow Y_{j}\right\} _{j\in\Ob\left(\sfJ\right)}$;
and
\item a $\Mor\left(\sfJ\right)$ indexed family of $2$-cells 
\[
\left\{ \vcenter{\vbox{\xyR{3pc}\xyC{3pc}\xymatrix{X_{i}\ar[r]^{X_{f}}\ar[d]_{a_{i}} & X_{j}\ar[d]^{a_{j}}\ar@{=>}[dl]|-{\alpha_{f}}\\
Y_{i}\ar[r]_{Y_{f}} & Y_{j}
}
}}\right\} _{f:i\rightarrow j\in\Mor\left(\sfJ\right)};
\]
 
\end{itemize}
such that, for every $i\overset{g}{\longrightarrow}j\overset{f}{\longrightarrow}k$
of $\sfJ$, we have an equality of $2$-cells 
\[
\vcenter{\vbox{\xyR{5pc}\xyC{5pc}\xymatrix{X_{i}\ar[r]^{X_{g}}\ar[d]_{a_{i}} & X_{j}\ar[d]|-{a_{j}}\ar@{=>}[dl]|-{\alpha_{g}}\ar[r]^{X_{f}} & X_{k}\ar[d]^{a_{k}}\ar@{=>}[dl]|-{\alpha_{f}}\\
Y_{i}\ar[r]_{Y_{g}} & Y_{j}\ar[r]_{Y_{f}} & Y_{k}
}
}}=\vcenter{\vbox{\xyR{5pc}\xyC{5pc}\xymatrix{X_{i}\ar[r]^{X_{g}}\ar[d]_{a_{i}} & X_{j}\ar[d]^{a_{j}}\ar@{=>}[dl]|-{\left(Y_{f}\underset{0}{\circ}\alpha_{f}\right)\underset{1}{\circ}\left(\alpha_{g}\underset{0}{\circ}X_{g}\right)}\\
Y_{i}\ar[r]_{Y_{g}} & Y_{j}
}
}}
\]
We define $\text{\ensuremath{\mathsf{Oplax}\left(X,Y\right)}}$ to
be the $1$-category with:
\begin{itemize}
\item oplax-natural transformations $\boldsymbol{\alpha}:X\overset{\mathsf{oplax}}{\Longrightarrow}Y$
as $0$-cells; and
\item modifications $\mathbf{c}:\boldsymbol{\alpha}\Rrightarrow\boldsymbol{\beta}$
as $1$-cells.
\end{itemize}
\end{defn}

Having recalled the notion of oplax-natural transformation, we may
recall the notion of the oplax limit.
\begin{defn}
For $2$-categories $\sfJ$ and $\sfD$ and a $2$-functor $X:\sfJ\longrightarrow\sfD$,
then the oplax\textbf{ limit }of the $2$-diagram $X$, denoted\footnote{in the usual abuse of notation}
$\underleftarrow{\sf{oplaxlim}}X$ , is comprised of:
\begin{itemize}
\item an object $\oplaxlim X$ of $\sfD$;
\end{itemize}
and an oplax cone into $X$, i.e.
\begin{itemize}
\item projection maps $\left\{ \widetilde{\pr_{i}}:\oplaxlim X\longrightarrow X_{i}\right\} _{i\in\Ob\left(\sf J\right)}$;
and
\item $2$-cells $\left\{ \pi_{f}:\widetilde{\pr_{i}}\Longrightarrow f\circ\widetilde{\pr_{j}}\right\} _{f:i\rightarrow j\in\Mor\left(\sf J\right)}$
\end{itemize}
with:
\begin{itemize}
\item for all $i\overset{f}{\longrightarrow}j\overset{g}{\longrightarrow}k$
of $\fJ$, the diagram of $2$-cells commuting, i.e. $\left(g\pi_{f}\right)\circ\pi_{g}=\pi_{g\circ f}$
\end{itemize}
enjoying the universal property:
\end{defn}

\begin{itemize}
\item post-composition with the oplax-cone $\left(\left(\widetilde{\pr_{i}}\right),\left(\pi_{f}\right)\right)$
induces an equivalence of categories.
\[
\text{\ensuremath{\sf{\sfD}}}\left(A,\oplaxlim X\right)\liso\mathsf{Oplax}\left(\sf J,\sfC\right)\left(\t A,X\right)
\]
 
\end{itemize}
\begin{rem}
It's worth noting that the suggestive (of the structure of a module)
criterion $\left(g\pi_{f}\right)\circ\pi_{g}=\pi_{g\circ f}$ in the
definition of an oplax cone can be born out (see \cite{Lack}).
\end{rem}

\subsection{\label{sec:Op-lax-limits-as-weighted-limits} The Oplax weight $\left(\protect\sfJ\downarrow\_\right)^{\protect\op}:\protect\sfJ\protect\longrightarrow\protect\Cat$}
\begin{defn}
Given a $2$-category $\sfJ$, call the $2$-functor
\[
\left(\sfJ\downarrow\_\right)^{\op}:\sfJ\longrightarrow\Cat
\]
the \textbf{oplax weight}.
\end{defn}

As one would hope, given the name, the oplax weight $\left(\sfJ\downarrow\_\right)^{\op}$
computes oplax limits. See that, for a given $2$-functor $X:\sfJ\longrightarrow\Cat$,
we've an oplax cone 
\[
\t\left[\sfJ,\Cat\right]\left(\left(\sfJ\downarrow\_\right)^{\op},X\right)\overset{\mathsf{oplax}}{\Longrightarrow}X
\]
comprised of $1$-cells 
\[
\left\{ \vcenter{\vbox{\xyR{0pc}\xyC{5pc}\xymatrix{\t\left[\sfJ,\Cat\right]\left(\left(\sfJ\downarrow\_\right)^{\op},X\right)\ar[r] & X_{j}\\
\left(\alpha:\left(\sfJ\downarrow\_\right)^{\op}\Rightarrow X\right)\ar@{|->}[r] & \alpha_{j}\left(\id_{j}:j\rightarrow j\right)\\
\left(\mathbf{c}:\alpha\Rrightarrow\beta\right)\ar@{|->}[r] & \mathbf{c}_{j,\id_{j}}:\alpha_{j}\left(\id_{j}\right)\rightarrow\beta_{j}\left(\id_{j}\right)
}
}}\right\} _{j\in\Ob\left(\sfJ\right)}
\]
and $2$-cells 
\[
\vcenter{\vbox{\xyR{3pc}\xyC{3pc}\xymatrix{\t\left[\sfJ,\Cat\right]\left(\left(\sfJ\downarrow\_\right)^{\op},X\right)\ar@{=}[r]\ar[d] & \t\left[\sfJ,\Cat\right]\left(\left(\sfJ\downarrow\_\right)^{\op},X\right)\ar[d]\ar@{=>}[dl]\\
X_{j}\ar[r] & X_{k}
}
}}
\]
whose component at some $\alpha:\left(\sfJ\downarrow\_\right)^{\op}\Rightarrow X$
is given by $1$-cells
\[
\xyR{1.5pc}\xyC{1.5pc}\xymatrix{\alpha:\left(\sfJ\downarrow\_\right)^{\op}\Rightarrow X\ar@{|..>}[dd]\ar@{|..>}[rr] &  & \alpha:\left(\sfJ\downarrow\_\right)^{\op}\Rightarrow X\ar@{|..>}[d]\\
 &  & \alpha_{k}\left(\id_{k}\right)\ar[dl]|-{\alpha_{k}\left(f\rightarrow k\right)}\\
\alpha_{j}\left(\id_{j}\right)\ar@{|..>}[r] & X_{f}\circ\alpha_{j}\left(\id_{j}\right)=\alpha_{k}\left(j\rightarrow k\right)
}
\]
where by $\alpha_{k}\left(f\rightarrow k\right)$ we mean the application
of $\alpha_{k}$ to the morphism of $\left(\sfJ\downarrow k\right)^{\op}$
corresponding to the following triangle.
\[
\xyR{1.5pc}\xyC{1.5pc}\xymatrix{j\ar[rr]^{f}\ar[dr]_{f} &  & k\ar@{=}[dl]\\
 & k
}
\]
 The oplax cone thus described endows $\left[\sfJ,\Cat\right]\left(\left(\sfJ\downarrow\_\right)^{\op},X\right)$
with the universal property of the oplax limit over $X$.

What remains a bit opaque however is how, precisely, we translate
between oplax-natural transformations 
\[
\text{\ensuremath{\boldsymbol{\alpha}:X\overset{\mathsf{oplax}}{\Longrightarrow}Y}}
\]
 and functors
\[
\oplaxlim\boldsymbol{\alpha}:\left[\sfJ,\Cat\right]\left(\left(\sfJ\downarrow\_\right)^{\op},X\right)\longrightarrow\left[\sfJ,\Cat\right]\left(\left(\sfJ\downarrow\_\right)^{\op},Y\right)
\]
Every $\beta:\left(\sfJ\downarrow\_\right)^{\op}\Longrightarrow X$
is comprised of a natural family of functors\footnote{The idiosyncrasy of using super-script here will make something a
little more natural in our examples section, so those irked by the
choice should rest assured that the author has not chosen this notation
just to be a pest - and is instead a pest with a mission.} 
\[
\left\{ \beta^{j}:\left(\sfJ\downarrow j\right)^{\op}\longrightarrow X_{j}\right\} _{j\in\Ob\left(\sfJ\right)}
\]
We will describe $\oplaxlim\boldsymbol{\alpha}$ by describing its
action on the objects of $\left[\sfJ,\Cat\right]\left(\left(\sfJ\downarrow\_\right)^{\op},X\right)$,
the natural transformations $\beta:\left(\sfJ\downarrow\_\right)^{\op}\Longrightarrow X$,
component-wise. We set $\oplaxlim\boldsymbol{\alpha}\left(\beta\right)$
to the natural transformation comprised of the family of functors
$\left\{ \widetilde{\boldsymbol{\alpha}}\left(\beta^{j}\right)\right\} _{j\in\Ob\left(\sfJ\right)}$ 
\begin{itemize}
\item whose actions on objects are given
\[
\vcenter{\vbox{\xyR{0pc}\xyC{3pc}\xymatrix{\widetilde{\boldsymbol{\alpha}}\left(\beta^{j}\right):\left(\sfJ\downarrow j\right)^{\op}\ar[r] & Y_{j}\\
\left(\overset{q}{p\rightarrow j}\right)\ar@{|->}[r] & Y_{q}\circ\alpha_{p}\circ\beta^{p}\left(\id_{p}\right)
}
}}
\]
and
\item which sends morphisms $\left(k\rightarrow j\right):\left(p\overset{q}{\rightarrow}j\right)\rightarrow\left(\ell\overset{q\circ k}{\longrightarrow}j\right)$
corresponding to triangles 
\[
\xyR{1.5pc}\xyC{1.5pc}\xymatrix{\ell\ar[rr]^{k}\ar[dr]_{q\circ k} &  & p\ar[dl]^{q}\\
 & j
}
\]
to the compositions 
\[
\xyR{1.5pc}\xyC{5pc}\xymatrix{Y_{q}\alpha_{p}\beta^{p}\left(\id_{p}\right)\ar[r]^{Y_{q}\alpha_{p}\beta^{p}\left(q\rightarrow p\right)} & Y_{q}\alpha_{p}\beta^{p}\left(\ell\xrightarrow{k}p\right)\ar@{=}[d]\\
 & Y_{q}\alpha_{p}X_{k}\beta^{\ell}\left(\id_{\ell}\right)\ar[r]^{Y_{q}\alpha_{k,\beta^{\ell}\left(\id_{\ell}\right)}\left(\beta^{\ell}\left(\id_{\ell}\right)\right)} & Y_{q}Y_{k}\alpha_{\ell}\beta^{\ell}\left(\id_{\ell}\right)\ar@{=}[d]\\
 &  & Y_{qk}\alpha_{\ell}\beta^{\ell}\left(\id_{\ell}\right)
}
\]
\end{itemize}
\begin{rem}
Those carefully type-checking the composition above may find the following
$2$-diagram useful. 
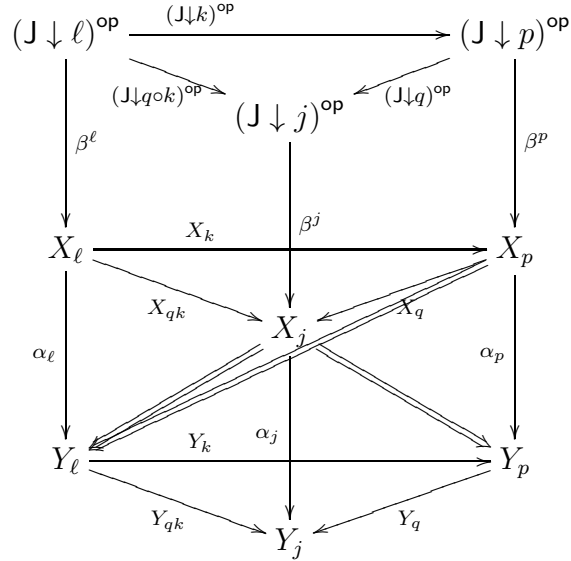
\[
\xyR{1pc}\xyC{3pc}\xymatrix{\left(\sfJ\downarrow\ell\right)^{\op}\ar[rr]\sp(0.3){\left(\sfJ\downarrow k\right)^{\op}}\ar[dr]_{\left(\sfJ\downarrow q\circ k\right)^{\op}}\ar[ddd]^{\beta^{\ell}} &  & \left(\sfJ\downarrow p\right)^{\op}\ar[dl]^{\left(\sfJ\downarrow q\right)^{\op}}\ar[ddd]^{\beta^{p}}\\
 & \left(\sfJ\downarrow j\right)^{\op}\ar[ddd]^{\beta^{j}}\\
\\
X_{\ell}\ar[rr]\sp(0.3){X_{k}}\ar[dr]_{X_{qk}}\ar[ddd]_{\alpha_{\ell}} &  & X_{p}\ar[dl]^{X_{q}}\ar[ddd]_{\alpha_{p}}\ar@{=>}[dddll]\\
 & X_{j}\ar[ddd]_{\alpha_{j}}\ar@{=>}[ddl]\ar@{=>}[ddr]\\
\\
Y_{\ell}\ar[rr]\sp(0.3){Y_{k}}\ar[dr]_{Y_{qk}} &  & Y_{p}\ar[dl]^{Y_{q}}\\
 & Y_{j}
}
\]
\end{rem}

\section{Extensions of Monads}
\begin{lem}
\label{lem:Extension-of-monads}Given categories $\sfA$ and $\sfB$,
with $\sfB$ co-complete, and a functor $N:\sfA\longrightarrow\sfB$,
then the composition 
\[
\Lan_{N}\left(N\circ\_\right):\left[\sfA,\sfA\right]\xrightarrow{\left[\sfA,N\right]}\left[\sfA,\sfB\right]\xrightarrow{\Lan_{N}\left(\_\right)}\left[\sfB,\sfB\right]
\]
is monoidal. If moreover:
\begin{itemize}
\item $N$ is full-and-faithful and dense; and
\item $\Lan_{N}\left(N\circ\_\right):\left[\sfA,\sfA\right]\longrightarrow\left[\sfB,\sfB\right]_{\cc}\hookrightarrow\left[\sfB,\sfB\right]$;
\end{itemize}
then $\text{\ensuremath{\Lan}}_{N}\left(N\circ\_\right)$ is strong
monoidal.
\end{lem}

\begin{proof}
The first part is well known.\footnote{\begin{proof}
although this author with his woeful preparation hadn't thought about
it much before this work was compiled in the context of fear, paranoia,
and incipient unemployment.
\end{proof}
} The second part is a weakening of a Math-overflow post by di-Liberti
(https://mathoverflow.net/questions/321909/extending-monads-along-dense-functors).

Since $N$ is dense if and only if $\Lan_{N}\left(N\right)=\id_{\sfA}$,
thus for dense $N$, $\Lan_{N}\left(N\circ\_\right)$ strictly preserves
the unit. Since:
\begin{itemize}
\item $N$ is full-and-faithful, $N\circ F=\Lan_{N}\left(N\circ F\right)\circ N$;
and
\item since $\Lan_{N}\left(N\circ F\right)$ is co-continuous for all $F:\sfA\rightarrow\sfA$,
$\Lan_{N}\left(N\circ F\right)$ preserves left Kan extensions, so
for any $F\circ G$ in $\left[\sfA,\sfA\right]$ we have that $\Lan_{N}\left(N\circ F\right)\circ\Lan_{N}\left(N\circ G\right)=\Lan_{N}\left(\Lan_{N}\left(N\circ F\right)\circ N\circ G\right)$;
\end{itemize}
thus 
\begin{eqnarray*}
\Lan_{N}\left(N\circ F\circ G\right) & = & \Lan_{N}\left(\Lan_{N}\left(N\circ F\right)\circ N\circ G\right)\\
 & = & \Lan_{N}\left(N\circ F\right)\circ\Lan_{N}\left(N\circ G\right).
\end{eqnarray*}
\end{proof}

\section{On Limits in Comma Categories}
\begin{lem}
\label{lem: Commas and the existence of limits therein}Let $\left(\sfV,\otimes,\bfOne\right)$
be a symmetric monoidal category and let $\sfA,\sfB,$ and $\sfC$
be $\sfV$-categories powered in $\sfV$. Then, given a comma category
\[
\xyR{1.5pc}\xyC{1.5pc}\xymatrix{ & F\downarrow G\ar[dl]_{\pr_{1}}\ar[dr]^{\pr_{2}}\\
\sfA\ar@{=>}[rr]\ar[dr]_{L} &  & \sfB\ar[dl]^{R}\\
 & \sfC
}
\]
where $G$ is continuous and $\sfC$ is complete\footnote{this can surely be weakened, but to do so in the statement of the
lemma would make it more onerous to read than this author is willing
to write.} it follows that, if for some diagram 
\[
\left(a_{\bullet},f_{\bullet}:L\left(a_{\bullet}\right)\longrightarrow R\left(b_{\bullet}\right),b_{\bullet}\right):\mathsf{I}\longrightarrow L\downarrow R
\]
 and some weight 
\[
W:\mathsf{I}\longrightarrow\sfV
\]
 the weighted limits $\llim^{W}a_{\bullet}$ and $\llim^{W}b_{\bullet}$
exist in $\sfA$ and $\sfB$ respectively, then 
\[
\left(\llim^{W}a_{\bullet},\left(\widetilde{f}_{\bullet}\right),\llim^{W}b_{\bullet}\right)
\]
 enjoys the universal property of the weighted limit $\llim^{W}\left(a_{\bullet},f_{\bullet},b_{\bullet}\right)$
- with $\left(\widetilde{f}_{\bullet}\right)$ induced from the universal
property\footnote{in a fashion made more explicit in the proof} enjoyed
by $\llim^{W}R\left(b_{\bullet}\right)=R\left(\llim^{W}b_{\bullet}\right)$
in $\sfC$.
\end{lem}

\begin{proof}
We've assumed $\sfA,\sfB$, and $\sfC$ to be $\sfV$-powered. It
therefore suffices to consider conical limits.

Now, the cone 
\[
L\left(\alpha\right):L\left(\llim a_{\bullet}\right)\Longrightarrow L\left(a_{\bullet}\right):\mathsf{I}\longrightarrow\sfC
\]
and the natural transformation 
\[
f_{\bullet}:L\left(a_{\bullet}\right)\Longrightarrow R\left(b_{\bullet}\right)
\]
compose to cone $\widetilde{f}_{\bullet}:=f_{\bullet}\circ\alpha:L\left(\llim a_{\bullet}\right)\Longrightarrow R\left(b_{\bullet}\right)$.
This cone induces, by the universal property of $\llim R\left(b_{\bullet}\right)$
enjoyed by $R\left(\llim b_{\bullet}\right)$ as $R$ is continuous,
a morphism in $\sfC$
\[
\left(\widetilde{f_{\bullet}}\right):L\left(\llim a_{\bullet}\right)\longrightarrow R\left(\llim b_{\bullet}\right)
\]
whence $\left(\llim a_{\bullet},\left(\widetilde{f_{\bullet}}\right),\llim b_{\bullet}\right)$
is an object of $L\downarrow R$. That $\left(\llim a_{\bullet},\left(\widetilde{f_{\bullet}}\right),\llim b_{\bullet}\right)$
enjoys the universal property of the limit in $L\downarrow R$ follows
immediately from the that of $\llim a_{\bullet}$ and $\llim b_{\bullet}$
in $\sfA$ and $\sfB$ respectively.
\end{proof}
\begin{cor}
\label{cor:creativity-from-continuity} Let $\left(\sfV,\otimes,\bfOne\right)$
be a symmetric monoidal category and let $\sfB,$ and $\sfC$ be $\sfV$-categories
powered in $\sfV$, with $\sfC$ complete. Then, given a comma category
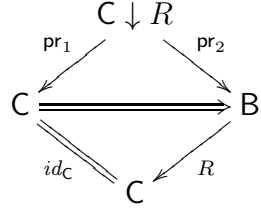
\[
\xyR{1.5pc}\xyC{1.5pc}\xymatrix{ & \sfC\downarrow R\ar[dl]_{\pr_{1}}\ar[dr]^{\pr_{2}}\\
\sfC\ar@{=>}[rr]\ar@{=}[dr]_{\id_{\sfC}} &  & \sfB\ar[dl]^{R}\\
 & \sfC
}
\]
where $R:\sfB\longrightarrow\sfC$ preserves all weighted limits,
then $\sfC\downarrow R\longrightarrow\sfC\times\sfB$ creates all
weighted limits.
\end{cor}

\begin{proof}
Again, by powering, it suffices to consider conical limits, and since
$R$ preserves all weighted limits, Lemma \ref{lem: Commas and the existence of limits therein}
applies and we have a formula for conical limits in $\sfC\downarrow R$.
\end{proof}
\begin{elabeling}{00.00.0000}
\item [{Preservation}] from the formula we have that, when $\llim b_{\bullet}$
exists in $\sfB$,
\[
\llim\left(c_{\bullet},f_{\bullet}:c_{\bullet}\rightarrow R\left(b_{\bullet}\right),b_{\bullet}\right)=\left(\llim c_{\bullet},\left(\widetilde{f_{\bullet}}\right),\llim b_{\bullet}\right)
\]
It is immediate then that $\sfC\downarrow R\longrightarrow\sfC\times\sfB$
preserves limits.
\item [{Reflection}] since limits and colimits in products of categories
are computed factor-wise, if a cone
\[
\left(c,f:c\rightarrow R\left(b\right),b\right)\Longrightarrow\left(c_{\bullet},f_{\bullet}:c_{\bullet}\rightarrow R\left(b_{\bullet}\right),b_{\bullet}\right)
\]
 in $\sfC\downarrow R$ passes to a limiting cone 
\[
\left(c,b\right)\Longrightarrow\left(c_{\bullet},b_{\bullet}\right)
\]
in $\sfC\times\sfB$, then $c\Longrightarrow c_{\bullet}$ and $b\Longrightarrow b_{\bullet}$
are limiting. But then, since $R$ preserves limits, it follows from
the formula that $\left(c,f:c\rightarrow R\left(b\right),b\right)\Longrightarrow\left(c_{\bullet},f_{\bullet}:c_{\bullet}\rightarrow R\left(b_{\bullet}\right),b_{\bullet}\right)$
is limiting.
\end{elabeling}
\begin{cor}
\label{cor:creativity-from-creativity}Let $\left(\sfV,\otimes,\bfOne\right)$
be a symmetric monoidal category and let $\sfB,$ and $\sfC$ be $\sfV$-categories
powered in $\sfV$, with $\sfC$ complete. Then, given a comma category
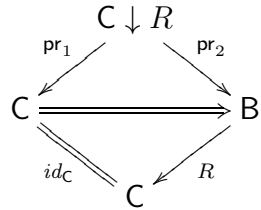
\[
\xyR{1.5pc}\xyC{1.5pc}\xymatrix{ & \sfC\downarrow R\ar[dl]_{\pr_{1}}\ar[dr]^{\pr_{2}}\\
\sfC\ar@{=>}[rr]\ar@{=}[dr]_{\id_{\sfC}} &  & \sfB\ar[dl]^{R}\\
 & \sfC
}
\]
where $R:\sfB\longrightarrow\sfC$ creates all weighted limits, the
functor $\sfC\downarrow R\longrightarrow\sfC^{\bfTwo}$ creates all
weighted limits.
\end{cor}

\begin{proof}
See that the square 
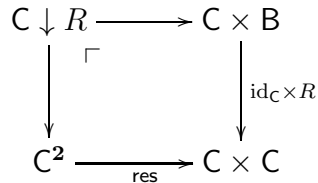
\[
\xyR{3pc}\xyC{3pc}\xymatrix{\sfC\downarrow R\ar[r]\ar[d]\pullbackcorner & \sfC\times\sfB\ar[d]^{\id_{\sfC}\times R}\\
\sfC^{\bfTwo}\ar[r]_{\mathsf{res}} & \sfC\times\sfC
}
\]
is a pullback. Since $\sfC\downarrow R\longrightarrow\sfC\times\sfB$
creates all limits by Corollary \ref{cor:creativity-from-continuity},
and $\mathsf{res}$ creates all limits, to prove $\sfC\downarrow R\longrightarrow\sfC^{\bfTwo}$
creates all limits, it suffices that $\id_{\sfC}\times R$ creates
all limits, which it does as it creates the limits that $R$ does
- all of them.
\end{proof}

\end{document}